\documentclass[a4paper,10pt]{amsart}
\usepackage{enumitem}   
\usepackage{cite}
\usepackage{colonequals}
\usepackage{tikz-cd}
\usetikzlibrary{er,positioning}
\usepackage{amsmath,amssymb,amsthm,graphicx,yfonts,mathrsfs,color}
\usepackage{amscd}  
\usepackage{stmaryrd}
\usepackage[all]{xy} 
\usepackage{multirow} 
\usepackage{array}
\usepackage[utf8]{inputenc}
\usepackage{braket}
\usepackage{rotating}
\usepackage{setspace}
\usepackage{multirow}
\usepackage{rotating}
\usepackage{pifont}
\usepackage{comment}
\PassOptionsToPackage{hyphens}{url}
\usepackage[colorlinks, backref, breaklinks]{hyperref}
\hypersetup{%
  colorlinks = true,
  linkcolor  = black
}

\newcommand*{\isoarrow}[1]{\arrow[#1,"\rotatebox{90}{\(\sim\)}"
]}

\usepackage{thmtools}
\usepackage{thm-restate}
\theoremstyle{plain}
\newtheorem{thm}{Theorem}[section]
\newtheorem{lemma}[thm]{Lemma}

\newtheorem{prop}[thm]{Proposition}
\newtheorem{cor}[thm]{Corollary}
\newtheorem{cjc}[thm]{Conjecture}
\newtheorem{claim}{Claim}
\newtheorem{claim*}{Claim}

\theoremstyle{definition}

\newtheorem{rmk}[thm]{Remark}
\newtheorem{quest}[thm]{Question}
\newtheorem{example}[thm]{Example}

\newtheorem{mydef}[thm]{Definition}

\makeatletter
\def\th@plain{%
  \thm@notefont{}
  \itshape 
}
\def\th@definition{%
  \thm@notefont{}
  \normalfont 
}
\makeatother
 
\newcommand{\ord}{\mathop{\mathrm{ord}_p}\nolimits}
\newcommand{\ordnop}{\mathop{\mathrm{ord}}\nolimits}
\newcommand{\Log}{\mathop{\mathrm{Log}}\nolimits}

\providecommand{\tr}{\mathop{\rm tr}\nolimits}

\newcommand{\Q}{\mathbb{Q}}

\newcommand{\Z}{\mathbb{Z}}
\newcommand{\F}{\mathbb{F}}
\newcommand{\OO}{\mathcal{O}} 
\newcommand{\fq}{\mathfrak{q}}
\newcommand{\fp}{\mathfrak{p}}

\DeclareMathOperator{\tors}{tors}

\DeclareFontFamily{U}{wncy}{}
\DeclareFontShape{U}{wncy}{m}{n}{<->wncyr10}{}
\DeclareSymbolFont{mcy}{U}{wncy}{m}{n}
\DeclareMathSymbol{\sh}{\mathord}{mcy}{"58} 

\author{Francesca Bianchi}

\address{Bernoulli Institute for Mathematics, 
Computer Science and Artificial Intelligence\\ University of Groningen, Groningen, The Netherlands}
\email{francesca.bianchi@rug.nl}

\title{Quadratic Chabauty for (bi)elliptic curves and Kim's conjecture}
\date{\today}
\begin{document}

\begin{abstract}
We explore a number of problems related to the quadratic Chabauty method for determining integral points on hyperbolic curves.
We remove the assumption of semistability in the description of the quadratic Chabauty sets $\mathcal{X}(\mathbb{Z}_p)_2$ containing the integral points $\mathcal{X}(\mathbb{Z})$ of an elliptic curve of rank at most $1$. Motivated by a conjecture of Kim, we then investigate theoretically and computationally the set-theoretic difference $\mathcal{X}(\mathbb{Z}_p)_2\setminus \mathcal{X}(\mathbb{Z})$. We also consider some algorithmic questions arising from Balakrishnan--Dogra's explicit quadratic Chabauty for the rational points of a genus-two bielliptic curve. As an example, we provide a new solution to a problem of Diophantus which was first solved by Wetherell.\\
Computationally, the main difference from the previous approach to quadratic Chabauty is the use of the $p$-adic sigma function in place of a double Coleman integral.
\end{abstract}


\maketitle
\section{Introduction}
\label{intro}
Let $(E,O)$ be an elliptic curve over $\mathbb{Q}$ and fix an odd prime $p$ of good reduction. Denote by $\mathcal{E}$ the minimal regular model of $E$ and by $\mathcal{X}$ the complement of the origin in $\mathcal{E}$.

When $E$ has Mordell--Weil rank $1$ and the Tamagawa number of $E/\mathbb{Q}$ is trivial at all primes, Kim \cite{KimMasseyProducts} and Balakrishnan--Kedlaya--Kim \cite{AppendixToMasseyProduct} described an explicit locally analytic function on $\mathcal{X}(\mathbb{Z}_p)$ which vanishes on the set $\mathcal{X}(\mathbb{Z})$ of global integral points.

Subsequently, Balakrishnan--Dan-Cohen--Kim--Wewers\cite{nonabelianconjecture} generalised the result to arbitrary semistable elliptic curves of rank $1$ and gave a similar $p$-adic characterisation of $\mathcal{X}(\mathbb{Z})$ when $E$ is semistable and has rank $0$.

The discussion fits into Kim's non-abelian Chabauty programme as introduced in \cite{KimP1} and \cite{Kimunipotent}. In particular, Kim constructed a sequence of subsets of $p$-adic points
\begin{equation*}
\mathcal{X}(\mathbb{Z}_p)\supset\mathcal{X}(\mathbb{Z}_p)_1\supset\mathcal{X}(\mathbb{Z}_p)_2\supset\dots \supset \mathcal{X}(\mathbb{Z}).
\end{equation*}
The $p$-adic locally analytic functions from \cite{nonabelianconjecture} are essentially those that define $\mathcal{X}(\mathbb{Z}_p)_2$, the set of \emph{cohomologically global points of level $2$}, in the larger $\mathcal{X}(\mathbb{Z}_p)$.

The subscript $n$ in $\mathcal{X}(\Z_p)_n$ indicates a particular quotient $U_n$ of the unipotent $p$-adic \'etale fundamental group $U$ of $\mathcal{X}_{\overline{\mathbb{Q}}}$ (at a tangential base point). The set $\mathcal{X}(\mathbb{Z}_p)_n$ is then defined in terms of certain ``unipotent Kummer maps'' from $\mathcal{X}(\Z)$ and $\mathcal{X}(\mathbb{Z}_q)$, at every prime $q$, to global and local cohomology sets with $U_n$-coefficients, respectively, in a way that generalises to objects with non-abelian \'etale fundamental group the role played by $\Q_p$-Selmer groups in our understanding of rational points on abelian varieties.

Despite its abstract cohomological definition, the set $\mathcal{X}(\Z_p)_n$ is believed to be computable in practice (\hspace{1sp}{\cite{nonabelianconjecture}}) as a union of intersections of zero loci of locally analytic functions defined in terms of iterated $p$-adic integrals. Unfortunately, such a characterisation is yet to be provided for $n\geq 3$.

Nevertheless, the explicit description of $\mathcal{X}(\mathbb{Z}_p)_2$ in the rank $0$ semistable case given in \cite{nonabelianconjecture} was already sufficient to collect some computational evidence for the following special case of a conjecture of Kim (see \cite[\S3.1]{nonabelianconjecture}).

\begin{cjc}[Kim, 2012]
\label{cjc:Kim}
For sufficiently large $n$, we have
\begin{equation*}
\mathcal{X}(\mathbb{Z}_p)_n=\mathcal{X}(\mathbb{Z}).
\end{equation*}
\end{cjc}

Indeed, the authors of \emph{loc.\ cit.}\ verified the equality $$\mathcal{X}(\mathbb{Z}_p)_2=\mathcal{X}(\mathbb{Z})$$ for the prime $p=5$ and for all the $256$ semistable elliptic curves of rank $0$ for which they computed $\mathcal{X}(\mathbb{Z}_p)_2$. An additional test that was performed by the same authors was that of fixing $\mathcal{X}$ and varying the prime $p$: once again, no point in $\mathcal{X}(\Z_p)_2\setminus\mathcal{X}(\Z)$ was found. No other study of the difference $\mathcal{X}(\Z_p)_2\setminus \mathcal{X}(\Z)$ appears in the literature, hence motivating the following two questions:
\begin{quest}
\label{quest:1}
Does there exist any elliptic curve of rank $0$ for which $\mathcal{X}(\mathbb{Z}_p)_2$ contains at least one point which is not in $\mathcal{X}(\Z)$?
\end{quest}

\begin{quest}
\label{quest:2}
What geometric or algebraic properties should a point in $\mathcal{X}(\mathbb{Z}_p)_2\setminus \mathcal{X}(\Z)$ satisfy?
\end{quest}

One goal of the present paper is to give answers to these questions, with the idea that elliptic curves should serve as a test case for a conjecture that is in fact formulated by Kim in much greater generality than how we stated it here, and that as such would have striking applications if it were to hold. Indeed, $\mathcal{X}$ could be replaced by a suitable $\Z$-model $\mathcal{C}$ of any hyperbolic curve over $\Q$ with good reduction at $p$. In particular, the conjecture would give an effective approach towards finding the set of rational points on a curve of genus $g\geq 2$.

 In the elliptic curve case, the conjecture might not have direct Diophantine interest, in the sense that there already exist algorithms for the computation of integral points on elliptic curves \cite{smart, Pethoetal, stroeker}, and the rank $0$ and $1$ instances which we will explore are particularly well understood. However, the known explicit versions of non-abelian Chabauty for curves of higher genus (cf.\ \cite{BBM0,BDQCI,SplitCartan}) all generalise the explicit description of $\mathcal{X}(\Z_p)_2$ for elliptic curves of rank at most $1$. Therefore, a conceptual understanding of the zero sets of the $p$-adic equations defining $\mathcal{X}(\Z_p)_2$ is essential to hope to achieve something similar in more complicated settings.

In general, even finiteness of $\mathcal{C}(\Z_p)_n$ for $n$ large enough is only conjectural (but see \cite{Kim:CMelliptic, CoatesKim, EllenbergHast, BDQCI} for results in this direction, and \cite{Kimunipotent} for a proof assuming the Bloch--Kato conjecture). If $g\geq 2$ and, for a given $n$, $\mathcal{C}(\Z_p)_n$ is finite and explicitly computable to arbitrary $p$-adic precision, then the Mordell--Weil sieve could be used to try to provably extrapolate $\mathcal{C}(\Z)$ from $\mathcal{C}(\Z_p)_n$. However, the Mordell--Weil sieve is not guaranteed to terminate.  
Thus, finiteness of $\mathcal{C}(\Z_p)_n$ would not be sufficient to imply an effective version of Faltings's Theorem. 

Suppose now that $\mathcal{C}(\Z_p)_n$ is finite for $n$ sufficiently large. One reason for expecting that the inclusion $\mathcal{C}(\mathbb{Z})\subset\mathcal{C}(\mathbb{Z}_p)_n$ should eventually become sharp is explained in \cite[\S 1.8]{nonabelianconjecture}: assuming some well known motivic conjectures, the number of algebraically independent locally analytic functions vanishing on $\mathcal{C}(\Z_p)_n$ is strictly increasing in $n$ (for $n\gg 0$). See also \cite[\S\S 1, 3.4]{nonabelianconjecture} for the philosophy behind Conjecture \ref{cjc:Kim} (in its general form) and for its relationship with the conjectural finiteness of the Tate--Shafarevich group and with Grothendieck's section conjecture. 

For an elliptic curve of rank $1$, finiteness of $\mathcal{X}(\Z_p)_n$ can only hold at level $n\geq 2$.  On the other hand, for a rank $0$ elliptic curve, $\mathcal{X}(\Z_p)_1$ is finite and there are two independent equations defining $\mathcal{X}(\Z_p)_2$ (see Theorem \ref{level2rank0} below), hence justifying why Question \ref{quest:1} had proved itself arduous.  We show, however, that the answer to the question is negative. More precisely, we prove the following two theorems (see also Theorem \ref{thm:variationp}).
\begin{thm}
\label{thm:nonequality}
There exist infinitely many rank $0$ elliptic curves for which $$\mathcal{X}(\mathbb{Z})\subsetneq \mathcal{X}(\mathbb{Z}_p)_2$$ for infinitely many good primes $p$. 
\end{thm}

\begin{thm}
\label{thm:nonequalitystrong}
There exists exactly one rank $0$ elliptic curve of conductor at most $30,000$ for which $$\mathcal{X}(\mathbb{Z})\subsetneq \mathcal{X}(\mathbb{Z}_p)_2$$ for all primes $p$ of good (ordinary and supersingular) reduction. This is the curve \href{http://www.lmfdb.org/EllipticCurve/Q/8712/u/5}{8712.u5} \cite{lmfdb}. 
\end{thm}
When we analyse these results in conjunction with Question \ref{quest:2}, it will become apparent that they should not be considered as negative evidence for Conjecture \ref{cjc:Kim}. 
 We will return below to discussing Theorems \ref{thm:nonequality}, \ref{thm:nonequalitystrong} and answers to Question \ref{quest:2} in the context of the methods we develop in order to prove them. For the reader's convenience, let us first digress to write down the equations defining $\mathcal{X}(\Z_p)_2$. 
In fact, the very first goal of this article is to extend the explicit description of $\mathcal{X}(\Z_p)_2$ to an arbitrary elliptic curve of rank $0$ and at the same time correct a slight imprecision in the analogous statement in the semistable case \cite[Theorem 1.12]{nonabelianconjecture} (see Remark \ref{rmk:corrections}).

Before stating the theorem, we introduce some additional notation, which is convenient to maintain similar to \cite{nonabelianconjecture}. Let $\mathcal{E}$ be described by
\begin{equation}
\label{weierstrasseq}
y^2+a_1xy+a_3y = x^3+a_2x^2+a_4x+a_6
\end{equation}
and let $S$ be the set of primes at which $E$ has bad reduction.

For each $q\in S$, define the set $W_q\subset\mathbb{Q}_p$ as follows. If the Tamagawa number at $q$ is $1$, let $$W_q=\{0\};$$ in all other cases, 
\begin{equation*}
W_q=\begin{cases}
W_{q}^{\mathrm{bad}} & \text{if}\ q=2\ \text{and}\ E\ \text{is split multiplicative at}\ q,\\
W_{q}^{\mathrm{bad}} \cup\{0\} & \text{otherwise},
\end{cases}
\end{equation*}
where $W_{q}^{\mathrm{bad}}$ is the finite subset of $\mathbb{Q}_p$ described in \textsc{Table} \ref{table:1} (with $F=\mathbb{Q}$ and $v=(q_v)=(q)$); in particular, $W_{q}^{\mathrm{bad}}$ only depends on the reduction type of $E$ at $q$. 
Let
\begin{equation*}
W = \prod_{q\in S}W_q
\end{equation*} 
and, if $w\in W$, write $||w||=\sum_{q\in S} w_q$. Let $b$ be the integral tangent vector at the origin which is dual to $\omega(O)$, where $$\omega = \frac{dx}{2y+a_1x+a_3}$$
and let $\eta = x\omega$. Furthermore, for $z\in\mathcal{X}(\mathbb{Z}_p)$ write 
\begin{equation*}
\Log(z) = \int_{b}^{z}\omega \quad \text{and} \quad D_2(z) = \int_{b}^z\omega\eta,
\end{equation*}
where the integrals are Coleman integrals.
\begin{thm}
\label{level2rank0}
Suppose that $E$ has rank $0$ and that the $p$-primary part of the Tate--Shafarevich group is finite.
\begin{enumerate}
\item\label{trivial} If, for at least one of $q\in\{2,3\}$, the reduction of $E$ at $q$ is good and $\overline{E}(\mathbb{F}_q)=\{O\}$, or if $E$ has split multiplicative reduction of Kodaira type $\mathrm{I}_1$ at $2$, then 
$$\mathcal{X}(\mathbb{Z}_p)_2=\mathcal{X}(\mathbb{Z})=\emptyset.$$
\item \label{nontrivial}Otherwise, 
\begin{equation*}
\mathcal{X}(\mathbb{Z}_p)_2 = \bigcup_{w\in W} \phi(w),
\end{equation*}
where
\begin{equation*}
\phi(w) = \{z \in \mathcal{X}(\mathbb{Z}_p) : \Log(z)  =0, 2D_2(z) + ||w||=0 \}.
\end{equation*}
\end{enumerate}
\end{thm}

We also remove the assumption of semistable reduction in the rank $1$ case \cite[Proposition 5.12]{nonabelianconjecture}. Assume $E$ has good ordinary\footnote{Although not explicitly stated in \cite{nonabelianconjecture}, their statement also holds only when $p$ is ordinary. However, a similar result holds in the supersingular case: see Remark \ref{rmk:generalisedheights}.} reduction at $p$. Let ${\bf E}_2$ be the Katz $p$-adic weight $2$ Eisenstein series \cite{Katzinterpolation} and let
\begin{equation}
\label{constantE2}
C = \frac{a_1^2+4a_2-{\bf E}_2(E,\omega)}{12}.
\end{equation}
Let $h_p\colon E(\mathbb{Q})\to \mathbb{Q}_p$ be $(-2p)$ times the $p$-adic height of \cite{MST} and define
\begin{equation*}
c=\frac{h_p(z_0)}{\Log(z_0)^2},
\end{equation*}
for a non-torsion point $z_0\in E(\mathbb{Q})$. 

\begin{thm}
\label{level2rank1}
Suppose that $E$ has rank $1$ and that $p$ is a prime of good ordinary reduction. 
\begin{enumerate}
\item\label{trivialrank1}  If, for at least one of $q\in\{2,3\}$, the reduction of $E$ at $q$ is good and $\overline{E}(\mathbb{F}_q)=\{O\}$, or if $E$ has split multiplicative reduction of Kodaira type $\mathrm{I}_1$ at $2$, then
$$\mathcal{X}(\mathbb{Z}_p)_2 = \mathcal{X}(\mathbb{Z})=\emptyset.$$
\item\label{nontrivialrank1} Otherwise, 
\begin{equation*}
\mathcal{X}(\mathbb{Z})\subset \mathcal{X}(\mathbb{Z}_p)_2^{\prime}\colonequals\bigcup_{w\in W}\psi(w),
\end{equation*}
where
\begin{equation*}
\psi(w) = \{z\in\mathcal{X}(\mathbb{Z}_p) : 2D_2(z)+C(\Log(z))^2+||w|| = c(\Log(z))^2\}.
\end{equation*}
\end{enumerate}
\end{thm}
According to \cite{nonabelianconjecture}, the set $\cup_{w\in W}\psi(w)$ should equal $\mathcal{X}(\mathbb{Z}_p)_2$: hence the notation $\mathcal{X}(\mathbb{Z}_p)_2^{\prime}$. Section \ref{sec:heights_and_main} is devoted to the proofs of Theorems \ref{level2rank0} and \ref{level2rank1}.

The equations defining the sets of $p$-adic points of the two theorems can be given an elementary interpretation as linear relations amongst $\mathbb{Q}_p$-valued quadratic functions on $E(\mathbb{Q})$, dictated by the assumptions on the rank. More precisely, any global $p$-adic height $E(\Q)\to \Q_p$ (of Bernardi, Coleman--Gross, Mazur--Tate) vanishes identically if the rank is $0$, and is a scalar multiple of $\Log^2\vert_{E(\Q)}$ if the rank is $1$. To go from here to a $p$-adic approximation of the global integral points, one invokes the decomposition of the $p$-adic height on $E(\Q)\setminus\{O\}$ as a sum, over the non-archimedean primes $q$, of local $p$-adic heights $\lambda_q\colon E(\Q_q)\setminus\{O\}\to \Q_p$. Indeed, the restriction of $\lambda_q$ to $\mathcal{X}(\Z_q)\supset\mathcal{X}(\Z)$ has finite image for all $q\neq p$, zero image for almost all $q\neq p$, and is locally analytic for $q=p$.

This point of view is crucial in our investigation, in Section \ref{sec:obstructions}, of what points could arise in $\mathcal{X}(\mathbb{Z}_p)_2\setminus \mathcal{X}(\mathbb{Z})$ in rank $0$. Recall that no example of an elliptic curve of rank $0$ and a prime $p$ for which $\mathcal{X}(\Z_p)_2\supsetneq \mathcal{X}(\Z)$ was previously known. A careful study of the Mazur--Tate and Bernardi local $p$-adic heights allows us to deduce possible obstructions to the sharpness of $\mathcal{X}(\Z_p)_2$, and to give necessary and sufficient conditions for a point in $\mathcal{X}(\Z_p)\setminus \mathcal{X}(\Z)$ to belong to $\mathcal{X}(\Z_p)_2$.

First note that a point in $\mathcal{X}(\Z_p)_2$ is algebraic, since it is in the zero set of the abelian logarithm $\Log$. Our sufficient conditions then come from studying how automorphisms of $E/\overline{\Q}$ affect the values of the local $p$-adic heights at certain algebraic points, and from an analysis of non-cyclotomic $p$-adic heights over non-totally real number fields. A combination of these two phenomena explains the appearance of extra points at level $2$ in the family of quadratic twists of the modular curve $X_0(49)$ (see Proposition \ref{prop:twists_X049}). As an application, in \S \ref{sec:proof_of_nonequality} we prove Theorem \ref{thm:nonequality}.

As regards necessary conditions for $\mathcal{X}(\Z_p)_2$ to contain parasite points, we prove a sort of ``$p$-adic height saturation'' condition (see the discussion in \S \ref{sec:quad_sat}): 
\begin{thm}[Theorem \ref{thm:extrapoints}]
\label{thm:extrapointsintro}
Let $E/\Q$ and $p$ be as in Theorem \ref{level2rank0}. Suppose that $z\in \mathcal{X}(\mathbb{Z}_p)_2\setminus \mathcal{X}(\mathbb{Z})$. Then $z$ is the localisation of a torsion point $P$ over a number field $K$ and, for each rational prime $q$, the value $\lambda_{\mathfrak{q}}(P)$ of the local height at $\fq$ of the cyclotomic $p$-adic height of $E/K$ is independent of the prime $\mathfrak{q}\mid q$ of $K$.
\end{thm}
 
We then present in Section \ref{sec:Computations} the computations of the sets $\mathcal{X}(\mathbb{Z}_p)_2$ for all the elliptic curves over $\mathbb{Q}$ of rank $0$ and conductor less than or equal to $30,000$ and for some choices of $p$. We propose a slightly different but equivalent way of computing the set $\mathcal{X}(\mathbb{Z}_p)_2$, compared to the one used in \cite{nonabelianconjecture}. In particular, our method does not rely on general algorithms to compute double Coleman integrals, but rather uses Bernardi's and Mazur--Tate's description of the $p$-dic height on an elliptic curve to express the double Coleman integrals in terms of $p$-adic sigma functions.

Our computations (run on \texttt{SageMath} \cite{sage}) suggest that the failure of sharpness of $\mathcal{X}(\Z_p)_2$ is still to be considered a rare phenomenon, which we were always able to explain using the sufficient conditions of Section \ref{sec:obstructions}. Extra points become even more exceptional if we allow the prime $p$ to vary. In particular, we prove Theorem \ref{thm:nonequalitystrong} (see Theorem \ref{thm:variationp}).

In future work, it would be interesting to verify whether Conjecture \ref{cjc:Kim} holds at level $3$ for the curves and primes for which we found $\#\mathcal{X}(\mathbb{Z}_p)_2>\#\mathcal{X}(\mathbb{Z})$.

When $E$ has rank $1$, the set described in Theorem \ref{level2rank1}\thinspace (\ref{nontrivialrank1}) is generally larger than $\mathcal{X}(\mathbb{Z})$. Naively, this is because $\mathcal{X}(\Z_p)_2^{\prime}$ is cut out by the vanishing of one function only. In \S \ref{sec:algebraicpointsrank1}, we ask what algebraic points can belong to $\mathcal{X}(\mathbb{Z}_p)_2^{\prime}\setminus \mathcal{X}(\mathbb{Z})$. In \S \ref{sec:rank1computations} we compute $\mathcal{X}(\mathbb{Z}_p)_2^{\prime}$ for all the $14,783$ rank $1$ elliptic curves of conductor at most $5,000$; for each curve, we let $p$ be the smallest prime greater than or equal to $5$ at which the curve has good ordinary reduction.

Finally, in Section \ref{sec:bielliptic} we apply some of our techniques for elliptic curves to the computation of rational points on certain genus $2$ curves $C$ over $\Q$. 
Indeed, when $C$ admits degree $2$ maps to two elliptic curves over $\Q$, each of rank $1$, Balakrishnan--Dogra \cite{BDQCI} described a $\Q_p$-valued locally analytic function on $C(\Q_p)$ vanishing on $C(\Q)$. 
This is defined using local $p$-adic heights on each elliptic curve. 
We explain how one can replace direct computations of double Coleman integrals with computations involving the $p$-adic sigma function and division polynomials also in this situation. 
We make the computation explicit for a curve arising from a problem from the \emph{Arithmetica} of Diophantus and use it to give an alternative proof to the one given by Wetherell in his thesis \cite{wetherell} of the fact that the curve has exactly $8$ rational points. 

Balakrishnan--Dogra implemented their method numerically in some examples. 
However, algorithmically, their approach was still based on a case-by-case study. This consideration applies especially to a preliminary step which consists in the computation of two finite subsets of $\Q_p$ (which play the role of the set $||W||$ above). By combining results of \cite{BDQCI} with properties of local $p$-adic heights on elliptic curves, in \S \ref{sec:explicitWi} we offer a more general and applicable numerical approach to the method. For example, we give an algorithm that takes as input a bielliptic curve $C$ whose associated elliptic curves have Mordell--Weil rank $1$ together with a good prime $p$ and outputs a finite set of $p$-adic points containing $C(\mathbb{Q})$ (i.e.\ we remove the preliminary computation step). In doing so, we also provide a more elementary proof and approach to the explicit result of \cite{BDQCI}.

The code used for the computations in this article is available at \cite{codeQC_FB}.

\subsection*{Acknowledgements}
I am very grateful to Jennifer Balakrishnan for many useful discussions and for providing me with helpful feedback on several drafts.
This paper started from questions raised at group meetings organised by Minhyong Kim in which I took part together with Jamie Beacom, Noam Kantor and Alex Saad. I am indebted to all of them and thank Minhyong for encouraging me to write up.
I also thank the anonymous referee, as well as Steffen M\"uller, and my DPhil examiners Victor Flynn and Chris Wuthrich, for several useful suggestions, which led to many improvements to the article.
The author was supported by EPSRC and by Balliol College through a Balliol Dervorguilla scholarship, and by an NWO Vidi grant.

\section{Description of $\mathcal{X}(\mathbb{Z}_p)_2$}
\label{sec:heights_and_main}
\subsection{The $p$-adic height and its local components}
\label{padicheight}
Let $p$ be an odd prime and extend the usual $p$-adic logarithm $\log\colon 1+p\mathbb{Z}_p\to \mathbb{Q}_p$ to $\mathbb{Q}_p^{\times}$ via $\log(p)=0$.

Let $E$ be an elliptic curve over $\mathbb{Q}$ as in Section \ref{intro} and assume $E$ has good reduction at $p$. We will sometimes need to consider the base-change of $E$ to a number field $F$ (whose ring of integers is denoted $\OO_F$); thus, we do not restrict the following definitions to $E(\mathbb{Q})$. Let
\begin{equation*}
h_p\colon E(F)\to \mathbb{Q}_p
\end{equation*}
be a cyclotomic\footnote{\label{footnote:anticyclotomic}If the space of continuous idele class characters $\mathbb{A}_F^{\times}/F^{\times}\to \mathbb{Q}_p$ has dimension larger than $1$, we will see in \S \ref{sec:other_characters} that one can define other types of $p$-adic heights.} $p$-adic height of Coleman--Gross (see \cite{ColemanGross} and \cite{ColemanGrosspAdicSigma}). The use of the indefinite article here is due to the dependence of $h_p$ on a choice that will be made explicit in \S \ref{sec:heightsabovep}. The function $h_p$ is quadratic, i.e.\ satisfies the relation
\begin{equation*}
h_p(mP) = m^2 h_p(P) \ \text{for all}\ m\in\mathbb{Z}\ \text{and}\ P\in E(F),
\end{equation*}
and is defined as a sum of local heights, one for each non-archimedean prime of $F$. In particular, we have\footnote{Here we choose to normalise the $p$-adic height in such a way that it becomes independent of the choice of the field $F$ containing the coordinates of $P$; note that this is not the case in many other articles, such as \cite{MST}.} $h_p(O) = 0$ and for $P\in E(F)\setminus \{O\}$
\begin{equation*}
h_p(P) = \frac{1}{[F : \mathbb{Q}]}\sum_{v}n_v\lambda_v(P),
\end{equation*}
where the sum is over all finite primes $v$ of $F$, $n_v=[F_v:\mathbb{Q}_v]$ and 
\begin{equation*}
\lambda_v\colon E(F_{v})\setminus \{O\}\to \mathbb{Q}_p
\end{equation*}
is a \emph{$p$-adic local N\'eron function at $v$}.

Let $q_v$ be the norm of $v$ and $|\cdot|_v$ be the normalised absolute value corresponding to $v$. That is, if $x\in F_v^{\times}$, we have
\begin{equation*}
|x|_v = q_v^{-\text{ord}_v(x)/n_v},
\end{equation*}
where the valuation $\ordnop_v$ is such that $\ordnop_v(F_v^{\times})=\mathbb{Z}$.

\subsubsection{}\label{sec:heightsawayfromp}
 The $p$-adic local N\'eron function at a non-archimedean prime $v\nmid p$ is equal to the real local N\'eron function $\widehat{\lambda}_v$ at $v$ with the $p$-adic logarithm in place of the real one. Thus, any reference that we provide for $\widehat{\lambda}_v$ can be applied also to our setting. For instance, analogously to the real case, for $v\nmid p$, the following properties determine a unique function $\lambda_v\colon E(F_v)\setminus \{O\}\to \Q_p$:  
\begin{enumerate}[label=(\roman*)]
\item\label{property:continuous_and_bounded} $\lambda_v$ is continuous on $E(F_v)\setminus \{O\}$ and bounded on the complement of any neighbourhood of $O$ with respect to the $v$-adic topology.
\item\label{property:limit} $\lim_{P\to O}(\lambda_{v}(P)-\log|x(P)|_v)$ exists.
\item\label{property:quasipar} $\lambda_v$ satisfies the \emph{quasi-parallelogram law}: for all $P,Q\in E(F_v)$ such that $P,Q,P\pm Q\neq O$, we have
\begin{equation}
\label{eq:quasi_parallelogram_law}
\lambda_v(P+Q)+\lambda_v(P-Q) = 2\lambda_v(P)+2\lambda_v(Q)-2\log|x(P)-x(Q)|_v.
\end{equation}
\end{enumerate}
Uniqueness follows from topological reasons. For existence, it suffices to show that the $p$-adic analogue of $\widehat{\lambda}_v$ obtained as described above satisfies \ref{property:continuous_and_bounded}{}-\ref{property:quasipar} (see \cite[VI, Exercise 6.3]{silvermanadvancedtopics}).

We also have
\begin{enumerate}[label = (\roman*)]
\setcounter{enumi}{3}
\item\label{property:quasiquadratic} For all $P\in E(F_v)$ and all $m\geq 1$ with $mP\neq O$,
\begin{equation*}
\lambda_v(mP)=m^2\lambda_v(P) -2\log|f_m(P)|_v,
\end{equation*}
where $f_m$ is the $m$-th division polynomial of $E$ (see for instance \cite[III, Exercise 3.7]{silvermanAEC} for the definition of $f_m$).
We say that $\lambda_v$ is \emph{quasi-quadratic}.
\end{enumerate}
Moreover, uniqueness implies the following key fact:
\begin{enumerate}[label = (\roman*)]
\setcounter{enumi}{4}
\item\label{property:invarianceautoawayp}  If $\psi$ is an automorphism of $E$ defined over $F_v$, then, for all $P\in E(F_v)$, $\lambda_v(\psi(P))=\lambda_v(P)$.
\end{enumerate}
See also \cite{bernardi} for a more general transformation property under isogeny.

We wish to determine which values $\lambda_v$ can attain on $\mathcal{X}(\OO_v)$, where $\OO_v$ is the ring of integers of $F_v$ and $v\nmid p$.
For this, it will be convenient to assume that $\mathcal{E}$ is minimal at $v$. If that is not the case, we can always switch to a minimal equation at $v$ and use the following (\hspace{1sp}\cite[Lemma 4]{heightdifference}):
\begin{equation}
\label{lambdaqlambdaqmin}
\lambda_{v} = \lambda_{v}^{\min{}}+\frac{1}{6}\log|\Delta/\Delta^{\min}|_{v},
\end{equation}
where $\Delta$ denotes the discriminant and the superscript $\min$ has the obvious meaning. See also Remark \ref{remarkonminimalmodels}.

So assume for the rest of this subsection that $\lambda_v$ is computed with respect to a minimal model at the prime $v$.
 Denote by $\overline{E}_{ns}(\mathbb{F}_v)$ the group of non-singular points of the reduction of $E$ modulo $v$ and let
\begin{equation*}
E_0(F_v) = \{P\in E(F_v): \overline{P}\in \overline{E}_{ns}(\mathbb{F}_v)\}.
\end{equation*}
If $E$ has good reduction at $v$, we may also write $\overline{E}(\mathbb{F}_v)$ for $\overline{E}_{ns}(\mathbb{F}_v)$.

\begin{lemma}
\label{lemmalocalheights}
Suppose $v\nmid p$. Then
\begin{enumerate}[label=(\roman*)]
\item\label{goodreduction} If $P\in E_0(F_v)\setminus\{O\}$, $\lambda_v(P)=\log(\max\{1, |x(P)|_v\})$.
\item\label{nontrivialcosets} If $P\not\in E_0(F_v)$, $\lambda_v(P)$ depends exclusively on the image of $P$ in $E(F_v)/E_0(F_v)$.
\end{enumerate}
\end{lemma}
\begin{proof}
See \cite{SilvermanComputingHeights} or \cite[Proposition 5]{heightdifference}.
\end{proof}

\begin{prop}
\label{propositionheightstamagawa1}
If $v\nmid p$ and $[E(F_v):E_0(F_v)]= 1$, then
\begin{equation*}
\lambda_v(\mathcal{X}(\OO_v))=\begin{cases}
\{0\} & \text{if}\ \#\overline{E}_{ns}(\mathbb{F}_v)>1\\
\emptyset & \text{otherwise}.
\end{cases}
\end{equation*}
\end{prop}

\begin{proof}
By Lemma \ref{lemmalocalheights}\thinspace\ref{goodreduction}, $\lambda_v(\mathcal{X}(\OO_v))\subseteq \{0\}$, with equality if and only if $\mathcal{X}(\OO_v)\neq \emptyset$. Let 
\begin{equation*}
E_1(F_v)=\{P\in E(F_v): \overline{P}=\overline{O}\};
\end{equation*}
in particular, $\mathcal{X}(\OO_v)\cap E_1(F_v)$ is empty and every point in $E_0(F_v)\setminus E_1(F_v)$ comes from a point in $\mathcal{X}(\OO_v)$. According to \cite[VII, Proposition 2.1]{silvermanAEC}, the sequence
\begin{equation*}
0\to E_1(F_v)\to E_0(F_v)\to \overline{E}_{ns}(\mathbb{F}_v)\to 0
\end{equation*}
is exact, which proves the proposition.
\end{proof}

We now give an elementary necessary condition for $\#\overline{E}_{ns}(\mathbb{F}_v)=1$. We show it is also a sufficient condition in all cases except when $v$ is of good reduction.
\begin{lemma}
\label{conditionsonq}
The group $\overline{E}_{ns}(\mathbb{F}_v)$ has cardinality at least $2$ in all of the following cases:
\begin{enumerate}
\item $E$ has additive or non-split multiplicative reduction at $v$;
\item $E$ has good reduction at $v$ and $q_v>4$;
\item $E$ has split multiplicative reduction at $v$ and $q_v>2$.
\end{enumerate}
Conversely, if $E$ has split multiplicative reduction at $v$ and $q_v=2$, then $$\#\overline{E}_{ns}(\mathbb{F}_v)=1.$$
\end{lemma}
\begin{proof}
If $E$ has additive reduction at $v$, then $\overline{E}_{ns}(\mathbb{F}_v)\cong \mathbb{F}_v^+$ always contains at least two elements. 
If the reduction is non-split multiplicative, then
\begin{equation*}
\overline{E}_{ns}(\mathbb{F}_v)\cong \{a\in k^{\times}:N_{k/\mathbb{F}_v}(a)=1\},
\end{equation*}
where $[k:\mathbb{F}_v]=2$ and $N_{k/\mathbb{F}_v}$ is the field norm of $k/\mathbb{F}_v$. Thus, if $q_v>2$, then the statement is clear; if $q_v=2$, then $k$ is the splitting field of $x^4-x$ over $\mathbb{F}_2$ and each element in $k^{\times}$ has norm $1$ over $\mathbb{F}_2$.

When $E$ has good reduction at $v$ and $q_v>4$, the Hasse bound yields $\#\overline{E}(\mathbb{F}_v)>1$. Finally, if the reduction is split multiplicative, then $\overline{E}_{ns}(\mathbb{F}_v)\cong \mathbb{F}_v^{\times}$.
\end{proof}

\begin{prop}
\label{localheightsnontrivialtamagawa}
If $v\nmid p$ and $[E(F_v):E_0(F_v)]\neq 1$, then
\begin{equation*}
n_v\lambda_v(\mathcal{X}(\OO_v))=\begin{cases}
W_v^{\mathrm{bad}} & \text{if}\ q_v=2\ \text{and}\ E\ \text{is split multiplicative at}\ v,\\
W_v^{\mathrm{bad}} \cup\{0\} & \text{otherwise},
\end{cases}
\end{equation*}
where $W_v^{\mathrm{bad}}$ is defined in \textsc{Table} 1.
\end{prop}
\begin{table}
\begin{center}
\begin{tabular}{ |c|c|c| } 
\hline
Kodaira symbol & $[E(F_v):E_0(F_v)]$ & $W_v^{\mathrm{bad}}$\\
\hline
\multirow{2}{5em}{$\mathrm{I}_n (n\geq 2)$} & $2$ (non-split) & $\{-\frac{n}{4}\log q_v\}$ \\ 
& $n$ (split) & $\{-\frac{i(n-i)}{n}\log q_v: 1\leq i\leq \left \lfloor{\frac{n}{2}}\right \rfloor \}$ \\ 
\hline
\multirow{1}{1em}{$\mathrm{III}$} & $2$  & $\{-\frac{1}{2}\log q_v\}$ \\ 
\hline
\multirow{1}{1em}{$\mathrm{IV}$} & $3$  & $\{-\frac{2}{3}\log q_v\}$ \\
\hline
\multirow{1}{1em}{$\mathrm{I}_0^*$} & $2$ or $4$  & $\{-\log q_v\}$\\
\hline
\multirow{2}{5em}{$\mathrm{I}_n^*(n\geq 1)$} & $2$  & $\{-\log q_v\}$\\
& $4$ & $\{-\log q_v, -\frac{n+4}{4}\log q_v \}$\\
\hline
\multirow{1}{1em}{$\mathrm{IV}^*$} & $3$ & $\{-\frac{4}{3}\log q_v\}$\\
\hline
\multirow{1}{1em}{$\mathrm{III}^*$} & $2$ & $\{-\frac{3}{2}\log q_v\}$\\
\hline
\end{tabular}
\end{center}
\caption{The sets $W_v^{\mathrm{bad}}$.}
\label{table:1}
\end{table}
\begin{proof}
By a similar argument to the proof of Proposition \ref{propositionheightstamagawa1}, each non trivial coset of $E(F_v)/E_0(F_v)$ is represented by an element in $\mathcal{X}(\OO_v)$ and there exists at least one point in $\mathcal{X}(\OO_v)$ which reduces to a non-singular point in $\overline{E}_{ns}(\mathbb{F}_v)$ if and only if $\#\overline{E}_{ns}(\mathbb{F}_v)>1$.

By Lemma \ref{lemmalocalheights}\thinspace \ref{goodreduction}, if $P\in\mathcal{X}(\OO_v)$ reduces to a non-singular point, then $\lambda_v(P)=0$; by Lemma \ref{conditionsonq}, such $P$ exists unless $q_v=2$ and $E$ is split multiplicative at $2$.

Therefore, by Lemma \ref{lemmalocalheights}\thinspace\ref{nontrivialcosets}, it suffices to show that $W_v^{\mathrm{bad}}$ coincides exactly with the values of $n_v\lambda_v$ on $E(F_v)/E_0(F_v)\setminus\{0\}$. For this, we use the work of Cremona--Prickett--Siksek\cite{heightdifference} for the local heights of the real canonical height. The proof of Proposition 6 in \emph{loc.\ cit.}\ can be used verbatim here with the $p$-adic logarithm in place of the real one and \textsc{Table} \ref{table:1} is nothing but the translation of \cite[Table 2]{heightdifference} to the $p$-adic setting.
\end{proof}

\subsubsection{}\label{sec:heightsabovep}The $p$-adic local N\'eron function at a prime $v\mid p$ is not unique: it depends on a choice of subspace $N_v\subset H_{\mathrm{dR}}^1(E/F_v)$ complementary to the space of holomorphic differentials (see \cite{ColemanGross}). Let $\xi_v$ be the one-form of the second kind with a double pole at $O$ and no others, representative of the class in $N_v$ dual to $\omega$ with respect to the cup product (i.e.\ such that $[\omega]\cup [\xi_v]=1$). Let $\tr_{F_v/\mathbb{Q}_p}$ denote the field trace. Then by \cite[Theorem 4.1]{ColemanGrosspAdicSigma}, for all $P\in E(F_v)\setminus\{O\}$ one has 
\begin{equation*}
\lambda_v(P) = \frac{1}{n_v}\tr_{F_v/\mathbb{Q}_p}\left(2\int_b^P\omega\xi_v\right).
\end{equation*}
In particular,
\begin{equation*}
\xi_v = \eta + \gamma \omega \quad \text{for some}\ \gamma\in F_v
\end{equation*}
and hence
\begin{equation*}
\lambda_v(P) = \frac{1}{n_v}\tr_{F_v/\mathbb{Q}_p}(2D_2(P) + \gamma\Log(P)^2).
\end{equation*}
In \cite[Corollary 3.2]{ColemanGrosspAdicSigma}, it is shown that if $E$ has good ordinary reduction at $v$ and $N_v$ is the unit root eigenspace of Frobenius, then $\lambda_v$ is related to the logarithm of the $v$-adic sigma function of Mazur--Tate \cite{padicsigma}.

In fact, it is easy to see that their proof shows the following stronger result.
\begin{prop}
\label{sigmadilogarithm}
Let $x(t)$ be the Laurent series expansion of $x$ in terms of the parameter for the formal group $t=-\frac{x}{y}$. Let $\sigma_v^{(\gamma)}(t)= t+\cdots \in F_v[[t]]$ be the unique odd\footnote{Odd as a function on a subset of the formal group and not as a function of $t$.} function satisfying
\begin{equation*}
x(t)+\gamma = -\frac{d}{\omega}\left(\frac{1}{\sigma_v^{(\gamma)}} \frac{d\sigma_v^{(\gamma)}}{\omega}\right)
\end{equation*}
and let $V$ be a neighbourhood of $O$ on which $\sigma_v^{(\gamma)}$ converges. Then, for all $P\in V\setminus\{O\}$, we have
\begin{equation*}
\lambda_v(P) = -\frac{2}{n_v}\tr_{F_v/\mathbb{Q}_p}(\log_v(\sigma_v^{(\gamma)}(P))),
\end{equation*}
where $\log_v\colon F_v^{\times}\to F_v$ extends $\log$.
\end{prop}
For our applications, we may assume that there is an isomorphism $F_v\simeq \mathbb{Q}_p$, which is now fixed.
 Since $\lambda_v$ is not unique, we will use the following convention. If the reduction is good \emph{ordinary} at each prime $v$ above $p$, we choose $N_v$ to be the unit root eigenspace of Frobenius, i.e.
\begin{equation*}
\gamma = C,
\end{equation*}
where $C$ is defined in (\ref{constantE2}). If $P$ belongs to the formal group at $v$, then Proposition \ref{sigmadilogarithm} says that
\begin{equation*}
\lambda_v(P) = -2\log(\sigma_p(P)),
\end{equation*}
where $\sigma_p$ is Mazur--Tate's $p$-adic sigma function. Furthermore, in this case the global $p$-adic height coincides with the $p$-adic height of Mazur--Tate. 

If $E$ is good \emph{supersingular} at some prime $v\mid p$, we let, for each $v\mid p$,
\begin{equation*}
\gamma = \frac{a_1^2+4a_2}{12},
\end{equation*}
so that $\sigma_p^{(\gamma)}$ is the $p$-adic sigma function of Bernardi \cite{bernardi}. This choice of $\gamma$ gives a power series $\sigma_v^{(\gamma)}(t)$ with coefficients in $F$ (and in fact $\Q$ in our case), which is related to the Taylor expansion $\sigma(z)$ of the complex Weierstrass sigma function by the change of variables $z=L_v(t)$, where $L_v$ is the formal group logarithm. Unlike the $p$-adic sigma function of Mazur and Tate, the one of Bernardi does not converge on the whole formal group over $\overline{F_v}$, as it may not have $p$-adically integral coefficients, as a power series in $t$. However, since we are assuming that $F_v\simeq \mathbb{Q}_p$, the function $\sigma_p^{(\gamma)}$ converges on all the points $P$ of the formal group whose coordinates are defined over $F_v$, since these satisfy $\ordnop_v(t(P))>\frac{1}{p-1}$.

In both the ordinary and supersingular cases, $\lambda_v$ satisfies (see \cite{ColemanGross}, \cite{padicsigma}, \cite{bernardi}):
\begin{enumerate}[label=(\roman*)]
\item $\lambda_v$ is locally analytic on $\mathcal{X}(\OO_v)$.
\item\label{prop:transformationunderm} For all
$P\in E(F_v)$ and all $m\geq 1$ with $mP\neq O$,
\begin{equation*}
\lambda_v(mP)=m^2\lambda_v(P) -\frac{2}{n_v}\tr_{F_v/\mathbb{Q}_p}(\log_v(f_m(P))).
\end{equation*}
\item \label{prop:quasi_parallelogram_law_above_p} For all $P,Q\in E(F_v)$ such that $P,Q,P\pm Q\neq O$,
\begin{equation*}
\lambda_v(P+Q)+\lambda_v(P-Q) = 2\lambda_v(P)+2\lambda_v(Q)-\frac{2}{n_v}\tr_{F_v/\mathbb{Q}_p}(\log_v(x(P)-x(Q))).
\end{equation*}
\item\label{prop:transformationunderauto} If $\psi$ is an automorphism of $E$ defined over $F_v$, then, for all $P\in E(F_v)$, $\lambda_v(\psi(P))=\lambda_v(P)$. In view of the assumption that $F_v\simeq\mathbb{Q}_p$ and by Deuring's criterion, at supersingular primes this is simply saying that $\lambda_v$ is an even function. At ordinary primes, let $\zeta$ be the root of unity such that $\psi^{*}(\omega) = \zeta \omega$. Then $f_m(\psi(P)) = \zeta^{1-m^2}f_m(P)$ by \cite[Appendix I, Proposition 2]{padicsigma}, and the Mazur--Tate $p$-adic sigma function satisfies $\sigma_p(\psi(P)) = \zeta \sigma_p(P)$ if $P$ is in the formal group \cite[\S 3]{padicsigma}. The claim then follows since $\log(\zeta)= 0$. Note that invariance under any automorphism would also hold if we used the Bernardi sigma function to define the local heights at ordinary primes, since for curves of $j$-invariant $0$ or $1728$ the weight 2 Eistenstein series vanishes, so the Bernardi and Mazur--Tate $p$-adic sigma function are equal.
\end{enumerate}
We also remark that if $L/F$ is a finite field extension, $w$ is a prime of $L$ above $v$, where $v$ is any prime of $F$, and $P\in E(F_v)$, then
\begin{equation*}
\lambda_v(P)=\lambda_w(P).
\end{equation*}

\subsection{Proof of Theorems \ref{level2rank0} and \ref{level2rank1}}
\label{sec:proofs}
It would be pointless to reproduce here the whole proofs, as they are straightforward from \S \ref{padicheight} and the proofs in \cite{nonabelianconjecture}. Thus we content ourselves with giving a sketch and correcting a few imprecisions in Theorem 1.12 of \emph{loc.\ cit.}

We start with some notation and we refer the reader to \cite{KimP1},\cite{Kimunipotent}, \cite{nonabelianconjecture} for more details. Let $T=S\cup\{p\}$ and denote by $G_T$ the Galois group of the maximal extension of $\mathbb{Q}$ unramified outside $T$. For a prime $q$, write $G_{q}$ for the absolute Galois group of $\mathbb{Q}_q$. For $q\in T$, $G_q$ may be identified with a subgroup of $G_T$. For $q\not\in T$, this is not possible; however, we may still define maps $G_q\to G_T$ which are trivial on the inertia subgroup $I_q\leq G_q$.

Let $U$ be the unipotent $p$-adic \'etale fundamental group of $\mathcal{X}_{\overline{\mathbb{Q}}}$ at $b$ and $U_n$ the quotient of $U$ by the $n$-th level of its central series.

For each prime $q$ and $n\geq 1$, we have commutative diagrams
\begin{equation*}
\begin{tikzcd}
\mathcal{X}(\mathbb{Z}) \arrow[r] \arrow[d] & \mathcal{X}(\mathbb{Z}_q) \arrow[d, "j_q^n"] \\
H^1_f(G_T,U_n) \arrow[r, "\mathrm{loc}^n_q"] & H^1(G_q,U_n);
\end{tikzcd}
\end{equation*}
here the $H^1$ are cohomology sets and $H^1_f(G_T,U_n)=(\mathrm{loc}^n_p)^{-1}(H_f^1(G_p,U_n))$, where $H_f^1(G_p,U_n)$ is the subset of $H^1(G_p,U_n)$ of crystalline $U_n$-torsors. 
We are interested in determining
\begin{equation*}
\mathcal{X}(\mathbb{Z}_p)_n = (j_p^n)^{-1}(\mathrm{loc}^n_p(\mathrm{Sel}^n(\mathcal{X}))),
\end{equation*}
where the Selmer scheme $\mathrm{Sel}^n(\mathcal{X})$ is defined as 
\begin{equation*}
\mathrm{Sel}^n(\mathcal{X}) = \bigcap_{q\neq p}(\mathrm{loc}^n_q)^{-1}(\mathrm{Im} j_q^n).
\end{equation*}
From now on, we will focus on $n=2$ and will drop the superscript $n$ from the maps $j_q$ and $\mathrm{loc}_q$. 
\begin{proof}[Proof of Theorem \ref{level2rank0}] 
If $\mathcal{X}(\mathbb{Z}_q)$ is empty for some $q$, then $\mathcal{X}(\mathbb{Z}_p)_2$ is trivially empty. Lemma \ref{conditionsonq} shows that this occurs precisely when $E$ has good reduction at $q$, where $q=2$ or $3$, and $\overline{E}(\mathbb{F}_q)=\{O\}$, or when $E$ has split multiplicative reduction of type $\mathrm{I}_1$ at $q=2$. This shows (\ref{trivial}).

We may now suppose that $\mathcal{X}(\mathbb{Z}_q)\neq \emptyset$ for all $q$ (including $q=p$).
Since $E(\mathbb{Q})$ has rank $0$ and the $p$-primary part of the Tate--Shafarevich group is finite, by Lemma 5.2 in \cite{nonabelianconjecture},
$$\mathrm{Sel}^2(\mathcal{X})\subset H_f^1(G_T,\mathbb{Q}_p(1)),$$
where $H_f^1(G_T,\mathbb{Q}_p(1))\subset H_f^1(G_T,U_2)$
via $0\to \mathbb{Q}_p(1)\to U_2\to V_p(E)\to 0$.

For $q\neq p$, we have
\begin{equation*}
j_q\colon \mathcal{X}(\mathbb{Z}_q)\to H^1(G_q,U_2) \simeq H^1(G_q,\mathbb{Q}_p(1))\simeq \mathbb{Q}_p,
\end{equation*}
the last map being $c\mapsto \log\chi \cup c\in H^2(G_q,\mathbb{Q}_p(1))\simeq \mathbb{Q}_p$, where $\chi$ is the $p$-adic cyclotomic character and $\cup$ is the cup product (note $\log\chi\in H^1(G_q,\Q_p)$). The  middle bijection is proved in \cite[\S 4.1.5]{nonabelianconjecture}.
Thus, finding the image of $j_q$ is equivalent to finding
\begin{equation*}
\{\phi_q(z) \colonequals \log\chi\cup j_q(z) : z\in\mathcal{X}(\mathbb{Z}_q) \}.
\end{equation*}
Theorem 4.1.6 in \cite{nonabelianconjecture} shows that 
\begin{equation*}
2\phi_q\colon \mathcal{X}(\mathbb{Z}_q)\to \mathbb{Q}_q,\qquad z\mapsto  2(\log\chi\cup j_q(z))
\end{equation*}
is the restriction to $\mathcal{X}(\mathbb{Z}_q)$ of a $p$-adic local N\'eron function in the sense of \S \ref{sec:heightsawayfromp} and must thus be equal to the function $\lambda_q$.

In particular, for each $q\neq p$, the set $2\phi_q(\mathcal{X}(\mathbb{Z}_q))$ is the finite set described by Proposition \ref{propositionheightstamagawa1} and Proposition \ref{localheightsnontrivialtamagawa}.

The cup product $\log\chi\cup c$, for $c\in H_f^1(G_p,\mathbb{Q}_p(1))$, and local reciprocity also yield an isomorphism
\begin{equation*}
\psi_p\colon H_f^1(G_p,\mathbb{Q}_p(1))\xrightarrow{\sim} H^2(G_p,\mathbb{Q}_p(1))\simeq \mathbb{Q}_p
\end{equation*}
and we get a commutative diagram
\begin{equation*}
\begin{tikzcd}
H_f^1(G_T,\mathbb{Q}_p(1)) \arrow[r, hook, "\oplus_{q\in T}\mathrm{loc}_q"] \arrow[d, "g=\log\chi\cup\,\cdot"]
& H_f^1(G_p,\mathbb{Q}_p(1))\oplus\bigoplus_{q\in S} H^1(G_q,\mathbb{Q}_p(1)) \isoarrow{d} \\
H^2(G_T,\mathbb{Q}_p(1)) \arrow[r, hook, "\oplus_{q\in T}\mathrm{loc}_q"]
& \bigoplus_{q\in T} H^2(G_q,\mathbb{Q}_p(1))\simeq \bigoplus_{q\in T}\mathbb{Q}_p.
\end{tikzcd}
\end{equation*}
On the other hand, by global class field theory and Hilbert's Theorem 90, the image of $H^2(G_T,\mathbb{Q}_p(1))$ in the bottom row is the kernel of the map
\begin{equation*}
\bigoplus_{q\in T}\mathbb{Q}_p\rightarrow \mathbb{Q}_p,\qquad (a_q)\to \sum_q a_q
\end{equation*}
and by dimension considerations, one concludes that the map $g$ is in fact also an isomorphism. 

From above we know that the image of $\bigoplus_{q\in S}j_q(\mathcal{X}(\mathbb{Z}_q))$ in $\bigoplus_{q\in S}\mathbb{Q}_p$ is precisely $(1/2)\bigoplus_{q\in S} W_q$, where
\begin{itemize}
\item if $[E(\mathbb{Q}_q):E_0(\mathbb{Q}_q)]=1$, then $W_q=\{0\}$ 
\item if $[E(\mathbb{Q}_q):E_0(\mathbb{Q}_q)]\neq 1$, then
\begin{equation*}
W_q=\begin{cases}
W_q^{\mathrm{bad}} & \text{if}\ q=2\ \text{and}\ E\ \text{is split multiplicative at}\ q,\\
W_q^{\mathrm{bad}} \cup\{0\} & \text{otherwise}
\end{cases}
\end{equation*}
and $W_q^{\mathrm{bad}}$ is defined in Proposition \ref{localheightsnontrivialtamagawa}.
\end{itemize}
Let $W=\prod_{q\in S}W_q$. It follows from the above that for every $w = (w_q)_{q\in S}\in W$ there exists a unique $c\in H_f^1(G_T,\Q_p(1))$ with $2(\log \chi \cup \mathrm{loc}_q(c))= w_q$ for every $q\in S$. Furthermore, this satisfies $2\psi_p(\mathrm{loc}_p(c))=-||w||$. 
On the other hand, if $q\not\in T$ and $c\in H_f^1(G_T,\mathbb{Q}_p(1))$ then $\mathrm{loc}_q(c)=0$ and $\mathrm{Im}(j_q)=\{0\}$ by Proposition \ref{propositionheightstamagawa1}.
Therefore,
\begin{equation*}
\mathrm{Sel}^2(\mathcal{X})=\bigcap_{q\in S} \mathrm{loc}_q^{-1}(\mathrm{Im}j_q)
\end{equation*}
and
\begin{equation*}
\mathrm{loc}_p(\mathrm{Sel}^2(\mathcal{X}))=\bigcup_{w\in W}\{c\in H^1_f(G_p,\mathbb{Q}_p(1)) : 2\psi_p(c)+||w|| =0\}.
\end{equation*}
It remains to compute the preimage of this set under $j_p$. As is shown in \cite[\S 5.7]{nonabelianconjecture}, we find
\begin{equation*}
\mathcal{X}(\mathbb{Z}_p)_2= \{z \in \mathcal{X}(\mathbb{Z}_p):  \Log(z)  =0, 2D_2(z) + ||w||=0 \};
\end{equation*}
indeed, the condition $\Log(z)=0$ is equivalent to requiring that $$j_p(z)\in H_f^1(G_p,\mathbb{Q}_p(1))\subset H_f^1(G_p,U_2)$$
and the other condition comes from the explicit formula
\begin{equation*}
\psi_p(j_p(z))=D_2(z)\quad \text{for}\ j_p(z)\in H_f^1(G_p,\mathbb{Q}_p(1)).\qedhere
\end{equation*}
\end{proof}

\begin{rmk}
\label{rmk:corrections}
The corrections to the proof in \cite{nonabelianconjecture} made here are the following. First of all, if $\mathcal{X}(\mathbb{Z}_q)$ is empty for some $q$, the proof does not hold. Of course this is a trivial case (treated in (\ref{trivial})), but it is not clear that the union given in Theorem 1.12 of \emph{loc.\ cit.} should be empty. In fact, in Example \ref{Example:2} we find a curve satisfying the hypotheses of Theorem \ref{level2rank0}\thinspace (\ref{trivial}), but for which $\bigcup_{w\in W}\phi(w)\neq \emptyset$.

Secondly, if the reduction type at $q$ is non-split multiplicative of type $\mathrm{I}_m$, with $m>2$, not all the values in their sets $W_q$ will be attained by a point in $E(\mathbb{Q}_q)\setminus E_0(\mathbb{Q}_q)$. Therefore, if a prime in $S$ is non-split multiplicative, their statement should just be an inclusion of $\mathcal{X}(\mathbb{Z}_p)_2$ into the union of the $\Psi(w)$. One should note, however, that it seems like this was taken care of in the computations when the Tamagawa number at $q$ is $1$, but not when it is $2$ (and hence $m$ is even).

For the same reasons, if $q=2$ is a prime of split multiplicative reduction of type $\mathrm{I}_m$, with $m>0$, the element $0$ should not be included in $W_q$.

We remark that in all the examples they provided the set they computed turned out to be equal to $\mathcal{X}(\mathbb{Z})$ and hence to $\mathcal{X}(\mathbb{Z}_p)_2$.
\end{rmk}

\begin{rmk}
\label{rmk:globaltorsion}
The proof of \cite[Theorem 1.12]{nonabelianconjecture} is rather technical. However, for an elliptic curve of any rank, denoting by $\mathcal{X}(\mathbb{Z})_{\mathrm{tors}}$ the set of points of $\mathcal{X}(\mathbb{Z})$ of finite order, the easier statement
\begin{equation}
\label{eq:XZtors}
\mathcal{X}(\mathbb{Z})_{\mathrm{tors}}\subseteq \bigcup_{w\in W}\{z \in \mathcal{X}(\mathbb{Z}_p):  \Log(z)  =0, 2D_2(z) + ||w||=0 \}
\end{equation}
is elementary to prove. Indeed, the condition $\Log(z)=0$ cuts out the torsion points in $\mathcal{X}(\mathbb{Z}_p)$. On the other hand, let $$\gamma =\begin{cases} C &\ \text{if}\ E\ \text{is ordinary at }p,\\
\frac{a_1^2+4a_2}{12} &\ \text{otherwise.}
\end{cases}$$
Then we have
\begin{align*}
\begin{cases}
\Log(z)  =0,\\
2D_2(z) + ||w||=0
\end{cases} &\iff\quad \begin{cases}
\Log(z) =0\\
2D_2(z)+\gamma\Log(z)^2+ ||w||=0
\end{cases}\\
&\iff \quad\begin{cases}
\Log(z) =0\\
\lambda_p(z)+ ||w||=0
\end{cases}
\end{align*}
and, for $z\in \mathcal{X}(\mathbb{Z})$, $h_p(z)=\lambda_p(z)+ ||w||$ for some $w\in W$, where $h_p$ and $\lambda_p$ are the global and local $p$-adic heights of \S \ref{padicheight}. In particular, if $z\in\mathcal{X}(\mathbb{Z})_{\mathrm{tors}}$, we have $h_p(z)=0$. In fact, we could have also obtained a height function by setting the local height at $p$ to be the dilogarithm $2D_2(z)$. 
\end{rmk}

\begin{proof}[Proof of Theorem \ref{level2rank1}]
The proof of part (\ref{trivialrank1}) is identical to the proof of Theorem \ref{level2rank0}\thinspace (\ref{trivial}). The proof of part (\ref{nontrivialrank1}) is straightforward from \S \ref{sec:proofs} and the proof of \cite[Proposition 5.12]{nonabelianconjecture}: the idea is that any two quadratic functions on the rank-one $E(\mathbb{Q})$ must be linearly dependent. Note that in the semistable case our statement is slightly different, as our set $W$ is smaller if there are primes of non-split multiplicative reduction of type $\mathrm{I}_{m}$, with $m>2$ and also if $q=2$ is a prime of split multiplicative reduction (cf.\ Remark \ref{rmk:corrections}).
\end{proof}

\begin{rmk}
\label{rmk:generalisedheights}
Theorem \ref{level2rank1} is a consequence of the quadraticity of the $p$-adic height and of the square of the elliptic curve logarithm. Of course, that the latter function is quadratic follows from the linearity of the logarithm. We remark that $\Log^2$ is in fact the $p$-adic height attached to the basis element $\omega$ of the Dieudonn\'e module of $E$, in the language of generalised $p$-adic heights (see for instance \cite[\S 4]{steinwuth}), whereas the $p$-adic height $h_p$ of Mazur--Tate is the one attached to an eigenvector with unit eigenvalue under the action of Frobenius. We could remove the assumption that $p$ is ordinary in the statement of Theorem \ref{level2rank1} if we replaced $C$ with $\frac{a_1^2+4a_2}{12}$ and let $h_p$ be the global $p$-adic height that we defined in \S \ref{padicheight} when $p$ is supersingular.
\end{rmk}

\section{Obstructions to $\mathcal{X}(\mathbb{Z})=\mathcal{X}(\mathbb{Z}_p)_2$ in rank $0$.}
\label{sec:obstructions}
We now derive some criteria for $\mathcal{X}(\mathbb{Z}_p)_2\supsetneq \mathcal{X}(\mathbb{Z})$. In Section \ref{sec:Computations}, we will compute $\mathcal{X}(\mathbb{Z}_p)_2$ for several curves and provide explicit examples for the results of this section.
Since a necessary condition for $z\in\mathcal{X}(\mathbb{Z}_p)_2$ is that $\Log(z)=0$, which can only occur if $z\in \mathcal{X}(\mathbb{Z}_p)_{\mathrm{tors}}$, after having fixed all appropriate embeddings, we must have $$\mathcal{X}(\mathbb{Z}_p)_2\subset\mathcal{E}(\overline{\mathbb{Z}})_{\mathrm{tors}}=E(\overline{\mathbb{Q}})_{\mathrm{tors}}.$$

In \S \ref{sec:quad_sat} we derive a stronger necessary condition, which roughly says that if a point lies in $\mathcal{X}(\Z_p)_2$, then its local heights cannot distinguish it from a point defined over $\Q$. To motivate the intuition behind this, it is more natural to first investigate sufficient conditions. In particular, we consider two reasons why extra points could arise in $\mathcal{X}(\mathbb{Z}_p)_2$: invariance of local heights under automorphism (\S\ref{sec:automorphisms}) and existence of non-cyclotomic local heights over certain number fields (\S\ref{sec:other_characters}). Sometimes, a combination of the two is needed, as is the case in Proposition \ref{prop:twists_X049}, which provides us with infinitely many curves over $\mathbb{Q}$ with points over a quartic field appearing in $\mathcal{X}(\mathbb{Z}_p)_2$ for suitable choices of $p$. In \S \ref{sec:proof_of_nonequality}, we use this to deduce Theorem \ref{thm:nonequality}.

We start by proving an elementary fact: any obstruction to $\mathcal{X}(\mathbb{Z}_p)_2 = \mathcal{X}(\mathbb{Z})$ must come from points defined over number fields larger than $\mathbb{Q}$.
\begin{prop}
\label{prop:integralequalrational}
Suppose that $E$ satisfies the assumptions of Theorem \ref{level2rank0}\thinspace (\ref{nontrivial}) and that $p$ is an odd prime of good reduction. Then 
\begin{equation*}
\mathcal{X}(\mathbb{Z}_p)_2\cap \mathcal{E}(\mathbb{Z}) = \mathcal{X}(\mathbb{Z}).
\end{equation*}
\end{prop}

\begin{proof}
Suppose $P\in \mathcal{X}(\mathbb{Z}_p)_2\cap\mathcal{E}(\mathbb{Z})$. In particular, $\Log(P)=0$, so $P$ is torsion and hence $h_p(P)=0$. On the other hand, since $P\in \mathcal{X}(\mathbb{Z}_p)_2$,
\begin{equation*}
\lambda_p(P)+||w||=0
\end{equation*}
for some $w\in W$. Therefore,
\begin{equation*}
\sum_{q\neq p}\lambda_q(P)=||w||.
\end{equation*}
By definition, we have
\begin{equation*}
\sum_{q\neq p}\lambda_q(z) = \sum_{q\neq p}\alpha_q\log {q},\qquad ||w||=\sum_{q\neq p}\beta_q\log{q}
\end{equation*}
for some $\alpha_q,\beta_q\in\mathbb{Q}$, $\alpha_q = 0$ for all but finitely many $q$, $\beta_q=0$ for all $q\not\in S$ and $\beta_q\leq 0$ for all $q$. Thus
\begin{equation*}
\log\biggl(\prod_{q\neq p} q^{d(\alpha_q-\beta_q)}\biggr) =0, 
\end{equation*}
for some non-zero integer $d$ such that $d(\alpha_q-\beta_q)\in\mathbb{Z}$ for all $q$.
This implies that $\alpha_q = \beta_q$ for all $q$, since the kernel of the $p$-adic logarithm is the subgroup of $\mathbb{Q}_p^{\times}$ generated by $p$ and by the roots of unity. Suppose that $z$ is not integral at $q$. Then by Lemma \ref{lemmalocalheights}\thinspace\ref{goodreduction}, $\alpha_q>0$, but $\beta_q \leq 0$, a contradiction.
\end{proof}

\begin{rmk}
According to \cite[VII Application 3.5]{silvermanAEC}, if $P\in\mathcal{E}(\mathbb{Z})_{\mathrm{tors}}$ then $P$ is integral at all primes except possibly at $2$ if $P$ is $2$-torsion.  Thus the only $q$ for which the proof of Proposition \ref{prop:integralequalrational} is non-empty is $q =2$. However, note that, with minor changes, the same proof shows the perhaps less trivial fact that $\mathcal{X}(\Z_p)_2^{\prime}\cap \mathcal{E}(\Z)=\mathcal{X}(\Z)$ in rank $1$. 
\end{rmk}

\begin{rmk}
\label{remarkonminimalmodels}
Unlike in \S \ref{padicheight}, given a prime $v$ of a number field, henceforth the notation $\lambda_{v}$ will be used for the local height at $v$ computed with respect to the model $\mathcal{E}$, which may not be minimal at $v$. The translation with the values computed with respect to a minimal model (Lemma \ref{lemmalocalheights} and Proposition \ref{localheightsnontrivialtamagawa}) is given by (\ref{lambdaqlambdaqmin}).
\end{rmk}
\subsection{Automorphisms}
\label{sec:automorphisms}
Recall that local heights are even functions. Therefore, if $K$ is a quadratic field with $\mathrm{Gal}(K/\Q)=\braket{\tau}$ and $z\in \mathcal{X}(\OO_K)$ satisfies $\tau(z)= -z$, then
\begin{equation*}
h_p(z)= \sum_{q}\lambda_{\mathfrak{q}}(z) = \lambda_{\fp}(z) + \sum_{q\in S}\lambda_{\fq}(z),
\end{equation*}
where $\fq$ (resp.\ $\fp$) is any prime of $K$ above $q$ (resp.\ $p$). Intuitively, in terms of local $p$-adic heights, the point $z$ behaves as if it were defined over $\mathbb{Q}$; if furthermore $z$ is a torsion point and $p$ is split in $K$, then $z$ will give rise to a point in $\mathcal{X}(\Z_p)_2$, provided that $\lambda_{\fq}(z)\in W_q$ at every $q\in S$. If the $j$-invariant of $E$ is different from $0$ and $1728$, the automorphism group of $E/\overline{\Q}$ is generated by $z\mapsto -z$. On the other hand, if $j(E)\in \{0,1728\}$, we can use the invariance of our local heights under any automorphism (cf.\ property \ref{property:invarianceautoawayp} in \S \ref{sec:heightsawayfromp} and property \ref{prop:transformationunderauto} in \S  \ref{sec:heightsabovep}) to generalise the above example as follows.
\begin{prop}
\label{automorphismeffect}
Suppose that $E$ satisfies the assumptions of Theorem \ref{level2rank0} and that $p$ is an odd prime of good reduction. Let $K$ be a Galois extension of $\mathbb{Q}$, such that there is an embedding $\rho\colon K\hookrightarrow \mathbb{Q}_p$. Extend $\rho$ to a map $\mathcal{E}(\OO_K)\hookrightarrow \mathcal{E}(\mathbb{Z}_p)$.  Let $z\in\mathcal{X}(\OO_K)_{\mathrm{tors}}$ and suppose that for every $\tau\in\mathrm{Gal}(K/\mathbb{Q})$ there exists $\psi_{\tau}\in \mathrm{Aut}(E/\overline{\mathbb{Q}})$ such that $\tau(z)=\psi_{\tau}(z)$.
\begin{enumerate}
\item \label{firststatement} For each rational prime $q$, let $\mathfrak{q}$ be one (any) prime of $K$ above $q$ and let $\lambda_{\mathfrak{q}}$ be the local height at $\mathfrak{q}$ with respect to the model $\mathcal{E}$. If 
\begin{equation*}
\sum_{q\in S}\lambda_{\mathfrak{q}}(z) = ||w||
\end{equation*}
for some $w\in W$, then $\rho(z)\in\mathcal{X}(\mathbb{Z}_p)_2$.
\item \label{secondstatement} In particular, if $z = \psi^{\prime}(P)$ for some $\psi^{\prime}\in \mathrm{Aut}(E/\overline{\mathbb{Q}})$ and some $P\in\mathcal{X}(\mathbb{Z})$, then $\rho(z)\in\mathcal{X}(\mathbb{Z}_p)_2$.
\end{enumerate}
\end{prop}

\begin{proof}
The assumption that $z$ is a torsion point implies that $h_p(z)=0$ and $\Log(z)=0$.
Since for $\tau\in\mathrm{Gal}(K/\mathbb{Q})$ there is an automorphism $\psi_{\tau}$ of $E$ which acts on $z$ in the same way as $\tau$ and local heights are invariant under automorphisms, for each prime $\mathfrak{q}$ of $K$ we have
\begin{equation*}
\lambda_{\mathfrak{q}}(z)=\lambda_{\mathfrak{q}}(\psi_{\tau}(z))= \lambda_{\mathfrak{q}}(\tau(z))=\lambda_{\tau^{-1}(\mathfrak{q})}(z).
\end{equation*}
Therefore,
\begin{equation*}
0 = [K\colon \mathbb{Q}]h_p(z) = [K\colon \mathbb{Q}]\bigg(\lambda_p(\rho(z))+\sum_{q\in S}\lambda_{\mathfrak{q}}(z)\bigg)
\end{equation*}
and (\ref{firststatement}) follows. For (\ref{secondstatement}), since $z = \psi^{\prime}(P)$, we have, similarly to above,
\begin{equation*}
\lambda_{\mathfrak{q}}(z) = \lambda_{\mathfrak{q}}(\psi^{\prime}(P)) = \lambda_q(P).
\end{equation*}
In particular, the hypothesis of (\ref{firststatement}) is satisfied.
\end{proof}
We now list a few consequences of Proposition \ref{automorphismeffect}. See Section \ref{sec:Computations} for explicit examples.

\begin{cor}
\label{cor:j0automorphism}
Suppose that $E$ satisfies the assumptions of Theorem \ref{level2rank0} and that $p$ is an odd prime of good ordinary reduction.
Suppose that 
\begin{equation*}
\mathcal{E}\colon y^2+a_3y = x^3 + a_6,\quad \text{for some } a_6\in\mathbb{Z}\ \text{and } a_3\in\{0,1\},
\end{equation*}
and that there exists $y_0\in\mathbb{Z}$ such that $a_6-y_0^2-a_3y_0$ is a cube in $\mathbb{Z}$ and the points over $\overline{\mathbb{Q}}$ with $y$-coordinate equal to $y_0$ have finite order. Then $s(x)=x^3+a_6-y_0^2-a_3y_0$ splits completely in $\mathbb{Q}_p$ and for each root $\alpha\in\mathbb{Z}_p$ of $s(x)$, $\pm(\alpha, y_0)\in \mathcal{X}(\mathbb{Z}_p)_2$.
\end{cor}
\begin{rmk}
If in the corollary we have $a_3 = 1$, then $E(\Q)_{\tors}$ is isomorphic to a subgroup of $\Z/3\Z$, since $E(\Q)[2]=\{O\}$ and $E$ has good reduction at $2$ with $\#\overline{E}(\F_2)=3$. By looking at the $3$-rd division polynomial for $E$, it is then straightforward to check that Corollary \ref{cor:j0automorphism} applies non-trivially only if $4a_6 = -(27n^6+1)$, for some $n\in \mathbb{Z}$, $n\equiv 1\bmod{2}$. All such curves are isomorphic over $\Q$ to the elliptic curve \href{http://www.lmfdb.org/EllipticCurve/Q/27/a/3}{27.a3} \cite{lmfdb}. When $a_3 = 0$, there are infinitely many curves non-isomorphic over $\Q$ for which the corollary applies with $y_0 = 0$: see for example \S \ref{sec:proof_of_nonequality}. There is also at least one curve for which the Corollary applies to points of order $6$, namely \href{http://www.lmfdb.org/EllipticCurve/Q/36/a/4}{36.a4} (see Table \ref{table:2}).
\end{rmk}
\begin{proof}
Since $E$ has vanishing $j$-invariant, its automorphism group $\mathrm{Aut}(E/\overline{\mathbb{Q}})$ is a cyclic group of order $6$ generated by $\psi\colon E\to E$, $\psi(x,y)=(\zeta x,-y-a_3)$, for a primitive third root of unity $\zeta$.

Let $x_0\in\mathbb{Z}$ such that $x_0^3=y_0^2+a_3y_0-a_6$. We may assume that $y_0^2+a_3y_0-a_6$ is non-zero, as otherwise the statement of the corollary is trivial. Thus $s(x)$ has three distinct roots $x_0$, $\zeta x_0$ and $\zeta^2 x_0$ in $\overline{\mathbb{Z}}$.

Note also that, by Deuring's criterion \cite[Ch.\ 13, Theorem 12]{Lang:elliptic_functions}, the primes of good ordinary reduction for $E$ split completely in $\mathbb{Q}(\zeta)$, so $s(x)$ splits completely over $\mathbb{Q}_p$.
Successively applying $\psi$ to $(x_0,y_0)\in\mathcal{X}(\mathbb{Z})$ and localising at $p$ we obtain all points of the form $\pm (\alpha,y_0)$. The corollary then follows from Proposition \ref{automorphismeffect}\thinspace (\ref{secondstatement}).
\end{proof}
The following corollary to Proposition \ref{automorphismeffect} is a special case of the motivating example of the beginning of this subsection.
\begin{cor}
\label{cor:semistablecaseautomorphism}
Suppose that $E$ satisfies the assumptions of Theorem \ref{level2rank0} and that $p$ is an odd prime of good reduction. Let $K$ be a quadratic field, in which $p$ splits. Fix an embedding $\rho\colon K\hookrightarrow \mathbb{Q}_p$ and let $\tau$ be the non-trivial element in $\mathrm{Gal}(K/\mathbb{Q})$. Assume that no prime in $S$ ramifies in $K$ and that, if $q\in S$ is inert, then either $E$ has Kodaira symbol $\mathrm{I}_0^*$ at $q$ with Tamagawa number at least $2$ or $E$ has maximal Tamagawa number for its Kodaira symbol. Then
\begin{equation*}
\mathcal{X}(\mathbb{Z}_p)_2\supset \mathcal{X}(\mathbb{Z})\cup \{\rho(z)\in \mathcal{X}(\mathbb{Z}_p) : z\in\mathcal{X}(\OO_K)_{\mathrm{tors}}, \tau(z) = - z \}.
\end{equation*} 
\end{cor}

\begin{proof}
Since no prime in $S$ ramifies in $K/\mathbb{Q}$, Tate's algorithm \cite[IV, \S9]{silvermanadvancedtopics} shows that the equation for $\mathcal{E}$ defines a global minimal model for the base change $E/K$ and that the Kodaira symbol at $\mathfrak{q}\mid q$ is the same as the Kodaira symbol at $q$. The Tamagawa number does not change if $q$ is split in $K$; if $q$ is inert, by assumption the Tamagawa number is unvaried, except possibly if the Kodaira symbol is $\mathrm{I}_0^{*}$.  

If $q$ splits in $K$, fix a prime $\mathfrak{q}$ above it and an isomorphism $\rho_q\colon K_{\mathfrak{q}}\simeq \mathbb{Q}_q$.
Let $z\in\mathcal{X}(\OO_K)_{\tors}$ such that $\tau(z)=-z$.
With the notation as in Proposition \ref{automorphismeffect} and by Proposition \ref{localheightsnontrivialtamagawa}, we have
\begin{equation*}
\sum_{q\in S}\lambda_{\mathfrak{q}}(z)=\sum_{\substack{q\in S\\
q\ \text{split}}}\lambda_q(\rho_q(z))+\sum_{\substack{q\in S\\
q\ \text{inert}}}\lambda_{\mathfrak{q}}(z) = ||w||
\end{equation*}
for some $w\in W$. For the last step note that Proposition \ref{localheightsnontrivialtamagawa} gives the values of $2\lambda_{\mathfrak{q}}(z)$ for $\mathfrak{q}$ inert. However, the norm of $\mathfrak{q}$ is $q^2$. The corollary then follows from Proposition \ref{automorphismeffect}\thinspace (\ref{firststatement}) with $\psi=-\mathrm{id}\in \mathrm{Aut}(E/\overline{\mathbb{Q}})$.
\end{proof}

\begin{rmk}
\label{rmk:j1728}
Another source of quadratic points in $\mathcal{X}(\mathbb{Z}_p)_2$ comes from elliptic curves with $j$-invariant equal to $1728$.
Suppose that $E$ satisfies the assumptions of Theorem \ref{level2rank0}, that $p$ is an odd prime of good reduction and that
\begin{equation*}
\mathcal{E}\colon y^2 = x^3 + a_4x \quad \text{for some } a_4\in\mathbb{Z}, -a_4\not \in \mathbb{Z}^2.
\end{equation*}
Let $z\in\{ (\pm \sqrt{-a_4},0)\}$ and $K=\mathbb{Q}(\sqrt{-a_4})$ be its field of definition.  Let $\psi\in \mathrm{Aut}(E/\overline{\mathbb{Q}})$ be defined by $\psi(x,y)= (-x,iy)$. Then $\psi(z)= \tau(z)$, where $\mathrm{Gal}(K/\mathbb{Q})=\braket{\tau}$. Therefore, under suitable conditions on how the reduction types change in $K/\mathbb{Q}$ and on the splitting of $p$ in $K$, the localisations of the points $z$ appear in $\mathcal{X}(\mathbb{Z}_p)_2$.
\end{rmk}

The following corollary explains how points over biquadratic extensions can show up in $\mathcal{X}(\mathbb{Z}_p)_2$ when the $j$-invariant is zero. For ease of notation, we assume that the $a_3$-coefficient in the equation defining $\mathcal{E}$ is zero, but this assumption could be removed.
\begin{cor}
\label{cor:biquadraticpoints}
Suppose that $E$ satisfies the assumptions of Theorem \ref{level2rank0} and that $p$ is an odd prime of good ordinary reduction.
Suppose that 
\begin{equation*}
\mathcal{E}\colon y^2 = x^3 + a_6 \quad \text{for some } a_6\in\mathbb{Z}
\end{equation*}
and that there exists $x_0\in\mathbb{Z}$ such that the points over $\overline{\mathbb{Q}}$ with $x$-coordinate equal to $x_0$ have finite order. Assume that $p$ splits in $\mathbb{Q}(\sqrt{x_0^3+a_6})$. Let $K=\mathbb{Q}(\sqrt{-3}, \sqrt{x_0^3+a_6})$. For each rational prime $q$, let $\mathfrak{q}$ be one (any) prime of $K$ above $q$ and $\lambda_{\mathfrak{q}}$ the local height at $\mathfrak{q}$ with respect to the model $\mathcal{E}$. Let $\beta\in\mathbb{Z}_p$ be a root of $t(y) = y^2-x_0^3-a_6$. If
\begin{equation*}
\sum_{q\in S}\lambda_{\mathfrak{q}}(x_0,\beta) = ||w||
\end{equation*}
for some $w\in W$, then for each root $\alpha\in\mathbb{Z}_p$ of $s(x) = x^3-x_0^3$ and for each root $\beta\in\mathbb{Z}_p$ of $t(y) = y^2-x_0^3-a_6$, we have $(\alpha,\beta)\in \mathcal{X}(\mathbb{Z}_p)_2$.
\end{cor}

\begin{proof}
If $x_0^3+a_6$ is a square in $\mathbb{Z}$, the statement is precisely Corollary \ref{cor:j0automorphism}. Thus, we may assume that either
\begin{enumerate}[label = (\roman*)]
\item\label{casedeg4} $K$ has degree $4$ over $\mathbb{Q}$, or
\item\label{casedegree2sqrt3} $K = \mathbb{Q}(\sqrt{-3}) = \mathbb{Q}(\sqrt{x_0^3+a_6})$.
\end{enumerate}
Let $\zeta\in K$ be a primitive third root of unity. The automorphism group $\mathrm{Aut}(E/K)$ is generated by $\psi\colon E\to E$, $\psi(x,y)=(\zeta x,-y)$. In case \ref{casedeg4}, the Galois group of $K$ over $\mathbb{Q}$ is generated by two elements: $\sigma$, whose fixed field is $\mathbb{Q}(\sqrt{x_0^3+a_6})$ and $\tau$, whose fixed field is $\mathbb{Q}(\sqrt{-3})$. In case \ref{casedegree2sqrt3}, the Galois group is generated by $\sigma\colon \sqrt{-3}\mapsto -\sqrt{-3}$. Let $P=(a,b)$ where $a\in K$ is a root of $s(x)$ and $b\in K$ is a root of $t(y)$. Then in \ref{casedeg4}
\begin{equation*}
\sigma(P)\in\{P,-\psi(P),\psi^2(P)\},\qquad \tau(P) = -P.
\end{equation*}
Similarly, in case \ref{casedegree2sqrt3}, we have 
\begin{equation*}
\sigma(P)\in \{-P,\psi(P),-\psi^2(P)\}.
\end{equation*}
Therefore, we may apply Proposition \ref{automorphismeffect}\thinspace(\ref{firststatement}).
\end{proof}

\subsection{Non-cyclotomic $p$-adic heights}
\label{sec:other_characters}
The set $\mathcal{X}(\mathbb{Z}_p)_2$ is a finite set of $p$-adic points containing $\mathcal{X}(\mathbb{Z})$. After having fixed a choice of a subspace of $H_{\mathrm{dR}}^1(E/\mathbb{Q}_p)$ complementary to the space of holomorphic forms, there is only one Coleman--Gross global height pairing on $E(\mathbb{Q})$, up to multiplication by a constant. The definition of $\mathcal{X}(\mathbb{Z}_p)_2$ depends on this height function. Nevertheless, when analysing what points could arise in the set $\mathcal{X}(\mathbb{Z}_p)_2\setminus \mathcal{X}(\mathbb{Z})$, we should bear in mind that other global height functions may exist on $E(F)$, where $F$ is a number field, and that these also vanish on $E(F)_{\tors}$. In particular, suppose that 
there exists at least one embedding $\rho\colon F\hookrightarrow \mathbb{Q}_p$. 
It may happen that, for some $w\in W$ and some $Q\in\mathcal{X}(\OO_F)$,
\begin{equation*}
2D_2(\rho(Q))+\gamma \Log(\rho(Q))^2+||w|| = h_{p}^{\star}(Q)
\end{equation*}
for some non-cyclotomic global height $h_p^{\star}$. Then, if $Q$ is in addition a torsion point, we have $\rho(Q)\in \mathcal{X}(\mathbb{Z}_p)_2$.

In order to introduce these more general types of heights, we need to recall the definition and properties of an idele class character.
\begin{mydef}
Let $\mathbb{A}_F^{\times}$ be the group of ideles of $F$. An idele class character is a continuous homomorphism
\begin{equation*}
\chi = \sum_{\mathfrak{q}}\chi_{\mathfrak{q}}\colon \mathbb{A}_F^{\times}/F^{\times}\to \mathbb{Q}_p;
\end{equation*}
here the sum is over all places of $F$.
\end{mydef}
We list some properties of an idele class character $\chi$ (see \cite{QCnfs} for more details).

\begin{enumerate}[label =(P\Roman*)]
\item The local character $\chi_{\mathfrak{q}}$ is trivial at an archimedean place $\mathfrak{q}$. Thus, henceforth $\mathfrak{q}$ will always denote a finite prime.
\item \label{prop:vanishes_on_units} At a prime $\mathfrak{q}$ not above $p$, the local character $\chi_{\mathfrak{q}}$ vanishes on the units $\OO_{\mathfrak{q}}^{\times}$. Thus, the value of $\chi_{\mathfrak{q}}$ at a uniformiser determines $\chi_{\mathfrak{q}}$ completely.
\item At a prime $\mathfrak{p}$ above $p$, the restriction of the character $\chi_{\mathfrak{p}}$ to $\OO_{\mathfrak{p}}^{\times}$ equals the composition
\begin{equation*}
\OO_{\mathfrak{p}}^{\times} \xrightarrow{\log_{\mathfrak{p}}} F_{\mathfrak{p}}\xrightarrow{t_{\mathfrak{p}}} \mathbb{Q}_p 
\end{equation*}
for some $\mathbb{Q}_p$-linear map $t_{\mathfrak{p}}$. Here $\log_{\mathfrak{p}}$ is the restriction to $\OO_\mathfrak{p}^{\times}$ of the extension of $\log$ to $F_{\mathfrak{p}}^{\times}$.
\item\label{prop:fund_units} The character $\chi$ is completely determined by the trace maps $(t_{\mathfrak{p}})_{\mathfrak{p}\mid p}$ and, conversely, a tuple of $\mathbb{Q}_p$-linear maps $(t_{\mathfrak{p}}\colon F_{\mathfrak{p}}\to \mathbb{Q}_p)_{\mathfrak{p}\mid p}$ gives an idele class character $\chi$ if and only if 
\begin{equation}
\label{eq:vanishonunits}
 \sum_{\mathfrak{p}\mid p} t_{\mathfrak{p}}(\log_{\mathfrak{p}}(\rho_{\mathfrak{p}}(\epsilon)))=0\quad \text{for all }\epsilon\in \OO_{F}^{\times},
\end{equation}
where $\rho_{\mathfrak{p}}\colon F\hookrightarrow F_{\mathfrak{p}}$ is the completion (see \cite{QCnfs} for a proof). 
\end{enumerate}
In particular, it suffices to check that (\ref{eq:vanishonunits}) is satisfied for a set of fundamental units and \ref{prop:fund_units} gives a concrete method for classifying all idele class characters for a given number field $F$. The maximal number of independent characters is at least $r_2+1$, where $r_2$ is the number of conjugate pairs of non-real embeddings of $F$ into $\mathbb{C}$ (with equality if Leopoldt's conjecture holds for $F$).

For instance, for any number field $F$, the \emph{cyclotomic} idele class character is the idele class character corresponding to the tuple of trace maps $(\tr_{F_{\mathfrak{p}}/\mathbb{Q}_p})_{\mathfrak{p}\mid p}$. When $F=\mathbb{Q}$ (or $F$ is a totally real abelian number field), this is the only non-trivial idele class character, up to multiplication by a scalar. The $p$-adic height we have considered so far is implicitly associated to this character.

More generally though, we can define a $p$-adic height as a  composition of two maps: firstly, we associate to a point $P\in E(F)$ an idele $i(P)$ and, secondly, we apply to $i(P)$ an idele class character $\chi$. We denote the corresponding local and global heights by $\lambda_{\fq}^{\chi}$ and $h_p^{\chi}$, respectively. The theory of local heights that we outlined in the cyclotomic case in \S \ref{padicheight} goes through unvaried at the primes $\fq\nmid p$, after replacing the $p$-adic logarithm with $-\frac{\chi_{\mathfrak{q}}}{n_{\mathfrak{q}}}$. At the primes $\fp\mid p$, we may assume here that we always work with points not in the residue disk of the point at infinity\footnote{There is a subtlety in the disk at infinity which has to do with the choice of branch of the $\fp$-adic logarithm. See also \cite[Remark 2.1]{QCnfs}.}. So let $z\in \mathcal{X}(\OO_{\fp})$ and $m\in \mathbb{N}$ such that $mz$ is in the domain of convergence of $\sigma_{\fp}^{(\gamma)}$. Then
\begin{equation}
\label{eq:local_p_non_trivial}
\lambda_{\fp}^{\chi}(z)=-\frac{2}{m^2}\frac{\chi_{\fp}}{n_{\fp}}\left(\frac{\sigma_{\fp}^{(\gamma)}(mz)}{f_m(z)}\right)=-\frac{2}{m^2n_{\fp}}t_{\fp}\left(\log_{\fp}\left(\frac{\sigma_{\fp}^{(\gamma)}(mz)}{f_m(z)}\right)\right),
\end{equation}
since
\[\ordnop_{\fp}(\sigma_{\fp}^{(\gamma)}(mz))=\ordnop_{\fp}(x(mz)y(mz)^{-1})=\ordnop_{\fp}(f_m(z))\]
(see \S \ref{sec:algorithmelliptic} and the proof of Theorem \ref{thm:extrapoints} for how to interpret \eqref{eq:local_p_non_trivial} when $mz=O$).
We will omit $\chi$ from our notation when using the cyclotomic character.

\begin{example}
Let $F$ be an imaginary quadratic field in which $p$ splits. Then by \ref{prop:fund_units}, any pair of $\mathbb{Q}_p$-linear maps $\mathbb{Q}_p\to\mathbb{Q}_p$ gives rise to an idele class character. In particular, choosing $(\mathrm{id}, -\mathrm{id})$ gives the so-called \emph{anticyclotomic} character. 
\end{example}

We now give an instance of how the existence of non-cyclotomic heights for imaginary quadratic fields can give rise to points in $\mathcal{X}(\mathbb{Z}_p)_2\setminus \mathcal{X}(\mathbb{Z})$.

\begin{prop}
\label{prop:anticyclotomiccyclotomic}
Suppose that $E$ satisfies the assumptions of Theorem \ref{level2rank0} and that $p$ is an odd prime of good reduction. Let $K$ be an imaginary quadratic field 
in which $p$ splits. Fix an embedding $\rho\colon K\hookrightarrow \mathbb{Q}_p$. Suppose that $z\in \mathcal{X}(\OO_K)_{\mathrm{tors}}$ has good reduction at all primes that split in $K$. Then
\begin{equation*}
2D_2(\rho(z))+\sum_{q\in S}\lambda_{\mathfrak{q}}(z) = 0,
\end{equation*}
where $\mathfrak{q}$ is a prime of $K$ above $q$. In particular, if $\sum_{q\in S}\lambda_{\mathfrak{q}}(z)=||w||$, for some $w\in W$, then $\rho(z)\in\mathcal{X}(\mathbb{Z}_p)_2$.
\end{prop}
\begin{proof}
It suffices to show that $2D_2(\rho(z))+\sum_{q\in S}\lambda_{\mathfrak{q}}(z)$ is the value at $z$ of a height function on $E(K)$, since then the assumption that $z$ is a torsion point will imply the vanishing. The height function that we are after is the one corresponding to an idele class character $\mathbb{A}_K^{\times}/K^{\times}\to \mathbb{Q}_p$ which vanishes on $\OO_{\overline{\mathfrak{p}}}^{\times}$, if $\mathfrak{p}$ is the prime corresponding to the embedding $\rho$. Indeed, with the notation of \ref{prop:fund_units}, consider the idele class character corresponding to $(\mathrm{id}\colon K_{\mathfrak{p}}\simeq \mathbb{Q}_p\to \mathbb{Q}_p,0\colon K_{\overline{\mathfrak{p}}}\simeq \mathbb{Q}_p\to \mathbb{Q}_p)$. Then
\begin{equation*}
\lambda_{\mathfrak{p}}(z)=\lambda_{\mathfrak{p}}^{\chi}(z)\quad \text{and}\quad \lambda_{\overline{\mathfrak{p}}}^{\chi}(z)=0.
\end{equation*}
Furthermore, since $\chi$ factors through $\mathbb{A}_K^{\times}/K^{\times}$ and in view of \ref{prop:vanishes_on_units}, if there is a unique prime $\mathfrak{q}$ above $q$, we have 
\begin{equation*}
\chi_{\mathfrak{q}}(q)= -\chi_{\mathfrak{p}}(q)-\chi_{\overline{\mathfrak{p}}}(q)= -\log(q),
\end{equation*}
so that $2\lambda_{\mathfrak{q}}^{\chi}=\lambda_{\mathfrak{q}}$ for all primes which are either inert or ramified. Thus $2h_p^{\chi}(z)=2D_2(\rho(z))+\sum_{q\in S}\lambda_{\mathfrak{q}}(z)$.
\end{proof}

In some cases, extra points in $\mathcal{X}(\mathbb{Z}_p)_2$ are explained by a combination of automorphisms and non-cyclotomic idele class characters, as in Proposition \ref{prop:twists_X049}. Before we state it and prove it, we first need an auxiliary lemma.

\begin{lemma}
\label{lemma:extending_characters}
Let $F$ be a number field and let $L$ be a finite extension of $F$. Suppose that $\chi\colon \mathbb{A}_{F}^{\times}/F^{\times}\to \mathbb{Q}_p$ is an idele class character determined by the tuple of $\mathbb{Q}_p$-linear maps $(t_{\mathfrak{p}}\colon F_{\mathfrak{p}}\to\mathbb{Q}_p)_{\mathfrak{p}\mid p}$. Then the tuple $(t_{\mathfrak{q}}^{L}\colon L_{\mathfrak{q}}\to\mathbb{Q}_p)_{\mathfrak{q}\mid p}$, defined by $t_{\mathfrak{q}}^L = t_{\mathfrak{p}}\circ \tr_{L_{\mathfrak{q}}/F_{\mathfrak{p}}}$ for $\mathfrak{q}\mid\mathfrak{p}$, determines an idele class character $\chi^L\colon \mathbb{A}^{\times}_L/L^{\times}\to\mathbb{Q}_p$ such that $\chi^{L}|_{\mathbb{A}^{\times}_F/F^{\times}}=[L\colon F]\chi$.
\end{lemma}

\begin{proof}
Each $t_{\mathfrak{q}}^L$ is $\mathbb{Q}_p$-linear as a composition of $\mathbb{Q}_p$-linear maps. We need to check that \ref{prop:fund_units} is satisfied. If $\epsilon\in \OO_L^{\times}$, then
\begin{align*}
\sum_{\mathfrak{q}\mid p}t_{\mathfrak{q}}^L(\log_{\fq}(\rho_{\mathfrak{q}}(\epsilon)))&=\sum_{\mathfrak{p}\mid p}t_{\mathfrak{p}}\circ \biggl(\sum_{\mathfrak{q}\mid\mathfrak{p}}\tr_{L_{\mathfrak{q}}/F_{\mathfrak{p}}}\circ\log_{\fq}(\rho_{\mathfrak{q}}(\epsilon))\biggr)\\
&= \sum_{\mathfrak{p}\mid p}t_{\mathfrak{p}}\circ\log_{\fp}\biggl(\prod_{\mathfrak{q}\mid\mathfrak{p}}N_{L_{\mathfrak{q}}/F_{\mathfrak{p}}}(\rho_{\mathfrak{q}}(\epsilon))\biggr)\\
&= \sum_{\mathfrak{p}\mid p}t_{\mathfrak{p}}\circ\log_{\fp}\big(\rho_{\mathfrak{p}}(N_{L/F}(\epsilon))\big)=0,
\end{align*}
since $N_{L/F}(\epsilon)\in \OO_F^{\times}$. By construction, the resulting idele class character $\chi^L$ restricts to $[L:F]\chi$ on $\mathbb{A}_{F}^{\times}/F^{\times}$. 
\end{proof}

\begin{prop}
\label{prop:twists_X049}
Let $d$ be a non-zero square-free integer and let $E^d$ be the quadratic twist of $X_0(49)$ by $d$; assume that $E^d$ satisfies the assumptions of Theorem \ref{level2rank0} and let $\mathcal{X}^d$ be the complement of the origin in the minimal regular model of $E^d$. Let $p\nmid 7d$ be an odd prime with at least $3$ primes lying above it in $L=\mathbb{Q}[x]/(x^4+7d^2)$. Then $$\mathcal{X}^d(\mathbb{Z}_p)_2\supseteq \mathcal{X}^d(\mathbb{Z})\cup \{\pm  \rho(Q)\}$$ for some $Q\in\mathcal{X}^d(\OO_L)$ of order $4$ and for every embedding $\rho\colon L\hookrightarrow \mathbb{Q}_p$.
\end{prop}
\begin{rmk}
The proposition also holds in rank $1$ if we replace $\mathcal{X}^d(\mathbb{Z}_p)_2$ with $\mathcal{X}^d(\mathbb{Z}_p)_2^{\prime}$.
\end{rmk}

\begin{proof}
The elliptic curve $E = X_0(49)$ has reduced minimal model
\begin{equation}
\label{eq:49minimal}
\mathcal{E}\colon y^2 + x y = x^{3} -  x^{2} - 2 x - 1;
\end{equation}
however, since we are considering quadratic twists of $E$, it is more convenient to work (at least until we introduce heights) with the following model
\begin{equation*}
\mathcal{E}_{\mathrm{short}}\colon y^2 = x^3 - 2835x - 71442,
\end{equation*} 
as then the twist $E^d$ of $E$ by the non-zero square-free integer $d$ admits the Weierstrass equation
\begin{equation*}
\mathcal{E}_{\mathrm{short}}^d\colon y^2 = x^3 - 2835d^2x - 71442d^3.
\end{equation*}

Recall that $E^d$ has complex multiplication by $K = \mathbb{Q}(a)$, where $a$ is a root of $x^2+7$. Over $K[x,y]$, the fourth division polynomial $f_4^d$ of $\mathcal{E}^d_{\mathrm{short}}$ has the following factorisation
\begin{align*}
f_4^d(x,y) = 4y(x -9ad)  (x + 9ad) (x + (-18a + 63)d)  (x + (18a + 63)d)\\
(x^2 - 126xd - 5103d^2).
\end{align*} 
In particular, since 
$$x^3 - 2835d^2x - 71442d^3= (x - 63d) (x + (-9/2a + 63/2)d) (x + (9/2a + 63/2)d),$$
all the points of order $2$ are defined over $K$. As for the points of order $4$, we see that, as a polynomial in $x$, $f_4^d(x,y)/y$ has two roots in $\mathbb{Q}(\sqrt{7})$ and four roots in $K$. Substituting the latter roots into the equation for $\mathcal{E}_{\mathrm{short}}^d$, we find that the $y$-coordinates of the points with
$x=9ad$ and $x=-(18a + 63)d$ are defined over $\mathbb{Q}(\sqrt{-ad})$, whereas those with $x=-9ad$ and $x=-(-18a + 63)d$ are over  $\mathbb{Q}(\sqrt{ad})$.

Therefore, over the quartic field $L= K[x]/(x^2-ad)= K(b)\cong \mathbb{Q}[x]/(x^4+7d^2)$, $E^d(L)[4]\cong \mathbb{Z}/2\times \mathbb{Z}/4$. Let $\mathrm{Gal}(L/K)=\braket{\overline{\tau}}$ and let 
\begin{equation*}
Q_{\mathrm{short}} = (18b^2-63d, \pm(54b^3 - 378bd) )\in \mathcal{E}^d_{\mathrm{short}}(L)[4],\qquad P_{\mathrm{short}}=(63d,0).
\end{equation*}
Then $Q_{\mathrm{short}}$ satisfies
\begin{equation}
\label{eq:QandQtauX049}
\overline{\tau}(Q_{\mathrm{short}}) = -Q_{\mathrm{short}}. 
\end{equation}
Let $Q$ be the image of $Q_{\mathrm{short}}$ in a minimal model $\mathcal{E}^d$ for $E^d$ over $\mathbb{Z}$ and let $P$ be the image of $P_{\mathrm{short}}$. Note that
\begin{itemize}
\item if $d\equiv 1 \bmod{4}$, we may apply to $\mathcal{E}^d_{\mathrm{short}}$ the change of variables
\begin{equation*}
x\mapsto 36x-9d,\qquad y\mapsto 216y+108x
\end{equation*}
to obtain the integral model
\begin{equation*}
\mathcal{E}^{d}_{1}\colon y^2 + xy = x^3  -\frac{3d+1}{4}x^2 - 2d^2x - d^3.
\end{equation*}
The discriminant of $\mathcal{E}^d_{1}$ is $\Delta=-7^3d^6$, so by \cite[VII, Remark 1.1]{silvermanAEC} $\mathcal{E}^d_{1}$ is a minimal model for $E^d/\mathbb{Q}$ and we may set $\mathcal{E}^d=\mathcal{E}^d_1$. Then $x(P) =2d\in \mathbb{Z}$ and $x(Q) \in \OO_K$. 
\item if $d\equiv 2,3 \bmod{4}$, then we may take
\begin{equation*}
\label{eq:cube_non_reduced}
\mathcal{E}^d\colon y^2 = x^3 - 3dx^2 - 32d^2x - 64d^3,
\end{equation*}
which has discriminant $\Delta = -2^{12} \cdot 7^3\cdot d^6$. Minimality of $\mathcal{E}^d$ at the primes different from $2$ follows as in the case $d\equiv 1 \bmod{4}$. At the prime $2$, it can be deduced following Tate's algorithm. We have $x(P)=8d\in\mathbb{Z}$ and $x(Q) \in \OO_K$.
\end{itemize} 

Now, let $p$ be an odd prime of good reduction for $E^d$ which splits in $K$. By Deuring's criterion, this is equivalent to requiring that $p$ is a prime of good ordinary reduction. Let $p\OO_K = \mathfrak{p}\overline{\mathfrak{p}}$. The prime $p$ is unramified also in $L$, since we are assuming that it is of good reduction. 
We suppose furthermore that $\mathfrak{p} \OO_L =\mathfrak{q}_1\overline{\mathfrak{q}_1}$, $\overline{\mathfrak{p}} \OO_L =\mathfrak{q}_2$ or $\overline{\mathfrak{p}} \OO_L =\mathfrak{q}_2\overline{\mathfrak{q}_2}$, for some primes $\mathfrak{q}_1$ and $\mathfrak{q}_2$ of $L$. These conditions, together, are equivalent to those of the statement of the proposition.

By Lemma \ref{lemma:extending_characters}, the idele class character on $\mathbb{A}_K^{\times}/K^{\times}$ which is trivial on $\mathcal{O}_{\overline{\mathfrak{p}}}^{\times}$ (and which we used also in the proof of Proposition \ref{prop:anticyclotomiccyclotomic}) extends to an idele class character on $\mathbb{A}_L^{\times}/L^{\times}$. In particular, the tuple of linear maps $$(\mathrm{id}_{L_{\mathfrak{q}_1}\simeq \mathbb{Q}_p}, \mathrm{id}_{L_{\overline{\mathfrak{q}_1}}\simeq \mathbb{Q}_p}, 0\colon \prod_{\mathfrak{q}\mid\overline{\mathfrak{p}}} L_{\mathfrak{q}}\to \mathbb{Q}_p)$$  determines an idele class character $\chi$. Consider the associated global height $h_p^{\chi}$ on $E^d(L)$ with local heights $\lambda_v^{\chi}$ with respect to the model $\mathcal{E}^d$. Since $Q$ is a torsion point, we must have
\begin{equation*}
h_p^{\chi}(Q) = 0.
\end{equation*}
It follows from (\ref{eq:QandQtauX049}) and the definition of $\chi$ that
\begin{equation*}
\lambda_{\mathfrak{q}_1}^{\chi}(Q) = \lambda_{\overline{\mathfrak{q}_1}}^{\chi}(Q)=\lambda_{\mathfrak{q}_1}(Q) \quad \text{and}\quad \lambda_{\mathfrak{q}}^{\chi}(Q)=0\ \text{for}\ \mathfrak{q}\mid\overline{\mathfrak{p}}.
\end{equation*}
Furthermore, using \ref{prop:vanishes_on_units} and the fact that $\chi$ is trivial on $L^{\times}$, we find that, for a fixed rational prime $\ell$,
\begin{equation}
\label{eq:character_away_p}
\sum_{v\mid\ell}\chi_{v}(\ell) = -2\log{\ell}.
\end{equation}
Since $P\in \mathcal{X}^d(\mathbb{Z})$, where $\mathcal{X}^d$ is the complement of the origin in $\mathcal{E}^d$, in order to prove the proposition it then suffices to show that
\begin{equation*}
\frac{1}{2}\sum_{v\nmid p}n_{v}\lambda_{v}^{\chi}(Q)= \sum_{\ell\nmid p}\lambda_{\ell}(P),
\end{equation*}
where the left sum runs over primes of $L$ and the right sum over rational primes.

In view of Lemma \ref{lemma:extending_characters} and (\ref{lambdaqlambdaqmin}), we are allowed to perform isomorphisms over extensions of $L$ to calculate local heights. In particular, the change of variables $(x,y)\mapsto (36dx-9d,216d\sqrt{d}y + 108d\sqrt{d}x)$, defined over $\mathbb{Q}(\sqrt{d})$, maps $\mathcal{E}^d_{\mathrm{short}}$ to (\ref{eq:49minimal}), which has discriminant $-7^3$. Under this isomorphism, 
\begin{equation*}
x(P_{\mathrm{short}})\mapsto 2\in\mathbb{Z}\quad \text{and}\quad x(Q_{\mathrm{short}})\mapsto \frac{a-3}{2}\in \OO_K.
\end{equation*}
Therefore, the local heights of $P$ and $Q$ away from $7p$ are trivial when computed with respect to (\ref{eq:49minimal}). Using (\ref{eq:character_away_p}) and letting $d^{\prime} = 7^{-\ordnop_7(d)}d$, we then have
\begin{equation*}
\frac{1}{2}\sum_{v\nmid 7p}n_{v}\lambda_{v}^{\chi}(Q) = -\frac{1}{12[F\colon L]}\sum_{w\nmid 7p}\chi_{w}^{F}(\Delta^{-1})=\begin{cases}
-\log {d^{\prime}}  &\ \text{if}\ d\equiv 1\bmod{4},\\
-\log{4d^{\prime}} & \ \text{otherwise},
\end{cases}
\end{equation*}
where $F=L(\sqrt{d})$, the second sum runs over the primes $w$ of $F$ and the character $\chi^F$ is the idele class character of $F$ obtained from $\chi$ as in Lemma \ref{lemma:extending_characters}.
Similarly, to calculate heights of $P$ away from $7p$ we may base change to $\mathbb{Q}(\sqrt{d})$ and get:
\begin{equation*}
\sum_{\ell\nmid 7p}\lambda_{\ell}(P) =\frac{1}{2}\sum_{u\nmid 7p}n_u\lambda_u(P)= \frac{1}{2}\sum_{u\nmid 7p}\frac{n_u}{6}\log(|\Delta|_u)=\begin{cases}
-\log {d^{\prime}}  &\ \text{if}\ d\equiv 1\bmod{4},\\
-\log{4d^{\prime}} & \ \text{otherwise},
\end{cases}
\end{equation*}
where $u$ runs over primes of $\mathbb{Q}(\sqrt{d})$.

It remains to calculate the local contributions at primes above $7$. For this, it is convenient to work with $\mathcal{E}^d_{\mathrm{short}}$ which is minimal over $\mathbb{Z}_7$. Let $v$ be the unique prime above $7$ in $L$. 
Then $\ordnop_v(\Delta) = 12+24\ordnop_7(d)$. On the other hand, by \cite[VII, Exercise 7.2]{silvermanAEC}, a minimal equation at $v$ has discriminant of valuation at most $11$. Therefore, $E^d$ has good reduction at $v$ and, since we are not in characteristic $2$ or $3$, a minimal equation at $v$ is obtained from $\mathcal{E}^d_{\mathrm{short}}$ via $(x,y)\mapsto (\pi_7^{\ordnop_v(\Delta)/6}x,\pi_7^{\ordnop_v(\Delta)/4}y)$, where $\pi_7$ is a uniformiser at $v$. Under such a change of variables, $P_{\mathrm{short}}$ and $Q_{\mathrm{short}}$ are mapped to $v$-adically integral points. 
Comparing discriminants as above, we then conclude that 
\begin{equation*}
\frac{1}{2}\sum_{v\nmid p}n_{v}\lambda_{v}^{\chi}(Q) = \sum_{\ell\nmid p}\lambda_{\ell}(P) = \begin{cases}
-\frac{1}{2}\log 7-\log d &\ \text{if}\ d\equiv 1\bmod{4},\\
-\frac{1}{2}\log 7-\log 4d &\ \text{otherwise}.
\end{cases}
\end{equation*}
\end{proof}

\subsection{Proof of Theorem \ref{thm:nonequality}}
\label{sec:proof_of_nonequality}
In this subsection we explain how Theorem \ref{thm:nonequality} can be deduced either from Corollary \ref{cor:j0automorphism} or from Proposition \ref{prop:twists_X049}. For an elliptic curve $E$ over $\mathbb{Q}$, denote by $L(E,s)$ its complex $L$-function.
\begin{thm}[\hspace{1sp}{\cite{waldspurger}, \cite{vigneras},\cite[Ch.\thinspace 6, Theorem 1.1]{Murty_non_vanishing}}]
\label{thm:non_vanishing_twist}
Let $E$ be an elliptic curve over $\mathbb{Q}$. There exist infinitely many non-zero square-free integers $d$ such that the quadratic twist $E^d$ of $E$ by $d$ satisfies $L(E^d,1)\neq 0$.
\end{thm}
\begin{thm}[Kolyvagin \cite{kolyvagin}]
\label{thm:kolyvagin}
Let $E$ be an elliptic curve over $\mathbb{Q}$ such that $L(E,1)\neq 0$. Then the rank of $E(\mathbb{Q})$ is zero and the Tate--Shafarevich group of $E/\mathbb{Q}$ is finite.
\end{thm}

It follows from Theorems \ref{thm:non_vanishing_twist} and \ref{thm:kolyvagin} that there are infinitely many twists of $X_0(49)$ satisfying the hypotheses of Proposition \ref{prop:twists_X049}. For each such curve, by Chebotarev's density theorem, there are infinitely many primes for which the proposition holds.

We now see how Corollary \ref{cor:j0automorphism} also provides us with an alternative proof of Theorem \ref{thm:nonequality}. Consider the elliptic curve $E$ with LMFDB label 36.a3 (see \cite[\href{http://www.lmfdb.org/EllipticCurve/Q/36/a/3}{36.a3}]{lmfdb}), which has reduced minimal equation
\begin{equation*}
\mathcal{E}\colon y^2=x^3-27.
\end{equation*}
We have $\mathcal{E}(\mathbb{Z})=\mathcal{E}(\mathbb{Z})[2]=\{O,(3,0)\}$. It follows that every quadratic twist $E^d$ of $E$ by a non-zero square-free integer $d$ satisfies $E^d(\mathbb{Q})[2]\cong \mathbb{Z}/2\mathbb{Z}$.
The equation
\begin{equation*}
\mathcal{E}^d\colon y^2=x^3-27d^3
\end{equation*}
has discriminant equal to $-2^4\cdot 3^9\cdot d^6$ and is hence globally minimal, except if $3\mid d$, in which case we apply $(x,y)\mapsto (9x,27y)$ to obtain a minimal model. Thus, the point of exact order $2$ of $E^d$ defined over $\mathbb{Q}$ is integral. Therefore, by Theorems \ref{thm:non_vanishing_twist} and \ref{thm:kolyvagin}, there exist infinitely many $d$ for which Corollary \ref{cor:j0automorphism} holds with $y_0=0$.

\subsection{A necessary condition: quadratic saturation}
\label{sec:quad_sat}
In \S\S \ref{sec:automorphisms}, \ref{sec:other_characters} we proved sufficient conditions for a point in $\mathcal{X}(\Z_p)$ to belong to $\mathcal{X}(\Z_p)_2$. We now prove the necessary condition given by Theorem \ref{thm:extrapointsintro}, which we restate here for the reader's convenience.
\begin{thm}
\label{thm:extrapoints} 
Let $E/\Q$ and $p$ be as in Theorem \ref{level2rank0}. Suppose that $z\in \mathcal{X}(\mathbb{Z}_p)_2\setminus \mathcal{X}(\mathbb{Z})$. Then $z$ is the localisation of a torsion point $P$ over a number field $K$ and, for each rational prime $q$, the value $\lambda_{\mathfrak{q}}(P)$ of the local height $\lambda_{\fq}$ is independent of the prime $\mathfrak{q}\mid q$ of $K$.
\end{thm}
\begin{proof}
As we observed at the beginning of this section, a point $z\in \mathcal{X}(\Z_p)_{2}$ is necessarily the $p$-adic localisation of a point $P\in \mathcal{E}(\overline{\Z})_{\tors}$. Let $K$ be the minimal number field over which the coordinates of $P$ are defined.
Let $m\in\mathbb{Z}$, $m\neq 0$, such that $mP=O$. Since $z\in \mathcal{X}(\Z_p)$, there exists an embedding $\psi$ of $K$ into $\Q_p$ under which $P$ is mapped to $z$, i.e.\ a prime $\fp_0$ of $K$ such that $\fp_0\mid p$ and
\begin{equation*}
\lambda_{\fp_0}(P) = \lambda_{p}(z) = -||w|| \equalscolon -\sum_{q\in S}\alpha_q \log q, \quad \text{for some } w\in W, \alpha_q\in \Q.
\end{equation*}
For any prime $\fp\mid p$ of $K$, the value $\lambda_{\fp}(P)$ can be computed as follows. Let $x(t),y(t)\in K[[t]]$ be coordinates around $P$, i.e.\ $P = (x(0),y(0))$, the power series $x(t),y(t)$ converge in the intersection over $\fp\mid p$ of small enough $\fp$-adic neighbourhoods of $P$ and $t$ vanishes to order $1$ at $P$. Let $Q(t) = m(x(t),y(t))\in E(K((t)))$. Since $K[[t]]$ is  a complete DVR with residue field $K$, by \cite[Proposition 1]{wuthrichheights} the $t$-adic valuation of $-x(Q(t))/y(Q(t))$ equals the one of $f_m(x(t),y(t))$. More precisely, since $f_m$ vanishes to order $1$ at every point of order dividing $m$, we have
\begin{align*}
-x(Q(t))/y(Q(t)) = at + O(t^2), \qquad \text{some}\ a\in K^{\times}\\
f_m(x(t),y(t)) = ct+ O(t^2), \qquad \text{some}\ c\in K^{\times}.
\end{align*}
Since $\sigma_p^{(\gamma)}(T) = T+O(T^2)$, by \S \ref{sec:heightsabovep}\thinspace\ref{prop:transformationunderm} we then have
\begin{align*}
\lambda_{\fp}(P) &= \lim_{t\to 0} {-\frac{2}{n_{\fp}m^2}\tr_{K_{\fp}/\Q_p}\left(\log_{\fp}\left(\frac{\sigma_p^{(\gamma)}(-x(Q(t))/y(Q(t)))}{f_m(x(t),y(t))}\right)\right)}\\
& = -\frac{2}{n_{\fp}m^2}\tr_{K_{\fp}/\Q_p}\log_{\fp}\left(\frac{a}{c}\right),
\end{align*}
where $\log_{\fp}$ is an extension of $\log$ to $K_{\fp}^{\times}$.

In particular, if $d$ is the least common multiple of the denominators of the $\alpha_q$, then
\begin{equation*}
\log\left(\psi\left(\frac{a}{c}\right)^{2d}\right) = \log \biggl(\prod_{q\in S} q^{d\alpha_q m^2}\biggr).
\end{equation*}
Since $\fp$ is a prime of good reduction, we also have $\ordnop_{\fp}(a/c) = 0$ (strictly speaking we could also avoid using this fact, since the branch of the logarithm corresponding to the cyclotomic character vanishes at $p$), so 
\begin{equation*}
\left(\frac{a}{c}\right)^{2d} = \zeta  \prod_{q\in S} q^{d\alpha_q m^2}
\end{equation*}
for some root of unity $\zeta\in K$. Thus
\begin{equation*}
\lambda_{\fp}(P)  = -\frac{1}{dn_{\fp}m^2}\tr_{K_{\fp}/\Q_p}\log_{\fp}\biggl( \zeta  \prod_{q\in S} q^{d\alpha_q m^2}\biggr) =  -\sum_{q\in S}\alpha_q \log q
\end{equation*}
is independent of $\fp$.

Let now $\fq$ be a prime not above $p$. By \S\ref{sec:heightsawayfromp}\thinspace\ref{property:limit},\ref{property:quasiquadratic}, we have
\begin{equation*}
\lim_{R\to P}\frac{1}{m^2}(\lambda_{\fq}(mR)+2\log|f_m(R)|_{\fq}) = \lim_{R\to P} \lambda_{\fq}(R) = \lambda_{\fq}(P).
\end{equation*}

Since $mR$ is in the formal group at $\fq$, then
\begin{align*}
\lambda_{\fq}(P) &= \lim_{R\to P}\frac{1}{m^2}\left(\log|x(mR)|_{\fq}|f_m(R)|_{\fq}^2\right)\\
& = \lim_{R\to P}\frac{1}{m^2}\log\left(\left|\frac{x(mR)}{y(mR)}\right|^{-2}_{\fq}|f_m(R)|_{\fq}^2\right)\\
&= \frac{1}{m^2}\log\left(\left\vert\frac{c}{a}\right\vert_{\fq}^2\right)= \frac{1}{d}\log\left(\big|\prod_{q\in S} q^{-d\alpha_q}\big|_{\fq}\right),
\end{align*}
which completes the proof.
\end{proof}
\begin{cor}
If $z\in \mathcal{X}(\mathbb{Z}_p)_2\setminus \mathcal{X}(\mathbb{Z})$ is the localisation of a point $P$ defined over $K$, we have $P\in \mathcal{X}(\OO_K)$.
\end{cor}
\begin{proof}
The proof is similar to that of Proposition \ref{prop:integralequalrational}. Note that a torsion point defined over an arbitrary number field can fail to be integral at $\mathfrak{q}$ only if its order is $q^n$ for some $n$ (where $q$ is the norm of $\mathfrak{q}$) and $\ordnop_{\mathfrak{q}}(q)\geq q^n-q^{n-1}$ (cf.\ \cite[VIII, Theorem 7.1]{silvermanAEC}).
\end{proof}
Theorem \ref{thm:extrapoints} is in some sense a natural analogue of a conjecture of Stoll for the classical abelian Chabauty method \cite[Conjecture 9.5]{stoll:findesc8}, which appears in an unpublished draft of \cite{StollFiniteDescent}. Let us restrict to the case when $C$ is a hyperelliptic curve over $\mathbb{Q}$ of genus $g$, whose Jacobian $J$ has rank $g-2$ over $\mathbb{Q}$ (for the conjecture in its full generality see \cite{stoll:findesc8}). Suppose that $\iota\colon C\hookrightarrow J$ is an embedding such that $\iota(C)$ generates $J$ and that $J$ is simple. Stoll's conjecture predicts the existence of a finite subscheme $Z\subset J$ and a set $R$ of primes which has density $1$ in the set of all primes such that, for each $\ell \in R$, we have $\overline{J(\mathbb{Q})}\cap \iota(C(\mathbb{Q}_{\ell}))\subset Z(\mathbb{Q}_{\ell})$, where $\overline{J(\mathbb{Q})}$ is the $\ell$-adic closure of $J(\mathbb{Q})$ in $J(\mathbb{Q}_{\ell})$.

In \cite{WIN4} some evidence for Stoll's conjecture was collected when $g=3$, with the scheme $Z$ being the intersection of $\iota(C)$ with the saturation of $J(\mathbb{Q})$, that is
\begin{equation*}
Z = \{\iota(P)\in J: n\iota(P)\in J(\mathbb{Q}) \text{ for some }n\in\mathbb{Z}_{\geq 1} \}.
\end{equation*}

In our setting of an elliptic curve of rank $0$, it is clear that $\mathcal{X}(\mathbb{Z}_p)_2$ should be contained in the saturation of the Mordell--Weil group $E(\mathbb{Q})$, since $E(\mathbb{Q})=E(\mathbb{Q})_{\mathrm{tors}}$ and the equation $\Log z =0 $ cuts out the torsion points in $E(\mathbb{Q}_p)$. However, this is not a very strong requirement in the elliptic curve case, because the curve and the Jacobian are identified.

Theorem \ref{thm:extrapoints} asserts that the extra constraint $\lambda_p(z)=-||w||$, for some $w\in W$, leads to another type of ``saturation'', in the sense that we have to consider those points for which the local heights behave as if the point were defined over $\mathbb{Q}$.

Note that Stoll's conjecture assumes that $g-r+l\geq 3$, where $l=1$ is the level, in the sense that the Chabauty--Coleman method computes the cohomologically global points of level $1$. In the situation discussed here we have $g=1$, $r=0$, $l=2$, i.e.\ $g-r+l=3$, so one could naively hope that for rank $1$ similar conjectures could be formulated at level $3$.
\section{Algorithm and computations in rank $0$}
\label{sec:Computations}
\subsection{The algorithm}
\label{sec:algorithmelliptic}
We wish to explicitly compute the sets $\phi(w)$ and $\psi(w)$ from Theorems \ref{level2rank0} and \ref{level2rank1}. Each of $D_2(z)$ and $\Log(z)$ are locally analytic functions on $\mathcal{X}(\mathbb{Z}_p)$. In other words, given a point $\overline{P}\in\overline{E}(\mathbb{F}_p)\setminus\{O\}$ and a fixed point $P\in \mathcal{X}(\mathbb{Z}_p)$ reducing to $\overline{P}$ modulo $p$, one can pick a uniformiser $t\in \mathbb{Q}_p(E)$ at $P$, which reduces to a uniformiser at $\overline{P}$. Then, for each $Q\in \mathcal{X}(\mathbb{Z}_p)$ in the residue disk of $P$, we have
\begin{equation*}
\Log(Q) = f_{P}(t(Q))\quad  \text{and}\quad  D_2(Q) = g_P(t(Q))
\end{equation*}
for some $f_P(x),g_P(x)\in\mathbb{Q}_p[[x]]$ convergent at all $x\in\mathbb{Z}_p$ with $|x|_p<1$.

On the other hand, let $\gamma=C$ if $p$ is of good ordinary reduction and $\gamma=\frac{a_1^2+4a_2}{12}$ if $p$ is of good supersingular reduction. By Proposition \ref{sigmadilogarithm} and \S\ref{sec:heightsabovep}\thinspace\ref{prop:transformationunderm}, provided that $mQ\neq O$, where $m=\#\overline{E}(\mathbb{F}_p)$, then
\begin{equation}
\label{rearrangingsigma}
2D_2(Q) + \gamma\Log(Q)^2= -\frac{2}{m^2}\log\left(\frac{\sigma_p^{(\gamma)}(mQ)}{f_m(Q)}\right).
\end{equation}
Since there are finitely many\footnote{In fact, at most one.} points in each residue disk satisfying $mQ=O$, the local expansion of the right hand side of (\ref{rearrangingsigma}) in terms of the local parameter $t$ holds in the whole residue disk. In fact, the local expansion of $\sigma
_p^{(\gamma)}(mQ)$ and $f_m(Q)$ have precisely the same zeros with the same multiplicity $1$ and two $p$-adic power series which agree at infinitely many points in $\Z_p$ of absolute value less than $1$ are equal by the $p$-adic Weierstrass Preparation Theorem \cite[Ch.\ IV, \S 4, Theorem 14]{Koblitz}. Note that we already used this in the proof of Theorem \ref{thm:extrapoints}.
 
By the same observation as in Remark \ref{rmk:globaltorsion}, we obtain a way of computing the intersections of $\phi(w)$ and $\psi(w)$ with each residue disk using local expansions of the $p$-adic sigma function (of Mazur--Tate or Bernardi) and the $m$-th division polynomial\footnote{Computationally, it is more convenient to take $m$ to be the order of $\overline{P}$ in $\overline{E}(\mathbb{F}_p)$, i.e.\ to choose potentially different values of $m$ for different residue disks.}, in place of the double Coleman integral $D_2(z)$.

The function $\Log\colon \mathcal{X}(\mathbb{Z}_p)\to \mathbb{Q}_p$ is odd; the function $\lambda_p\colon \mathcal{X}(\mathbb{Z}_p)\to \mathbb{Q}_p$ is even (cf.\ property \ref{prop:transformationunderauto} in \S \ref{sec:heightsabovep}). Therefore, 
\begin{equation*}
z\in \phi(w) \iff -z\in\phi(w)
\end{equation*}
and it will thus suffice to consider residue disks up to $\overline{P}\mapsto -\overline{P}$. The same holds for $\psi(w)$.

We also notice that different models can be used for computing the $p$-adic heights and the single Coleman integrals. In fact, we defined local $p$-adic heights using an integral minimal model and, for instance, there is an implementation for the Mazur--Tate $p$-adic sigma function in \texttt{SageMath} due to Harvey (see \cite{harvey} and also \cite{MST}). On the other hand, for Coleman computations on \texttt{SageMath} (see \cite{ExplicitColeman}), one requires the elliptic curve to be described by a Weierstrass model whose $a_1$ and $a_3$ coefficients are zero and there is no requirement on minimality; the only requirement on integrality is $\mathbb{Z}_p$-integrality. To avoid explicit Coleman integration computations, we could also work directly with the formal logarithm.

\subsection{Examples for Section \ref{sec:obstructions}}
In the examples that follow, as well as in the ones of the next sections, we avoid making distinctions between the curve $E/\mathbb{Q}$ and the model $\mathcal{E}/\mathbb{Z}$. The Weierstrass equations that we work with are always minimal and reduced, unless stated otherwise. 
\begin{example}[Corollary \ref{cor:semistablecaseautomorphism}, Proposition \ref{prop:anticyclotomiccyclotomic}]
\label{Example:1}
Consider the rank $0$ elliptic curve \href{http://www.lmfdb.org/EllipticCurve/Q/17/a/1}{17.a1}\cite{lmfdb}
\begin{equation}
\label{weierstrasseq:example1}
E\colon y^2 + x y + y = x^{3} -  x^{2} - 91 x - 310
\end{equation}
and the prime $p=5$ of good ordinary reduction.

Since none of the conditions of Theorem \ref{level2rank0}\thinspace (\ref{trivial}) are satisfied, we need to explicitly compute $\mathcal{X}(\mathbb{Z}_p)_2$ as a union of $\phi(w)$.
The curve has split multiplicative reduction at $17$ with Kodaira symbol $\mathrm{I}_1$ and good reduction everywhere else: thus, $W=\{0\}$. We find
\begin{equation*}
\mathcal{X}(\mathbb{Z}_p)_2=\{ (-5,2\pm \rho(i))\},
\end{equation*}
where $\rho\colon \mathbb{Z}[i]\hookrightarrow \mathbb{Z}_p$ is a fixed embedding. If a priori our computations only return approximations of $p$-adic points, by Corollary \ref{cor:semistablecaseautomorphism} the $p$-adic points found are the localisations of the points over $\mathbb{Z}[i]$ listed above. Let us nevertheless explain it in detail for this example.

The Weierstrass equation (\ref{weierstrasseq:example1}) defines a global minimal model also for the base-change $E/\mathbb{Q}(i)$ and the prime $p$ splits in $K=\mathbb{Q}(i)$. We extend $\rho$ to a map $E(K)\hookrightarrow E(\mathbb{Q}_p)$. The point $$Q = (-5,2+ i)\in E(K)$$ is integral with respect to the global minimal model above and satisfies $$4Q = O.$$  
Thus, since the reduction types at the bad primes of $E/\mathbb{Q}(i)$ are the same as over $\mathbb{Q}$, we have
\begin{equation*}
0 = h_p(Q) = \frac{1}{2}(\lambda_{\mathfrak{p}_1}(Q)+\lambda_{\mathfrak{p}_2}(Q)),
\end{equation*}
where $p\mathbb{Z}[i]=\mathfrak{p}_1\mathfrak{p}_2$ and explicitly (without loss of generality)
\begin{equation*}
\lambda_{\mathfrak{p}_1}(Q)=\lambda_p(\rho(Q)),\qquad \lambda_{\mathfrak{p}_2}(Q)=\lambda_p(\rho(\tau(Q))),
\end{equation*}
where $\mathrm{Gal}(K/\mathbb{Q})=\braket{\tau}$, so $\tau(Q)=(-5,2-i)$. On the other hand, $\tau(Q)=-Q$ and $\lambda_p$ is an even function. Therefore 
\begin{equation*}
0 = \lambda_p(\rho(Q))=\lambda_p(\rho(\tau(Q))).
\end{equation*}
Note that Proposition \ref{prop:anticyclotomiccyclotomic} would also explain why the $p$-adic  localisation of the point $Q$ belongs to $\mathcal{X}(\mathbb{Z}_p)_2$.
\end{example}

\begin{example}[Proposition \ref{prop:anticyclotomiccyclotomic}, Proposition \ref{automorphismeffect}\thinspace (\ref{firststatement})]
\label{Example:2}
Consider the elliptic curve \href{http://www.lmfdb.org/EllipticCurve/Q/121/d/3}{121.d3}\cite{lmfdb}
\begin{equation}
\label{weierstrasseq:example2}
E\colon y^2 + y = x^{3} -  x^{2} - 40 x - 221
\end{equation}
and the prime $p=5$, at which $E$ has good ordinary reduction. We have
\begin{equation*}
\#\overline{E}(\mathbb{F}_2)=1
\end{equation*}
and thus $$\mathcal{X}(\mathbb{Z}_p)_2=\emptyset$$ by Theorem \ref{level2rank0}\thinspace (\ref{trivial}). On the other hand, we can still compute $\bigcup_{w\in W}\phi(w)$ of Theorem \ref{level2rank1}\thinspace (\ref{nontrivial}). The curve has additive reduction of type $\mathrm{I}_1^*$ at $11$ with Tamagawa number $2$ and has good reduction everywhere else. Therefore,
\begin{equation*}
W =\{0,-\log 11\}.
\end{equation*}
We find that $\phi(0) =\emptyset$, but
\begin{equation*}
\phi(-\log 11)=\left\{\bigg(-7,\rho\Big(\frac{-1\pm 11\sqrt{-11}}{2}\Big)\bigg),\left(4,\rho\Big(\frac{-1\pm 11\sqrt{-11}}{2}\Big)\right)\right\},
\end{equation*}
where $\rho\colon \mathbb{Z}[(1+\sqrt{-11})/2]\hookrightarrow \mathbb{Z}_p$ is a fixed embedding.

Let $K=\mathbb{Q}(\sqrt{-11})$. The prime $11$ ramifies in $K/\mathbb{Q}$ and $E/K$ has split multiplicative reduction of type $\mathrm{I}_2$ at $\mathfrak{q}$, where $\mathfrak{q}^2=11$.
Let $$Q\in\left\{ \left(-7,\frac{-1\pm 11\sqrt{-11}}{2}\right),\left(4,\frac{-1\pm 11\sqrt{-11}}{2}\right)\right\}\subset E(K).$$

The point $Q$ has order $5$. Unlike in Example \ref{Example:1}, the Weierstrass equation (\ref{weierstrasseq:example2}) is minimal at all primes except at $\mathfrak{q}$ and hence we cannot use straightforwardly the explicit formulae for the local height at $\mathfrak{q}$ given in \S \ref{sec:heightsawayfromp}.
The curve $E/K$ admits the global minimal model
\begin{equation*}
E^{\min{}}\colon y^2 + y = x^{3} -  x^{2}
\end{equation*}
and the image of $Q$ in $E^{\min}(K)$ has good reduction at $\mathfrak{q}$, so that $\lambda_{\mathfrak{q}}^{\min}(Q)=0$.
Equation (\ref{lambdaqlambdaqmin}) then yields $\lambda_{\mathfrak{q}}(Q)=-\log 11$.
Therefore, by Proposition \ref{prop:anticyclotomiccyclotomic}, we have
\begin{align*}
0 = \lambda_p(\rho(Q))+\lambda_{\mathfrak{q}}(Q)=\lambda_p(\rho(Q))-\log 11.
\end{align*}
Similarly to Example \ref{Example:1}, the appearance of $\rho(Q)$ in $\phi(-\log{11})$ is also justified by Proposition \ref{automorphismeffect}\thinspace (\ref{firststatement}).
\end{example}

\begin{example}[Remark \ref{rmk:j1728}]
\label{example:j1728}
Consider the elliptic curve \href{http://www.lmfdb.org/EllipticCurve/Q/14112/q/1}{14112.q1}\cite{lmfdb} 
\begin{equation}
\label{weierstrasseq:example4}
E\colon y^2 = x^{3} - 9261 x
\end{equation}
and the prime $p=5$, which is of good ordinary reduction. Note that $p$ splits in $K=\mathbb{Q}(\sqrt{21})$ and by Remark \ref{rmk:j1728}, the localisations of the point $Q^{\pm}=(\pm 21\sqrt{21},0)$ belong to $\mathcal{X}(\mathbb{Z}_p)_2$ provided that $(1/2)$ times the sum of its local heights at bad primes is in $||W||$. Both at $3$ and $7$, the curve has bad reduction of additive type $\mathrm{III}^*$ with Tamagawa number $2$; at $2$ the curve has reduction of type $\mathrm{III}$ with Tamagawa number $2$. Thus, $W=W_2\times W_{3}\times W_7$, with $W_q = \{0,-\frac{3}{2}\log q\}$ for each $q\in\{3,7\}$ and $W_2=\{0,-\frac{1}{2}\log 2\}$.

A global minimal model for the base-change of $E$ to $K$ is given by $y^2=x^3-21x$. Furthermore, $2$ is inert in $K$ and its reduction type does not change. The primes $3$ and $7$ become of type $\mathrm{I}_0^*$ with Tamagawa number $4$. By Proposition \ref{localheightsnontrivialtamagawa} and Proposition \ref{automorphismeffect}\thinspace(\ref{firststatement}) (see also Remark \ref{rmk:j1728}), we then find that $Q^{\pm}$ is indeed in $\mathcal{X}(\mathbb{Z}_p)_2$.

Our computation of $\mathcal{X}(\mathbb{Z}_p)_2$ recovers precisely the integral points and the ones coming from $Q^{\pm}$.
\end{example}

\begin{example}[Proposition \ref{automorphismeffect}\thinspace (\ref{firststatement})]
\label{Example:3}
Consider the elliptic curve \href{http://www.lmfdb.org/EllipticCurve/Q/11025/y/2}{11025.y2}\cite{lmfdb}, 
whose reduced minimal model is
\begin{equation*}
E\colon y^2+y = x^3+15006
\end{equation*}
and let $p=13$, which is the smallest prime of good ordinary reduction for $E$. Note that $E$ has vanishing $j$-invariant.
We find that
\begin{equation}
\label{XZp211025j1}
\mathcal{X}(\mathbb{Z}_p)_2 = \{\pm(0,122)\}\cup\{\pm(\zeta_3^{i}\sqrt[3]{120050},367): 0\leq i\leq 2\},
\end{equation}
where $\zeta_3$ is a primitive third root of unity and we assume that we have fixed an embedding of $\mathbb{Q}(\zeta_3,\sqrt[3]{120050})$ into $\mathbb{Q}_p$. As usual, the equality (\ref{XZp211025j1}) is deduced from computations combined with theoretical results. In this particular example, the theory needed is that the Galois group of $\mathbb{Q}(\zeta_3,\sqrt[3]{120050})/\mathbb{Q}$ acts on the the order-$6$ points $\pm(\zeta_3^{i}\sqrt[3]{120050},367)$ by automorphisms. We leave the reader to check that these points also have the right local heights at bad primes.
\end{example}

\begin{example}[Corollary \ref{cor:j0automorphism}, Corollary \ref{cor:biquadraticpoints}]
\label{Example:biquadratic}
Consider the elliptic curve \href{http://www.lmfdb.org/EllipticCurve/Q/900/g/3}{900.g3}\cite{lmfdb} 
with reduced minimal model
\begin{equation*}
E\colon y^2 = x^{3} - 3375
\end{equation*}
and the prime $19$, which is of good ordinary reduction for $E$. We have
\begin{equation*}
\mathcal{X}(\mathbb{Z}_{19})_2= \{(\zeta_3^i 15,0),(-\zeta_3^i 30, \pm\sqrt{-30375}):0\leq i\leq 2\}
\end{equation*}
where $\zeta_3$ is a primitive third root of unity and we assume that we have fixed an embedding of $\mathbb{Q}(\zeta_3,\sqrt{-30375})$ into $\mathbb{Q}_{19}$.
\end{example}

\subsection{Large-scale data}
\label{sunsec:largescaledata}
Using the database \cite{lmfdb}, we could run the code on all the $86,213$ elliptic curves over $\mathbb{Q}$ of rank $0$ and conductor less than or equal to $30,000$; for each curve we let $p$ be the smallest prime $\geq 5$ of good ordinary reduction\footnote{We could have allowed $p$ to equal $3$ and used the method of \cite{3adicheights} to compute the quantities involved in the $3$-adic heights. Some computations with supersingular primes were carried out for \S \ref{sec:variationprime}.}.

Out of these, we found exactly $470$ pairs $(E,p)$ for which $\mathcal{X}(\mathbb{Z}_p)_2\supsetneq \mathcal{X}(\mathbb{Z})$. The 10 such pairs with $E$ of conductor $\leq  100$ are listed in \textsc{Table} \ref{table:2}.

\begin{table}
\makebox[\textwidth][c]{\begin{tabular}{ |c|c|c|c|c|c|c| } 
\hline
LMFDB & $p$ & SS/CM & $\mathcal{X}(\mathbb{Z}_p)_2\setminus \mathcal{X}(\mathbb{Z})$ & Order & Explanation  \\
\hline
\href{http://www.lmfdb.org/EllipticCurve/Q/17/a/1}{17.a1} & $5$ & SS & $(-5,2\pm i)$ & $4$ & Cor.\thinspace \ref{cor:semistablecaseautomorphism}/Prop.\thinspace \ref{prop:anticyclotomiccyclotomic} \\
\hline
\href{http://www.lmfdb.org/EllipticCurve/Q/27/a/3}{27.a3} & $7$ & $j=0$ & $\pm\left(\frac{-3\pm 3\sqrt{-3}}{2},4\right)$ & $3$ & Cor.\thinspace \ref{cor:j0automorphism}/Prop.\thinspace \ref{prop:anticyclotomiccyclotomic}\\
\hline
\href{http://www.lmfdb.org/EllipticCurve/Q/32/a/4}{32.a4} & $5$ & $j=1728$ & $(-2,\pm 4i)$ & $4$ & Prop.\thinspace \ref{automorphismeffect}\thinspace (\ref{firststatement})/Prop.\thinspace \ref{prop:anticyclotomiccyclotomic}\\
\hline
\href{http://www.lmfdb.org/EllipticCurve/Q/36/a/3}{36.a3} & $7$ &  $j=0$ & $\left(\frac{-3\pm 3\sqrt{-3}}{2},0\right)$ & $2$ & Cor.\thinspace \ref{cor:j0automorphism}/Prop.\thinspace \ref{prop:anticyclotomiccyclotomic}\\
\hline
\href{http://www.lmfdb.org/EllipticCurve/Q/36/a/4}{36.a4} & $7$ & $j=0$ &  $\left(\frac{1\pm \sqrt{-3}}{2},0\right)$ & $2$  & Cor.\thinspace \ref{cor:j0automorphism}/Prop.\thinspace \ref{prop:anticyclotomiccyclotomic}\\
& & & $\pm (-1\pm \sqrt{-3},3)$ & $6$ & Cor.\thinspace \ref{cor:j0automorphism}/Prop.\thinspace \ref{prop:anticyclotomiccyclotomic}\\
\hline
\href{http://www.lmfdb.org/EllipticCurve/Q/49/a/2}{49.a2}  & $11$ & $j=-3375$ & $\pm\left(\frac{-7b^2+25}{2}, \frac{-49b^3 + 7b^2 - 49b - 25}{4}\right)$
& $4$ & Prop.\thinspace \ref{prop:twists_X049}\\
\hline 
\href{http://www.lmfdb.org/EllipticCurve/Q/49/a/4}{49.a4} & $11$ & $j=-3375$ & $\pm\left(\frac{b^2 - 3}{2}, \frac{b^3 - b^2 - 7b + 3}{4}\right)$
& $4$ &  Prop.\thinspace \ref{prop:twists_X049}\\
\hline
\href{http://www.lmfdb.org/EllipticCurve/Q/75/b/4}{75.b4} & $11$ &-  & $(27,-14\pm 5\sqrt{5})$ & $4$ & Prop.\thinspace \ref{automorphismeffect}\thinspace (\ref{firststatement})\\
 \hline
 \href{http://www.lmfdb.org/EllipticCurve/Q/75/b/6}{75.b6} & $11$ &- & $\left(12,\frac{-13\pm 25\sqrt{5}}{2}\right)$ & $4$ & Prop.\thinspace \ref{automorphismeffect}\thinspace (\ref{firststatement})\\
& & & $\left(2, -\frac{3}{2}(1\pm 5\sqrt{5})\right)$ &$4$ & Prop.\thinspace \ref{automorphismeffect}\thinspace (\ref{firststatement})\\ 
 \hline
\href{http://www.lmfdb.org/EllipticCurve/Q/75/b/7}{75.b7} & $11$ &- & $\left(2,\frac{1}{2}(-3\pm 5\sqrt{5})\right)$ & $4$   & Prop.\thinspace \ref{automorphismeffect}\thinspace (\ref{firststatement})\\
 \hline
\end{tabular}}\\
\caption{All curves of rank $0$ and conductor $\leq 100$ for which $\mathcal{X}(\mathbb{Z}_p)_2\supsetneq \mathcal{X}(\mathbb{Z})$ ($p\geq 5$ smallest good ordinary prime); $b$ satisfies $x^4+7=0$. The curve is given in the first column as an LMFDB label \cite{lmfdb}. In the third column, SS means `semistable' and `-' neither semistable nor CM.}
\label{table:2}\end{table}

We summarise the results of the computations in Propositions \ref{prop:datastandardcurves}, \ref{prop:dataoverquadraticfields}, \ref{j-3375}. 
\begin{prop}
\label{prop:datastandardcurves}
Let $E$ be an elliptic curve of rank $0$ and conductor less than or equal to $30,000$ and let $p\geq 5$ be the smallest prime of good ordinary reduction. Assume that $j(E)\not\in\{0,1728,-3375\}$. Then $\mathcal{X}(\mathbb{Z}_p)_2\setminus \mathcal{X}(\mathbb{Z})$ is either empty or consists of localisations of points defined over the ring of integers of a quadratic field $K$ on which the Galois group of $K/\mathbb{Q}$ acts as multiplication by $\pm 1$. 
\end{prop}

\begin{prop}
\label{prop:dataoverquadraticfields}
Let $E$ be an elliptic curve of rank $0$ and conductor less than or equal to $30,000$ and let $p\geq 5$ be the smallest prime of good ordinary reduction. If $\mathcal{X}(\mathbb{Z}_p)_2\setminus \mathcal{X}(\mathbb{Z})$ contains localisations of points defined over a quadratic field $K$, these points satisfy the hypotheses of Proposition \ref{automorphismeffect}, i.e.\ the non-trivial element of $\mathrm{Gal}(K/\mathbb{Q})$ acts on them in the same way as an automorphism of $E$.
\end{prop}

\begin{rmk}
In view of the large data collected, it might have been tempting to expect that Propositions \ref{prop:datastandardcurves} and \ref{prop:dataoverquadraticfields} would be true for arbitrary prime and conductor. It turns out that this is not the case: varying the prime for the curve 8712.u5 (which does not have CM), we found some quadratic points on which Galois does not act by automorphisms (cf.\ Example \ref{Example:mysterious}). We note nevertheless that in the latter case the extra points were explained by Proposition \ref{prop:anticyclotomiccyclotomic}.
\end{rmk}

We now turn to the extra points in our data defined over number fields of degree\footnote{Note that by degree we mean the degree of the smallest number field over which a point in $\mathcal{X}(\mathbb{Z}_p)_2$ is defined and not the degree of the number field containing all the coordinates of the points in $\mathcal{X}(\mathbb{Z}_p)_2$, which could be larger.} at least equal to $3$. By Proposition \ref{prop:datastandardcurves}, these can only show up if $E$ has complex multiplication and, in fact, its $j$-invariant is one of $0,1728,-3375$. 
There was only one curve beside the one of Example \ref{Example:3} where cubic points were recovered, namely the curve \href{http://www.lmfdb.org/EllipticCurve/Q/19881/g/2}{19881.g2}\cite{lmfdb}.  
As in Example \ref{Example:3}, the curve has $j$-invariant equal to zero and the appearance of these points in $\mathcal{X}(\mathbb{Z}_p)_2$ is explained by Proposition \ref{automorphismeffect}.

Finally, we recovered points defined over number fields of degree $4$ on the curve \href{http://www.lmfdb.org/EllipticCurve/Q/14112/q/2}{14112.q2} ($j=1728$)\cite{lmfdb}, 
which is explained by Proposition \ref{automorphismeffect}, and on all the twists of the modular curve $X_0(49)$, as predicted by Proposition \ref{prop:twists_X049}. In fact the inclusion in the statement of Proposition \ref{prop:twists_X049} is an equality in all the following cases.
\begin{prop}
\label{j-3375}
Let $E$ be an elliptic curve of rank $0$, conductor less than or equal to $30,000$ and $j(E)=-3375$. Then $\mathcal{X}(\mathbb{Z}_{11})_2=\mathcal{X}(\mathbb{Z})\cup\{\pm Q\}$, where $Q$ has order $4$ and comes from a point over the ring of integers of the smallest number field $L$ over which $E(L)[4]\cong \mathbb{Z}/2\mathbb{Z}\times\mathbb{Z}/4\mathbb{Z}$. 
\end{prop}

\subsection{Variation of the prime}
\label{sec:variationprime}
In what follows, we assume that the rank of the elliptic curve is equal to zero. However, the discussion could easily go through word for word with $\mathcal{X}(\mathbb{Z}_p)_{2,\mathrm{tors}}^{\prime}$ in place of $\mathcal{X}(\mathbb{Z}_p)_2$. 

We have established that there exist curves of rank $0$ and primes $p$ for which $\mathcal{X}(\mathbb{Z}_p)_2\supsetneq\mathcal{X}(\mathbb{Z})$. The next question we ask is whether there always exists a (not necessarily ordinary) prime for which the cohomologically global points of level $2$ are precisely the global integral points. It turns out that the answer is negative, as we will see in Example \ref{Example:mysterious} below.

Let us first gather some intuition on what is happening. Recall that, after having fixed all appropriate embeddings, $$\phi(w)\subset\mathcal{E}(\overline{\mathbb{Z}})_{\mathrm{tors}}=E(\overline{\mathbb{Q}})_{\mathrm{tors}}.$$
Therefore, if $P\in E(F)_{\mathrm{tors}}$ for some minimal number field $F$, by picking $p$ such that $[F_v:\mathbb{Q}_p]>1$ for all $v\mid p$, we can guarantee that $P\not \in \mathcal{X}(\mathbb{Z}_p)_{2}$. For instance in Example \ref{Example:1}, if we pick $p'=7$, which is of good ordinary reduction and which is inert in $\mathbb{Q}(i)$, we find $\mathcal{X}(\mathbb{Z}_{p'})_2=\mathcal{X}(\mathbb{Z})=\emptyset$.

Note that, in view of Corollary \ref{cor:j0automorphism} and \S \ref{sec:proof_of_nonequality}, there exist (infinitely many) curves for which $\mathcal{X}(\mathbb{Z}_p)_2$ is strictly larger than $\mathcal{X}(\mathbb{Z})$ for all odd primes $p$ of good \emph{ordinary} reduction. More generally, if $E$ has complex multiplication by the quadratic field $K$ and there exist points defined over $K$ and satisfying the assumptions of Proposition \ref{automorphismeffect}\thinspace (\ref{firststatement}), then these points will show up in $\mathcal{X}(\mathbb{Z}_p)_2$ for any good ordinary odd prime $p$ by Deuring's criterion. On the other hand, Deuring's criterion also implies that the good supersingular primes cannot split in $K$.

We ran the code on all the $470$ curves of \S \ref{sunsec:largescaledata} for which we had found some extra points: this time, we varied the good ordinary prime until we found a prime for which no extra points showed up or we proved that such prime does not exist. If a good ordinary prime $p$ for which $\mathcal{X}(\mathbb{Z}_p)_2=\mathcal{X}(\mathbb{Z}_p)$ does not exist, we repeated the calculations with supersingular primes.
We summarise the results in the following theorem (which includes also the statement of Theorem \ref{thm:nonequalitystrong}).

\begin{thm}
\label{thm:variationp}
Let $E$ be an elliptic curve over $\mathbb{Q}$ of rank $0$ and conductor less than or equal to $30,000$. Then there exists a good ordinary odd prime $p$ for which $\mathcal{X}(\mathbb{Z}_p)_2=\mathcal{X}(\mathbb{Z})$, unless:
\begin{enumerate}
\item\label{32a1} E is \href{http://www.lmfdb.org/EllipticCurve/Q/32/a/4}{32.a4} ($j=1728$) \cite{lmfdb};
\item\label{j0extra} E is one of the $20$ elliptic curves of rank $0$ with $j=0$ and $\mathbb{Z}/2\mathbb{Z}\subset E(\mathbb{Q})$ or $E$ is \href{http://www.lmfdb.org/EllipticCurve/Q/27/a/3}{27.a3} ($j=0$) \cite{lmfdb};
\item\label{mysterious} $E$ is \href{http://www.lmfdb.org/EllipticCurve/Q/8712/u/5}{8712.u5} \cite{lmfdb}.
\end{enumerate}
Moreover, in cases (\ref{32a1}) and (\ref{j0extra}), $\mathcal{X}(\mathbb{Z}_p)_2\setminus \mathcal{X}(\mathbb{Z})\neq \emptyset$ for all good ordinary odd primes $p$, but there exists a supersingular prime $p$ for which $\mathcal{X}(\mathbb{Z}_p)_2=\mathcal{X}(\mathbb{Z}_p)$; in case (\ref{mysterious}) $\mathcal{X}(\mathbb{Z}_p)_2\setminus \mathcal{X}(\mathbb{Z})\neq \emptyset$ for all good (ordinary and supersingular) primes $p$.
\end{thm}
\begin{proof}
That there exists a good ordinary prime for which Conjecture \ref{cjc:Kim} holds at level $2$ for all curves not in (\ref{32a1}), (\ref{j0extra}) and (\ref{mysterious}) is shown computationally. The assertion of cases (\ref{32a1}) and (\ref{j0extra}) follows from the discussion before the statement of the theorem together with explicit computations of $\mathcal{X}(\mathbb{Z}_p)_2$ at some supersingular primes $p$. Finally, we treat the curve 8712.u5 in detail in Example \ref{Example:mysterious}.
\end{proof}

\begin{example}
\label{Example:mysterious}
Consider the elliptic curve 8712.u5, given by 
\begin{equation}
\label{curve:mysterious}
E\colon y^2 = x^{3} + 726 x + 9317.
\end{equation}
We have $S=\{2,3,11\}$: in particular, the reduction is of type $\mathrm{III}$ with Tamagawa number $2$ at $2$, of type $\mathrm{I}_1^{*}$ with Tamagawa number $4$ at $3$ and of type $\mathrm{I}_0^*$ with Tamagawa number $2$ at $11$. Thus $W = W_2\times W_3\times W_{11}$, where
\begin{equation*}
W_2 = \left\{0,-\frac{1}{2}\log 2 \right\},\quad W_3 = \left\{0,-\log 3, -\frac{5}{4}\log 3\right\},\quad W_{11} = \{0,-\log 11\}.
\end{equation*}
Consider
\begin{equation*}
A \colonequals \left\{(-44,\pm 99\sqrt{-11}),(22, \pm 33\sqrt{33}),\left(\frac{11}{2}(1\pm 3\sqrt{-3}),0\right)\right\}\subset\mathcal{X}(\overline{\mathbb{Z}})_{\tors}.
\end{equation*}
If $p\not\in S$, then $p$ splits in at least one of $\mathbb{Q}(\sqrt{-11})$, $\mathbb{Q}(\sqrt{33})$ and $\mathbb{Q}(\sqrt{-3})$ and therefore $A\cap \mathcal{X}(\mathbb{Z}_p)\neq \emptyset$ (after having fixed embeddings). The fact that $$A\cap \mathcal{X}(\mathbb{Z}_p)_2\neq \emptyset,$$ then follows from Proposition \ref{automorphismeffect}\thinspace(\ref{firststatement}) (for the points over $\mathbb{Q}(\sqrt{-11})$ and $\mathbb{Q}(\sqrt{33})$), Proposition \ref{prop:anticyclotomiccyclotomic} (for the points over $\mathbb{Q}(\sqrt{-3})$) and the following table, which shows how the reduction changes at the primes in $S$. The symbol - means that the reduction type has not changed. In the last column, there are the possible values of $\lambda_{\mathfrak{q}}(\mathcal{X}(\OO_{\mathfrak{q}}))$ where $\OO_{\mathfrak{q}}$ is the ring of integers of the completion $K_{\mathfrak{q}}$ at a prime $\mathfrak{q}$ above $q$ in the field $K$. We briefly explain how the table is computed. When the prime $q$ splits in $K$, there is nothing to show: $\lambda_{\mathfrak{q}}(\mathcal{X}(\OO_{\mathfrak{q}})) = W_{q}$ if $\mathfrak{q}\mid q$. If $q$ is inert, by Tate's algorithm, (\ref{curve:mysterious}) is minimal at $\mathfrak{q}\mid q$ and the Kodaira symbol is unchanged. Once we know the Tamagawa number at $\mathfrak{q}$, we can find $\lambda_{\mathfrak{q}}(\mathcal{X}(\OO_{\mathfrak{q}}))$ directly from Proposition \ref{localheightsnontrivialtamagawa}. Finally, if $q$ ramifies in $K$, then $\lambda_{\mathfrak{q}}(\mathcal{X}(\OO_{\mathfrak{q}}))$ may be deduced from Proposition \ref{localheightsnontrivialtamagawa}, Lemma \ref{lemmalocalheights}\thinspace \ref{goodreduction} and (\ref{lambdaqlambdaqmin}). Note that some points in $\mathcal{X}(\OO_{\mathfrak{q}})$ may map to non-integral points in a minimal model at $\mathfrak{q}$ (see also Lemma \ref{lemma:heightsonintegralnonminimal}).

\begin{table}[h]
\makebox[\textwidth][c]{\begin{tabular}{ |c|c|c|c|c|c| } 
\hline
$K$ & $q$ & Splitting & Reduction (Tamagawa) & $|\Delta/\Delta_{\min}|_{\mathfrak{q}}$ & $\lambda_{\mathfrak{q}}(\mathcal{X}(\OO_{\mathfrak{q}}))$  \\
\hline
$\mathbb{Q}(\sqrt{-11})$ & 2 & inert &- &$1$  & $W_2$\\
& 3& split & - & $1$ & $W_3$\\
& 11 & ramified & good & $11^{-6}$ & $W_{11}$\\
\hline
$\mathbb{Q}(\sqrt{33})$ & $2$ & split & - &$1$ &$W_2$\\
& $3$ & ramified & non-split $\mathrm{I}_2$ (2) & $3^{-6}$ & $W_3$\\
& $11$ & ramified & good & $11^{-6}$ & $W_{11}$\\
\hline
$\mathbb{Q}(\sqrt{-3})$ & $2$ & inert & - & $1$ & $W_2$\\
 & $3$ & ramified & non-split $\mathrm{I}_2$ (2) & $3^{-6}$ & $W_3$\\
 & $11$ & inert & $\mathrm{I}_0^*$ (4) & $1$ & $W_{11}$ \\
 \hline
\end{tabular}}\\
\label{table:Example:mysterious}\end{table}
\end{example}

\section{The rank $1$ case}
\label{sec:rank_1}
\subsection{Algebraic non-rational points in $\mathcal{X}(\mathbb{Z}_p)_2^{\prime}$ in rank 1}
\label{sec:algebraicpointsrank1}
We retain the notation of Theorems \ref{level2rank0} and \ref{level2rank1}. 
The set consisting of the torsion points in $\psi(w)$ is equal to $\phi(w)$. Therefore, the results of Section \ref{sec:obstructions} translate into results for $\mathcal{X}(\mathbb{Z}_p)_{2,\mathrm{tors}}^{\prime}$.

Since each $\psi(w)$ is defined by a single $p$-adic equation, in most cases it is expected that $\mathcal{X}(\mathbb{Z}_p)_2^{\prime}$ should be strictly larger than $\mathcal{X}(\mathbb{Z})$. The question we investigate in this subsection is which algebraic non-torsion points could arise in $\mathcal{X}(\mathbb{Z}_p)_2^{\prime}$. The following elementary lemma shows that if a non-torsion point in $\mathcal{X}(\mathbb{Z}_p)_2^{\prime}$ comes from a quadratic point in the saturation of $E(\mathbb{Q})$, then its belonging to $\mathcal{X}(\mathbb{Z}_p)_2^{\prime}$ cannot be explained by automorphisms (cf.\ \S \ref{sec:automorphisms}).

\begin{lemma}
\label{lemma:saturationautoimpliestorsion}
Let $E$ be an elliptic curve over $\mathbb{Q}$ and $K$ a quadratic field with $\mathrm{Gal}(K/\mathbb{Q})=\braket{\tau}$. Let $P\in E(K)\setminus E(\mathbb{Q})$ such that $mP\in E(\mathbb{Q})$ for some non-zero integer $m$. Then, if $\psi(P) = \tau(P)$ for some $\psi\in \mathrm{Aut}(E/\overline{\mathbb{Q}})$, $P$ has finite order.  
\end{lemma}

\begin{proof}
The hypotheses on $P$ imply that $\psi \neq \mathrm{id}$. Since $mP\in E(\mathbb{Q})$, we have $O = mP-\tau(mP) = m(P-\tau(P))$. If $\psi = -\mathrm{id}$, then 
\begin{equation*}
\begin{cases}
\tau(P) = -P\\
m(P-\tau(P)) = O
\end{cases} \iff\quad \begin{cases}
\tau(P) = -P\\
2mP = O;
\end{cases}
\end{equation*}
thus, $P$ has order dividing $2m$.

For more general $\psi$, let 
\[
[\,\cdot\,]\colon R\simeq \mathrm{End}(E)
\]
where $R\subset\mathbb{C}$. Then there exists a root of unity $\zeta$ such that $\psi(P) = [\zeta]P$. Therefore,

\begin{equation*}
\begin{cases}
\tau(P) = \left[\zeta\right]P\\
m(P-\tau(P)) = O
\end{cases} \iff\quad \begin{cases}
\tau(P) = \left[\zeta\right]P\\
\left[m(1-\zeta)\right]P = O
\end{cases}
\end{equation*}
and so $P\in E(K)[m N_{\mathbb{Q}(\zeta)/\mathbb{Q}}(1-\zeta)]$.
\end{proof}
We could try and use non-cyclotomic idele class characters to motivate the existence of some algebraic points of infinite order in $\mathcal{X}(\mathbb{Z}_p)_2^{\prime}$, following the ideas of \S \ref{sec:other_characters}. For example, what we could hope to prove is that at a certain $z\in\mathcal{X}(\overline{\mathbb{Z}})$ satisfying $mz\in E(\mathbb{Q})$ for some non-zero integer $m$, the quantity $2D_2(z)+C(\Log(z))^2 +||w||$, for some $w\in W$, equals the value of \emph{some} $p$-adic height function at $z$. If such a $p$-adic height comes from a character which restricts to the cyclotomic character on $\mathbb{A}_{\mathbb{Q}}^{\times}$ with the right normalisations, then looking at the equation defining $\psi(w)$ we see that this is enough to show $z\in\mathcal{X}(\mathbb{Z}_p)_2^{\prime}$.

However, our computations (more on this in \S \ref{sec:rank1computations}) also recovered some algebraic non-torsion points defined over real quadratic fields and we know that that the space of idele class characters of a real quadratic field is one-dimensional. Therefore, in the following proposition we present a sufficient condition for a point defined over a quadratic field to belong to $\mathcal{X}(\mathbb{Z}_p)_2^{\prime}$, which looks less geometric or algebraic in nature compared to the results of Section \ref{sec:obstructions}. However, we then discuss in Remark \ref{rmk:when_does_this_hold} when we expect the hypotheses of the proposition to be satisfied.
\begin{prop}
\label{prop:extrapointsrank1}
Suppose that $E$ satisfies the assumptions of Theorem \ref{level2rank1} and that $p$ is an odd prime of good ordinary reduction. Let $K$ be a quadratic field in which $p$ splits. Fix an embedding $\rho\colon K\hookrightarrow\mathbb{Q}_p$ and let $\tau$ be the non-trivial element in $\mathrm{Gal}(K/\mathbb{Q})$. Suppose that $z\in\mathcal{X}(\OO_K)$ is such that $mz\in E(\mathbb{Q})\setminus \{O\}$ for some non-zero integer $m$ and that
\begin{equation}
\label{eq:div_poly}
f_m(z) = \zeta f_m(\tau(z)), 
\end{equation}
for some root of unity $\zeta$. For each rational prime $q$, let $\mathfrak{q}$ be one (any) prime of $K$ above $q$ and $\lambda_{\mathfrak{q}}$ the local height at $\mathfrak{q}$ with respect to the model $\mathcal{E}$. If 
\begin{equation*}
\sum_{q\in S}\lambda_{\mathfrak{q}}(z) = ||w||
\end{equation*}
for some $w\in W$, then $\rho(z)\in\mathcal{X}(\mathbb{Z}_p)_2^{\prime}$.
\end{prop}

\begin{proof}
Let $\mathfrak{q}\nmid p$. By quasi-quadraticity (\S\ref{sec:heightsawayfromp}\thinspace\ref{property:quasiquadratic}) applied twice and the assumptions on $z$ and $m$, we have
\begin{align*}
\lambda_{\mathfrak{q}}(z)&= \frac{1}{m^2}\left(\lambda_{\mathfrak{q}}(mz) + 2\log|f_m(z)|_{\mathfrak{q}}\right)\\
& = \frac{1}{m^2}\left(\lambda_{\mathfrak{q}}(\tau(mz)) + 2\log|\zeta f_m(\tau(z))|_{\mathfrak{q}}\right)\\
& = \frac{1}{m^2}\left(\lambda_{\mathfrak{q}}(m\tau(z)) + 2\log|f_m(\tau(z))|_{\mathfrak{q}}\right)\\
& = \lambda_{\mathfrak{q}}(\tau(z)) = \lambda_{\tau(\mathfrak{q})}(z).
\end{align*}
Similarly, if $p\OO_K =\mathfrak{p}_1\mathfrak{p}_2$, then (without loss of generality) 
\begin{align*}
m^2\lambda_{\mathfrak{p}_1}(z) &= \lambda_{p}(mz)+ 2 \log(\rho(f_m(z)))\\
&= \lambda_p(mz)+ 2\log(\rho(f_m(\tau(z))))= m^2\lambda_{\mathfrak{p}_2}(z).
\end{align*}
Therefore,
\begin{equation*}
h_p(z) = \lambda_p(\rho(z))+ ||w||.
\end{equation*}
Since $z$ is in the saturation of $E(\mathbb{Q})$ (i.e.\ $mz\in E(\mathbb{Q})$), then $z\in \psi(w)$.
\end{proof}

\begin{rmk}
\label{rmk:when_does_this_hold}
Let $E$, $K$ and $p$ be an in Proposition \ref{prop:extrapointsrank1}. Suppose that $z\in\mathcal{X}(\OO_K)$ is such that $mz\in E(\mathbb{Q})\setminus \{O\}$ and write $x(mz)=\frac{n(mz)}{d(mz)^2}$, for some coprime integers $n(mz)$ and $d(mz)>0$. When can we expect (\ref{eq:div_poly}) to hold? By \cite[\S 2]{wuthrichheights} and our assumptions, we know that 
\begin{equation*}
\frac{g_m(z)}{f_m(z)^2}=x(mz)=x(m\tau(z))=\frac{g_m(\tau(z))}{f_m(\tau(z))^2},
\end{equation*}
where $g_m(z)$ can be written as a univariate polynomial in $x(z)$ over $\Z$. 
For $w\in\{z,\tau(z)\}$, define
\begin{equation*}
\delta_m(w) = \frac{f_m(w)}{d(mw)}.
\end{equation*}
By Proposition 1 of \emph{loc.\ cit.}\ $\delta_m(w)$ is a unit at all primes at which $w$ has non-singular reduction. Furthermore, we have
\begin{equation*}
 \frac{\delta_m(z)}{\tau(\delta_m(\tau(z)))}=\frac{f_m(z)}{\tau(f_m(\tau(z)))} =1.
\end{equation*}
Thus, if for example $z$ has good reduction at all primes which are split in $K$, then $\tau(\delta_m(\tau(z)))=\delta_m(\tau(z))$ up to multiplication by elements of $\OO_K^{\times}$: thus, in this case, $f_m(z)=uf_m(\tau(z))$ for some $u\in \OO_K^{\times}$. If $K$ is imaginary then $u$ is a root of unity; otherwise $u$ may or may not be. Note that if $K$ is imaginary, we could have also avoided talking about division polynomials and followed a strategy similar to Proposition \ref{prop:anticyclotomiccyclotomic}. Conversely, Proposition \ref{prop:extrapointsrank1} is often not applicable for torsion points as it requires the existence of a \emph{non-zero} multiple of $z$ in $E(\mathbb{Q})$.
\end{rmk}

\subsection{Computations in rank $1$}
\label{sec:rank1computations}
The technique explained in \S \ref{sec:algorithmelliptic} to compute $\mathcal{X}(\mathbb{Z}_p)_2$ in the rank $0$ case can easily be adapted to compute $\mathcal{X}(\mathbb{Z}_p)_2^{\prime}$ when the Mordell--Weil group has rank $1$. As remarked in \S \ref{sec:algebraicpointsrank1}, the expectation is that $\mathcal{X}(\mathbb{Z}_p)_2^{\prime}$ should generally be larger than $\mathcal{X}(\mathbb{Z})$.

We ran the code on all the $14,783$ rank $1$ elliptic curves of conductor at most $5,000$ and let $p$ be the smallest prime greater than or equal to $5$ at which the curve has good ordinary reduction.
The first observation is that it can happen that there are no points in $\mathcal{X}(\mathbb{Z}_p)_2^{\prime}$ beside those in $\mathcal{X}(\mathbb{Z})$. For example, the curves of conductor at most $500$
\begin{itemize}[label = -]
\item satisfying the assumptions of Theorem \ref{level2rank1}\thinspace(\ref{trivialrank1}) are \href{http://www.lmfdb.org/EllipticCurve/Q/254/b/1}{254.b1}, \href{http://www.lmfdb.org/EllipticCurve/Q/430/c/1}{430.c1} \cite{lmfdb};
\item not satisfying Theorem \ref{level2rank1}\thinspace(\ref{trivialrank1}), but for which $\mathcal{X}(\mathbb{Z}_p)_2^{\prime}=\mathcal{X}(\mathbb{Z})$ are:
\begin{center}
\href{http://www.lmfdb.org/EllipticCurve/Q/297/b/1}{297.b1}, \href{http://www.lmfdb.org/EllipticCurve/Q/325/b/1}{325.b1}, \href{http://www.lmfdb.org/EllipticCurve/Q/325/b/2}{325.b2}, \href{http://www.lmfdb.org/EllipticCurve/Q/467/a/1}{467.a1} \cite{lmfdb}.
\end{center}
\end{itemize}

Studying the torsion points of $\mathcal{X}(\mathbb{Z}_p)_2^{\prime}$ morally provides more data on the extra points that can arise in $\mathcal{X}(\mathbb{Z}_p)_2$ when $E$ has rank $0$. No new phenomenon was observed, except that torsion points defined over some degree $4$ number fields were also recovered on the two CM elliptic curves \href{http://www.lmfdb.org/EllipticCurve/Q/576/e/1}{576.e1} and \href{http://www.lmfdb.org/EllipticCurve/Q/576/e/2}{576.e2} \cite{lmfdb} of $j$-invariant $54000$. 
The appearance of the latter points can be proved in a similar way to Proposition \ref{prop:twists_X049}.

As far as algebraic non-torsion points are concerned, on $26$ curves we identified non-torsion points defined over quadratic extension of $\mathbb{Q}$. All of these were explained by Proposition \ref{prop:extrapointsrank1} and only on two curves the points were defined over real quadratic fields. We now present an example in which some algebraic torsion and non-torsion points were recovered in $\mathcal{X}(\mathbb{Z}_p)_2^{\prime}\setminus \mathcal{X}(\mathbb{Z})$. Afterwards, we also include for completeness an example in which the extra algebraic non-torsion points are real.
\begin{example}
Consider the elliptic curve \href{http://www.lmfdb.org/EllipticCurve/Q/576/e/4}{576.e4} \cite{lmfdb}
\begin{equation*}
E\colon y^2 = x^3 + 8,
\end{equation*}
whose Mordell--Weil group over $\mathbb{Q}$ has rank $1$ and is generated, modulo torsion, by the point $z_0= (1,3)$. Let $p=7$, at which $E$ has good ordinary reduction. We have $S = \{2,3\}$ and $W = W_2\times W_3$ where
\begin{align*}
W_2 &= \{0,-\log 2\}\\
W_3 &= \left\{0,-\frac{1}{2}\log 3\right\}.
\end{align*}
Write
\begin{equation*}
\mathcal{X}(\mathbb{Z}_p)_2^{\prime} = \mathcal{X}(\mathbb{Z}_p)_{2,\mathrm{tors}}^{\prime}\cup \mathcal{X}(\mathbb{Z}_p)_{2,\mathrm{nontors}}^{\prime},
\end{equation*}
where the subscripts $\mathrm{tors}$ and $\mathrm{nontors}$ have the obvious meaning. Let $K=\mathbb{Q}(\sqrt{-3})$ and let $\tau$ generate the Galois group of $K/\mathbb{Q}$. Assuming that we have fixed an embedding of $\mathbb{Q}(\sqrt{-3})$ into $\mathbb{Q}_p$, we find that 
\begin{align*}
& \mathcal{X}(\mathbb{Z}_p)_{2,\mathrm{tors}}^{\prime} = \{(-2,0), (1\pm \sqrt{-3},0)\}\\
& \mathcal{X}(\mathbb{Z}_p)_{2,\mathrm{nontors}}^{\prime} = \pm\{(1,3), (2,-4), (46,-312), (-5\pm\sqrt{-3},6\pm 6\sqrt{-3})\}\cup A^{\mathrm{nonalg?}}
\end{align*}
where $A^{\mathrm{nonalg?}}$ denotes the set of points of $\mathcal{X}(\mathbb{Z}_p)_2^{\prime}$ which have not been recognised as algebraic. Note that $A^{\mathrm{nonalg?}}$ modulo $\pm$ consists of $15$ points. Corollary \ref{cor:j0automorphism}, together with the observation at the beginning of this section, proves why the two-torsion quadratic points $(1\pm \sqrt{-3},0)$ belong to $\mathcal{X}(\mathbb{Z}_p)_2^{\prime}$.

Consider now $$Q\in \{\pm(-5\pm\sqrt{-3},6\pm 6\sqrt{-3})\}.$$ We will show why $Q\in\mathcal{X}(\mathbb{Z}_p)_{2,\mathrm{nontors}}^{\prime}$. Without loss of generality we may assume that $Q = (-5+\sqrt{-3},6+ 6\sqrt{-3})$.
As
\begin{equation*}
Q = -z_0 + (1-\sqrt{-3},0),
\end{equation*}
$2Q\in E(\mathbb{Q})$. We have
\begin{equation*}
f_2(Q) =2y(Q) = \left(\frac{-1+\sqrt{-3}}{2}\right)f_2(\tau(Q)).
\end{equation*}
Therefore, in order to apply Proposition \ref{prop:extrapointsrank1}, it suffices to verify the condition on the local heights at the bad primes. For each prime $\mathfrak{q}\nmid p$, we could use the formula involving $f_2(Q)$ in order to compute $\lambda_{\mathfrak{q}}(Q)$, as in the proof of the proposition. We choose to compute it instead by the quasi-parallegram law
\begin{align*}
\lambda_{\mathfrak{q}}(Q) &= \lambda_{\mathfrak{q}}(z_0)+\lambda_{\mathfrak{q}}(1-\sqrt{-3},0)-\log|\sqrt{-3}|_{\mathfrak{q}}\\
&= \lambda_{\mathfrak{q}}(z_0)+\lambda_{\mathfrak{q}}(-2,0)-\log|\sqrt{-3}|_{\mathfrak{q}},
\end{align*}
which gives
\begin{equation*}
\lambda_{\mathfrak{q}}(Q)=\begin{cases}
-\frac{1}{2}\log 3 &\text{if}\ \mathfrak{q}\mid 3\\
-\log 2 & \text{if}\ \mathfrak{q}\mid 2\\
0 & \text{if}\ \mathfrak{q}\nmid 2,3,p.
\end{cases}
\end{equation*}
The fact that $Q$ is in $\mathcal{X}(\mathbb{Z}_p)_{2}^{\prime}$ then follows from Proposition \ref{prop:extrapointsrank1}.
\end{example}

\begin{example}
Consider the elliptic curve \href{http://www.lmfdb.org/EllipticCurve/Q/525/c/1}{525.c1}\cite{lmfdb}
\begin{equation*}
E\colon y^2 + x y = x^{3} + x^{2} - 450 x + 3375.
\end{equation*}
Let $p$ be an odd prime of good ordinary reduction split in $\mathbb{Q}(\sqrt{5})$ and fix an embedding $\mathbb{Q}(\sqrt{5})\hookrightarrow \mathbb{Q}_p$. Then by Proposition \ref{prop:extrapointsrank1} with $m=2$, the infinite order points $$\pm \left(10\pm 5\sqrt{5},\frac{5}{2}(23\mp \sqrt{5})\right)$$ belong to $\mathcal{X}(\mathbb{Z}_p)_2^{\prime}$.
\end{example}

\section{Rational points on bielliptic curves}
\label{sec:bielliptic}
Let $C$ be a smooth projective curve over $\mathbb{Q}$ of genus $g$ and whose Jacobian $J$ has Mordell--Weil rank equal to $g$. Assume in addition that the N\'eron--Severi group of $J$ has rank at least equal to $2$, that $p$ is an odd prime of good reduction for $C$ and that the $p$-adic closure of $J(\mathbb{Q})$ has finite index in $J(\mathbb{Q}_p)$. In \cite[Theorem 1.2]{BDQCI} Balakrishnan and Dogra use the Chabauty--Kim method to explicitly describe a finite set
\begin{equation*}
C(\mathbb{Q}_p)_Z\subset C(\mathbb{Q}_p),
\end{equation*}
which depends upon the choice of a correspondence $Z\subset C\times C$ and which contains all the $\mathbb{Q}$-rational points of $C$. Note that one has $C(\mathbb{Q}_p)_2\subset C(\mathbb{Q}_p)_Z$. The authors then make Theorem 1.2 algorithmic when $C$ is a bielliptic curve of genus $2$, under the extra assumption that $J$ is ordinary at $p$ (we remove this hypothesis here). Let
\begin{equation*}
C\colon y^2=x^6+a_4x^4+a_2x^2+a_0,\qquad a_i\in\mathbb{Q},
\end{equation*}
be a genus $2$ bielliptic curve with a rank-$2$ Jacobian and consider the associated maps $\varphi_i\colon C\to (E_i,O_{E_i})$, described affinely by
\begin{align*}
E_1&\colon y^2=x^3+a_4x^2+a_2x+a_0, \qquad &&\varphi_1(x,y) = (x^2,y)\\
E_2&\colon y^2=x^3+a_2x^2+a_4a_0x+a_0^2, &&\varphi_2(x,y) = (a_0x^{-2},a_0yx^{-3}).
\end{align*}
Assume that each of $E_1$ and $E_2$ has rank $1$. Then one may pick $Z$ in such a way that $C(\mathbb{Q}_p)_Z$ can be described in terms of local and global $p$-adic heights of the images of the points of $C$ in the two elliptic curves (we assume that the given equations for $E_1$ and $E_2$ are minimal at $p$). The superscript $^{E_i}$ indicates on which curve we are computing these quantities.
Let $P_i$ be a point of infinite order in $E_i(\mathbb{Q})$ and write
\begin{equation*}
c_i = \frac{h_p^{E_i}(P_i)}{\Log^{E_i} (P_i)^2},
\end{equation*}
where, as usual, $h_p^{E_i}$ is the global $p$-adic height of Mazur--Tate or Bernardi depending on whether the reduction is ordinary or not.
 Let $Q_1=(0,\sqrt{a_0})\in E_1(\mathbb{Q}(\sqrt{a_0}))$ and $Q_2=(0,a_0)\in E_2(\mathbb{Q})$. We assume that $a_0$ is a square in $\mathbb{Q}_p$. Furthermore, let
 \begin{align*}
 C^{(1)}(\mathbb{Q}_p)&=C(\mathbb{Q}_p)\setminus (](0,\sqrt{a_0})[\ \cup\ ](0,-\sqrt{a_0})[)\\
  C^{(2)}(\mathbb{Q}_p)&=C(\mathbb{Q}_p)\setminus (]\infty^{+}[\ \cup\ ]\infty^{-}[)\\
  C^{(i)}(\mathbb{Q}) &=C(\mathbb{Q})\cap C^{(i)}(\mathbb{Q}_p)\quad \text{for } i=1,2,
 \end{align*}
 where the inverted square brackets denote the residue disk modulo $p$ around the given point and $\infty^{\pm} = (1:\pm 1:0)\in C(\mathbb{Q})$.
\begin{thm}[Balakrishnan--Dogra\protect\footnotemark]\footnotetext{With a small correction: see Remark \ref{rmk:Q1notinsaturation}.}
\label{quadraticchabautybielliptic}
For each $i\in\{1,2\}$, the following set is finite
\begin{align*}
W^i = \biggl\{\sum_{q\neq p}\left(\lambda^{E_i}_q(\varphi_i(z_q)+Q_i)+\lambda^{E_i}_q(\varphi_i(z_q)-Q_i)-2\lambda^{E_{3-i}}_{q}(\varphi_{3-i}(z_q))\right):&\\
(z_q)\in \prod_{q\neq p}C(\mathbb{Q}_q)\setminus \{\varphi_i^{-1}(\pm Q_i)\}&\biggr\}.
\end{align*}
Furthermore,
\begin{align*}
C^{(i)}(\mathbb{Q})\subset\bigl\{z\in C^{(i)}(\mathbb{Q}_p):2\lambda^{E_{3-i}}_{p}(\varphi_{3-i}(z))-\lambda^{E_i}_p(\varphi_i(z)+Q_i)-\lambda^{E_i}_p(\varphi_i(z)-Q_i)&\\
-2c_{3-i}\Log^{E_{3-i}}(\varphi_{3-i}(z))^2+2c_i\Log^{E_i}(\varphi_i(z))^2+2h_p^{E_i}(Q_i)\in W^i&\bigr\}.
\end{align*}
\end{thm}
The second assertion in Theorem \ref{quadraticchabautybielliptic} can be derived from the parallelogram law satisfied by the global $p$-adic height (as a consequence of \S\ref{sec:heightsawayfromp}\thinspace \ref{property:quasipar} and \S\ref{sec:heightsabovep}\thinspace\ref{prop:quasi_parallelogram_law_above_p}) and the fact that there is at most one quadratic function (up to multiplication by a scalar) on each of $E_1(\mathbb{Q})$ and $E_2(\mathbb{Q})$, due to the assumption on their ranks, in analogy to the proof of Theorem \ref{level2rank1}. 
The set of $p$-adic points described in the theorem is $C(\mathbb{Q}_p)_Z\cap C^{(i)}(\mathbb{Q}_p)$. We will explicitly determine the sets $W^i$ for Example \ref{Example:wetherell}. For a general bielliptic curve we will give in Proposition \ref{prop:valuesforbielliptic} a description of a finite set containing $W^i$, hence giving a proof of finiteness of $W^i$.
\begin{rmk}
\label{rmk:Q1notinsaturation}
If $Q_1$ is not defined over $\mathbb{Q}$, with the notation $\lambda_q^{E_1}(\varphi_1(z)\pm Q_1)$ in Theorem \ref{quadraticchabautybielliptic} we mean $\lambda_{\mathfrak{q}}^{E_1}(\varphi_1(z)\pm Q_1)$, where $\mathfrak{q}$ is any prime of $\mathbb{Q}(\sqrt{a_0})$ lying above the rational prime $q$. Indeed, since $\tau(Q_1)=-Q_1$, where $\braket{\tau}=\mathrm{Gal}(\mathbb{Q}(\sqrt{a_0})/\mathbb{Q})$, there is no dependence on $\mathfrak{q}\mid q$ and 
\begin{equation*}
h_p^{E_1}(Q_1)=\sum_{q}\lambda_q^{E_1}(Q_1).
\end{equation*} 
The equations defining the sets $C(\mathbb{Q}_p)_Z\cap C^{(1)}(\mathbb{Q}_p)$ in \cite[Corollary 8.1]{BDQCI} contain a typo in the case when $Q_1$ is not in the saturation of $E_1(\mathbb{Q})$, which we have corrected in Theorem \ref{quadraticchabautybielliptic}. Their formula has the term $2c_1\Log^{E_1}(Q_1)^2$ in place of $2h_p^{E_1}(Q_1)$ and does not hold unless the two quantities are equal.
\end{rmk}

One of the advantages of computing rational points using Theorem \ref{quadraticchabautybielliptic} for a bielliptic curve, rather than the more general techniques developed in \cite{SplitCartan}, is that we do not need to have prior knowledge of any affine point in $C(\mathbb{Q})$. 

Balakrishnan--Dogra--M\"uller \cite{BDQCI} use Theorem \ref{quadraticchabautybielliptic}, combined with the Mordell--Weil sieve, to determine precisely the rational points of two bielliptic curves. As in \S \ref{sec:algorithmelliptic}, we suggest here that one can replace the computations of double Coleman integrals with computations involving the $p$-adic sigma function and division polynomials. We use the resulting algorithm to compute $C(\mathbb{Q}_p)_Z$ for a bielliptic curve whose rational points were already found using different Chabauty-type techniques by Wetherell \cite[Proposition 5.1]{wetherell} and Flynn--Wetherell \cite[Example 3.1]{FlynnWetherell}. Our methods lead to an alternative provable determination of $C(\mathbb{Q})$.

\begin{example}
\label{Example:wetherell}
Consider the bielliptic curve
\begin{equation*}
C\colon y^2 = x^6+x^2+1;
\end{equation*}
the associated elliptic curves are \href{http://www.lmfdb.org/EllipticCurve/Q/496/a/1}{496.a1} and \href{http://www.lmfdb.org/EllipticCurve/Q/248/a/1}{248.a1} \cite{lmfdb}, given by the following minimal models
\begin{equation*}
E_1\colon y^2=x^3+x+1,\qquad E_2\colon y^2=x^3+x^2+1.
\end{equation*}
We have
\begin{equation*}
Q_1 = (0,1)\in E_1(\mathbb{Q}),\qquad Q_2=(0,1)\in E_2(\mathbb{Q}).
\end{equation*}
Both elliptic curves have rank $1$. Let $p=3$, which is a prime of good reduction for $C$. Since $E_1$ is supersingular at $p$, we use Bernardi's $p$-adic height for our calculations.
\begin{claim}
\label{claim:1}
If $q\neq p,2$ and $z\in C(\mathbb{Q}_q)\setminus\{\varphi_i^{-1}(\pm Q_i)\}$, then $$w_{q,i}(z)\colonequals \lambda^{E_i}_q(\varphi_i(z)+Q_i)+\lambda^{E_i}_q(\varphi_i(z)-Q_i)-2\lambda^{E_{3-i}}_{q}(\varphi_{3-i}(z))=0.$$
\end{claim}
\begin{proof}
The curves $E_1$ and $E_2$ have everywhere good reduction except at $2$ and $31$. At $31$, the Tamagawa number of each $E_i$ is trivial. 
Assume first that $\varphi_i(z)\neq O_{E_i}$. Then, for each $q\neq p$, the quasi-parallelogram law (\ref{eq:quasi_parallelogram_law}) gives
\begin{align*}
w_{q,i}(z) =  2\left(\lambda^{E_i}_q(\varphi_i(z))+\lambda^{E_i}_q(Q_i)-\log|x(\varphi_i(z))|_q-\lambda^{E_{3-i}}_{q}(\varphi_{3-i}(z))\right);
\end{align*}
thus, if $q\neq 2$, by Lemma \ref{lemmalocalheights}\thinspace\ref{goodreduction}, 
\begin{align*}
w_{q,i}(z)&=2(\log(\max\{1,|x(\varphi_i(z))|_q\})-\log(\max\{1,|x(\varphi_{3-i}(z))|_q\})-\log|x(\varphi_i(z))|_q)\\
&=2(\log(\max\{1,|x(\varphi_i(z))|_q\})-\log(\max\{1,|x(\varphi_i(z))|_q^{-1}\})-\log|x(\varphi_i(z))|_q)\\
&=0.
\end{align*}
It remains to consider the case $\varphi_i(z) =  O_{E_i}$. Then $\varphi_{3-i}(z) \in\{\pm Q_{3-i}\}$ and
\begin{equation*}
w_{q,i}(z) = 2\lambda_q^{E_i}(Q_i)-2\lambda_q^{E_{3-i}}(Q_{3-i}) = 0
\end{equation*}
by Proposition \ref{propositionheightstamagawa1}, since $Q_i$ and $Q_{3-i}$ are integral and the Tamagawa numbers at $q$ are equal to $1$.
\end{proof}
\begin{claim}
\label{claim:2}
We have
\begin{equation*}
W^1 = \{0,\log 2\}, \qquad  W^2 = \{-\log 2, -2\log 2 \}.
\end{equation*}
\end{claim}
\begin{proof}
By Claim \ref{claim:1}, 
\begin{align*}
W^i = \{w_{2,i}(z)\colonequals\lambda^{E_i}_2(\varphi_i(z)+Q_i)+\lambda^{E_i}_2(\varphi_i(z)-Q_i)-2\lambda^{E_{3-i}}_{2}(\varphi_{3-i}(z)) :&\\
 z\in C(\mathbb{Q}_2)\setminus\{\varphi_i^{-1}(\pm Q_i)\}&\}.
\end{align*}
First note that the curve $E_1$ has Tamagawa number equal to $1$ at $2$, whereas $E_2$ has Tamagawa number equal to $2$ and reduction type $\mathrm{III}$. 
If $|x(z)|_2\leq 1$, then $\varphi_2(z)$ has good reduction at $2$ and $\lambda_2^{E_2}(\varphi_2(z)) = \log |x(z)^{-2}|_2$; otherwise $\varphi_2(z)$ reduces to a singular point modulo $2$ and, by Proposition \ref{localheightsnontrivialtamagawa}, $\lambda_2^{E_2}(\varphi_2(z))= -\frac{1}{2}\log 2$.  Furthermore $\lambda_2^{E_2}(Q_2) = -\frac{1}{2}\log 2$. Therefore, similarly to the proof of Claim \ref{claim:1}, if $\varphi_1(z)\neq O_{E_1}$, then by the quasi-parallelogram law we have
\begin{align*}
w_{2,1}(z) =  2\left(\lambda^{E_1}_2(\varphi_1(z))-\lambda^{E_{2}}_{2}(\varphi_2(z))-\log|x(\varphi_1(z))|_2\right) = \begin{cases}
0 &\text{if}\ |x(z)|_2\leq 1,\\
\log 2 &\text{if}\ |x(z)|_2>1;
\end{cases}
\end{align*}
for the remaining points $z=\infty^{\pm}$ we get
\begin{equation*}
w_{2,1}(z)=-2\lambda_2^{E_2}(Q_2)=\log 2,
\end{equation*}
thus proving the claim for $W^1$. The set $W^2$ is determined in a very similar fashion and we leave the details to the reader.
\end{proof}
We can now compute $C(\mathbb{Q}_p)_Z$ as a union of the two $C(\mathbb{Q}_p)_Z\cap C(\mathbb{Q}_p)^{(i)}$.
We find
\begin{equation}
C(\mathbb{Q}_p)_Z= \{\infty^{\pm},(0,\pm 1 ), (\pm 1/2,\pm 9/8)\}\sqcup A
\end{equation}
where $A$ is a set of size $4$. Note, however, that up to the automorphisms $(x,y)\mapsto (-x,y)$, $(x,y)\mapsto (x,-y)$ and their composites, $A$ actually consists of one point:
\begin{equation*}
P = (2\cdot 3 + 2\cdot 3^3 + 2\cdot 3^5 + 3^8 + O(3^9) , 2 + 2\cdot 3 + 2\cdot 3^3 + 2\cdot 3^5 + 3^6 + 2\cdot 3^8 + O(3^9)).
\end{equation*}
We follow the same strategy to the one of the proof of \cite[Theorem 8.6]{BDQCI} to rule out the possibility that $P$ could be rational. The image of the point $P$ under $\varphi_2$ is a point in $E_2(\mathbb{Q}_p)$ whose $x$-coordinate has valuation $\ord(x(\varphi_2(P)))=-2$. On the other hand, the Mordell--Weil group $E_2(\mathbb{Q})\cong \mathbb{Z}$ is generated by $Q_2$ and, thus, if $\varphi_2(P)\in E_2(\mathbb{Q})$, there must exist a multiple of $Q_2$ whose $x$-coordinate has $p$-adic valuation equal to $-2$. As the smallest multiple of $Q_2$ in the formal group at $p$ is $6Q_2 = \left(\frac{55}{81},-\frac{971}{729}\right)$, we have reached a contradiction, since the set of points in the formal group whose $x$-coordinate has valuation at most $-4$ is a group. 
\end{example}

\subsection{Explicit formulae for the sets $W^i$}
\label{sec:explicitWi}
Theorem \ref{quadraticchabautybielliptic} asserts that the sets $W^i$, for $i=1,2$, are finite, but does not describe them explicitly. In order to obtain an implementation for the computations of the sets $C(\mathbb{Q}_p)_Z\cap C^{(i)}(\mathbb{Q}_p)$ for an arbitrary genus $2$ bielliptic curve $C$, it would be convenient to have a characterisation of $W^i$ that can be made algorithmic, in analogy with that of the sets $W_q$ of Theorems \ref{level2rank0} and \ref{level2rank1}.

We assume in this section that the coefficients $a_0,a_2,a_4$ defining $C$ are in $\mathbb{Z}$.

For $i=1,2$, let $W_q^{E_{i,\min}}$ be the set $W_q$ from Section \ref{intro} for a global minimal model $E_{i,\min}$ of $E_i$ if $q$ is a prime of bad reduction and with non-trivial Tamagawa number. If $E_{i,\min}$ has good reduction at $q\in \{2,3\}\setminus\{p\}$ and $\overline{E}_{i,\min}(\mathbb{F}_q)=\{O\}$ or $q=2$ and $E_{i,\min}$ has split multiplicative reduction $\mathrm{I}_1$ at $q$, let $W_q^{E_{i,\min}}=\emptyset$. For all other $q\neq p$, let $W_q^{E_{i,\min}}=\{0\}$.

Let $W_q^{E_{i}}$ be the set of values attained by $\lambda_q^{E_i}$ on the points of $E_i(\mathbb{Q}_q)$ of the form $(x,y)$ with $x,y\in\mathbb{Z}_q$ and let
\begin{equation*}
V_q^{E_{i}} =W_q^{E_{i}}\cup \{0\}.
\end{equation*}
Write
\begin{equation*}
\delta^{E_i} = \frac{\Delta^{E_i}}{\Delta^{E_{i,\min}}},
\end{equation*}
where $\Delta^{E_i}$ and $\Delta^{E_{i,\min}}$ are the discriminants of $E_i$ and $E_{i,\min}$.

For sets $A,B$ of elements in a field $F$, we write $A+B$ for their Minkowski addition and $-A$ for the set consisting of the additive inverses of the elements in $A$. If $A=\{a\}$, we write $a+B$ for $A+B$.
\begin{lemma}
\label{lemma:heightsonintegralnonminimal}
\begin{equation*}
W_q^{E_{i}}= \frac{1}{6}\log|\delta^{E_i}|_q + \left(W_q^{E_{i,\min}}\cup \left\{2k\log q: 1\leq k\leq \frac{1}{12}\ordnop_q(\delta^{E_i})\right\}\right).
\end{equation*}
In particular, $W_q^{E_i}$ is finite; it equals $\{0\}$ for all but finitely many $q$.
\end{lemma}
\begin{proof}
Let $x,y$ and $x_{\min},y_{\min}$ be the Weierstrass coordinates for $E_i$ and $E_{i,\min}$, respectively. Then there exist $u,r,s,t\in \mathbb{Q}$, $u\neq 0$, such that
\begin{equation*}
x = u^2x_{\min}+r,\qquad y = u^3y_{\min}+su^2x_{\min}+t.
\end{equation*}
Since $\ordnop_q(\Delta^{E_{i,\min}})\leq \ordnop_q(\Delta^{E_{i}})$ for each prime $q$, the scalars $u,r,s,t$ are furthermore all integral (see \cite[Lemma 5.3.1]{Connell}).

 Now let $P\in E_i(\mathbb{Q}_q)$ such that $x(P),y(P)\in \mathbb{Z}_q$. By (\ref{lambdaqlambdaqmin}), we have
\begin{equation*}
\lambda_q^{E_i}(P) = \lambda_q^{E_{i,\min}}(P)+\frac{1}{6}\log|\delta^{E_i}|_q.
\end{equation*}
If $x_{\min}(P)\in \mathbb{Z}_q$, then $\lambda_q^{E_{i,\min}}(P)\in W_q^{E_{i,\min}}$ by Propositions \ref{localheightsnontrivialtamagawa}, \ref{propositionheightstamagawa1} and Lemma \ref{conditionsonq}. Otherwise, by Lemma \ref{lemmalocalheights}\thinspace \ref{goodreduction}, we have $\lambda_q^{E_{i,\min}}(P)=\log\big|\frac{x(P)-r}{u^2}\big|_q$, where by assumption $|x(P)-r|_q\leq 1$ and the valuation of $\frac{x(P)-r}{u^2}$ is an even negative integer. Since $\delta^{E_i}= u^{12}$, this completes the proof of $\subseteq$. To see why the inclusion is actually an equality, notice that the preimage of an integral point in $E_{i,\min}(\mathbb{Q}_q)$ is certainly an integral point on $E_i(\mathbb{Q}_q)$. Furthermore, the points on $E_{i,\min}(\mathbb{Q}_q)$ in the formal group are parametrised by $t\in q\mathbb{Z}_q$ as follows: $t\mapsto (x_{\min}(t),y_{\min}(t))$, where
\begin{equation*}
x_{\min}(t)=\frac{1}{t^2}-\frac{a_{1,\min}}{t}-a_{2,\min}-a_{3,\min}t+\cdots\in \mathbb{Z}[a_{1,\min},\dots,a_{6,\min}]((t)),
\end{equation*}
where $a_{1,\min},\dots, a_{6,\min}$ are the Weierstrass coefficients of $E_{i,\min}$; in particular, we have $\ordnop_q(x_{\min}(t))=-2\ordnop_q(t)$. Thus for each $1\leq k\leq \frac{1}{12}\ordnop_q(\delta^{E_i})$, setting $t=q^k$ gives a point on $E_{i,\min}(\mathbb{Q}_q)$ whose preimage in $E_i(\mathbb{Q}_q)$ has integral $x$-coordinate.
The second assertion of the lemma follows from the explicit description of the sets $W_q^{E_{i,\min}}$.
\end{proof}
\begin{prop}
\label{prop:valuesforbielliptic}
With the notation of Theorem \ref{quadraticchabautybielliptic}, suppose that $\ordnop_{\ell}(a_0)\in\{0, 1\}$ for each prime $\ell$. For each prime $q\neq p$, let 
\begin{align*}
W_q^{1\prime} &= \left\{2v+2\lambda_q^{E_1}(Q_1): v\in V_q^{E_1}+(-W_q^{E_{2}})\right\}\cup \left\{2v+2\lambda_q^{E_1}(Q_1)-2\log|a_0|_q: v\in W_q^{E_1}\right\}\\
W_q^{2\prime} &= \left\{2v+2\lambda_q^{E_2}(Q_2)-2\log|a_0|_q: v\in W_q^{E_2}+(-V_q^{E_1})\right\}\cup \left\{2v+2\lambda_q^{E_2}(Q_2): v\in -W_q^{E_1}\right\}.
\end{align*}
Then $W^{i}$ is a subset of the finite set
\begin{equation*}
W^{i\prime}=\biggl\{\sum_{q\neq p} w_{q,i}^{\prime}: w_{q,i}^{\prime}\in W_q^{i\prime}\biggr\}=\biggl\{2h_p^{E_i}(Q_i)-2\lambda_p^{E_i}(Q_i)+\sum_{q\neq p} (w_{q,i}^{\prime}-2\lambda_q^{E_i}(Q_i)): w_{q,i}^{\prime}\in W_q^{i\prime}\biggr\}.
\end{equation*}
\end{prop}
\begin{proof}
By definition,
\begin{equation*}
W^i = \biggl\{\sum_{q\neq p} w_{q,i}: w_{q,i}\in W_q^{i}\biggr\},
\end{equation*}
where
\begin{align*}
W_q^{i} = \bigl\{w_{q,i}(z)\colonequals\lambda^{E_i}_q(\varphi_i(z)+Q_i)+\lambda^{E_i}_q(\varphi_i(z)-Q_i)-2\lambda^{E_{3-i}}_{q}(\varphi_{3-i}(z)):&\\
z\in C(\mathbb{Q}_q)\setminus\left\{\varphi_i^{-1}(\pm Q_i)\right\}&\bigr\}.
\end{align*}
Let $z\in C(\mathbb{Q}_q)\setminus\{\varphi_i^{-1}(\pm Q_i)\}$. If $\varphi_i(z) = O_{E_i}$, then
\begin{equation*}
w_{q,i}(z) = 2\lambda_q^{E_i}(Q_i)-2\lambda_q^{E_{3-i}}(Q_{3-i}),
\end{equation*}
which belongs to $2\lambda_q^{E_i}(Q_i)-2W_{q}^{E_{3-i}}$. When $i=2$, note that, while it is not always the case that $Q_1$ is defined over $\mathbb{Q}_q$ (and hence that $\lambda_q^{E_1}(Q_1)\in W_q^{E_1}$), here this follows from the assumption that its preimage under $\varphi_1$ is in $C(\mathbb{Q}_q)$.

Otherwise, by the quasi-parallelogram law (\ref{eq:quasi_parallelogram_law}), we have
\begin{align*}
w_{q,i}(z) &=  2\left(\lambda^{E_i}_q(\varphi_i(z))+\lambda^{E_i}_q(Q_i)-\log|x(\varphi_i(z))|_q-\lambda^{E_{3-i}}_{q}(\varphi_{3-i}(z))\right)\\
&=  2\left(\lambda^{E_i}_q(\varphi_i(z))+\lambda^{E_i}_q(Q_i)+\log|x(\varphi_{3-i}(z))|_q-\log|a_0|_q-\lambda^{E_{3-i}}_{q}(\varphi_{3-i}(z))\right).
\end{align*}
Note that the assumption that $0\leq \ordnop_q(a_0)\leq 1$ implies that, for each $i=1,2$,
\begin{align}
\label{eq:vals_implications}
|x(\varphi_i(z))|_q > 1 \Rightarrow |x(\varphi_{3-i}(z))|_q< 1\\
|x(\varphi_{3-i}(z))|_q < 1 \Rightarrow |x(\varphi_i(z))|_q \geq 1.
\end{align}
If $|x(\varphi_i(z))|_q>1$, then $\varphi_i(z)$ reduces to a non-singular point modulo $q$, with respect to the Weierstrass equation for $E_i$. Thus $\lambda^{E_i}_q(\varphi_i(z))=\log|x(\varphi_i(z))|_q$. Furthermore, by (\ref{eq:vals_implications}), $\varphi_{3-i}(z)$ is integral with respect to the Weierstrass equation defining $E_{3-i}$ and we have $\lambda^{E_{3-i}}_{q}(\varphi_{3-i}(z))\in W_q^{E_{3-i}}$. Therefore
\begin{equation*}
\frac{w_{q,i}(z)}{2}\in \lambda^{E_i}_q(Q_i)+(-W_q^{E_{3-i}}).
\end{equation*}
Similarly, if $|x(\varphi_{3-i}(z))|_q > 1$, then 
\begin{equation*}
\frac{w_{q,i}(z)}{2}\in (\lambda^{E_i}_q(Q_i)-\log|a_0|_q)+W_q^{E_{i}}.
\end{equation*} 
It remains to consider the case when $|x(z)|_q = 1$. Then
\begin{equation*}
\frac{w_{q,i}(z)}{2}\in W_q^{E_{i}}+(-W_q^{E_{3-i}})+\begin{cases}
\lambda^{E_i}_q(Q_i) &\text{if}\ i=1,\\
\lambda^{E_i}_q(Q_i)-\log|a_0|_q &\text{if}\ i=2.
\end{cases}
\end{equation*}
\end{proof}
Proposition \ref{prop:valuesforbielliptic} and Lemma \ref{lemma:heightsonintegralnonminimal} turn Theorem \ref{quadraticchabautybielliptic} into an algorithm for computing a finite set of $p$-adic points containing $C(\mathbb{Q})$, for an arbitrary genus $2$ bielliptic curve $C$ whose associated elliptic curves $E_1$ and $E_2$ have Mordell--Weil rank equal to $1$. Furthermore, to improve the estimates of the sets $W^i$ provided by Proposition \ref{prop:valuesforbielliptic} we may use the fact that the contributions at primes of potential good reduction for $C$ are trivial (cf.\ \cite[Theorem 1.2\thinspace (i)]{BDQCI}). It seems unlikely to the author that the elementary approach of Proposition \ref{prop:valuesforbielliptic} could show the latter for an arbitrary curve, since it is hard to imagine how the proof could be made sensitive to the difference between $C$ being of potential good reduction and its Jacobian only being of potential good reduction. Furthermore, even at primes not of potential good reduction, the sets $W_{q}^{i\prime}$ might be larger than $W_q^{i}$. Nevertheless, having fixed an explicit curve $C$, the steps of the proof of the proposition should guide the reader through computing $W^i$ precisely.

Here we are instead interested in an algorithm which does not require prior computations of $W^i$. So let $W^{i\prime\prime}$ be obtained from $W^{i\prime}$ of Proposition \ref{prop:valuesforbielliptic} by replacing $W_{q}^{i\prime}$ with $\{0\}$ whenever $q$ is a prime of potential good reduction.  We implemented in \texttt{SageMath} the results of this section and could test them for several bielliptic curves, including the ones of \cite{BDQCI} and the bielliptic curve
\begin{equation}
\label{eq:WetherellFlynn}
C\colon y^2 = (x^2+1)(x^2+3)(x^2+7),
\end{equation}
which appears in \cite[p.\thinspace 532]{FlynnWetherell} as the only curve amongst $50$ bielliptic curves for which the methods of Flynn and Wetherell to find rational points failed.

For instance, in Example \ref{Example:wetherell} and for the curve of \S 8.3 in \cite{BDQCI} we have $W^i=W^{i\prime}=W^{i\prime\prime}$ for each $i=1$ and $i=2$, but for the curve of \cite[\S 8.4]{BDQCI}, the sets $W^{i\prime\prime}$ have size $3$, whereas $W^i$ has size $1$.

We also remark that in Example \ref{Example:wetherell}, as well as the two examples of \cite{BDQCI}, the elliptic curve $E_1$ has trivial Tamagawa number at all primes and $E_2$ has trivial Tamagawa numbers everywhere except for at one prime where it has Tamagawa number equal to two or three. In other words, finding precise expressions for the sets $W^i$ by hand is straightforward. In Example \ref{Example:wetherell} as well as \cite[\S 8.3]{BDQCI} the task is further simplified by the fact that the Weierstrass equations for $E_1$ and $E_2$ have minimal discriminant. On the other hand, we cannot expect this for a generic curve. For example, the elliptic curves corresponding to (\ref{eq:WetherellFlynn}) have Tamagawa numbers $(2,1)$ respectively $(4,2)$ at the primes not of potential good reduction and analysing what happens at each prime by hand might be rather tedious. In this case, we find $\# W^{1\prime\prime}= \# W^{2\prime\prime}=18$ and in fact this results in many points in $C(\mathbb{Q}_p)$ that are probably not rational ($142$ when $p=5$). 

\bibliographystyle{alpha}
\bibliography{biblio}
\end{document}